\documentclass[a4paper,10pt]{article} 
\usepackage{amsmath,amssymb,amsfonts,amsthm,graphicx,latexsym,a4wide,color,framed,comment,url, subcaption, cite} 

\usepackage{srcltx}
\usepackage{multirow}
\usepackage{bbm}
\usepackage{framed, enumitem}
\usepackage[normalem]{ulem} 
\usepackage{tcolorbox,mdframed,xcolor}
\usepackage[export]{adjustbox}

\newcommand{\R}{\mathbb{R}}
\newcommand{\N}{\mathbb{N}}

\newcommand{\fd}{f^{\delta}}

\newtheorem{theorem}{Theorem}[section]

\newtheorem{lemma}[theorem]{Lemma}

\newtheorem{proposition}[theorem]{Proposition}
\newtheorem{remark}[theorem]{Remark}

\begin{document}
	
	\title{Multiscale hierarchical decomposition methods for images corrupted by multiplicative noise}
	
	\author{Joel Barnett\footnotemark[1]\,\,\, Wen Li\footnotemark[2]\,\,\, Elena Resmerita\footnotemark[3]\,\,\, and\,\,\, Luminita Vese\footnotemark[4]}
	
	\maketitle
	\renewcommand{\thefootnote}{\fnsymbol{footnote}}
	\footnotetext[1]{Department of Mathematics, University of California
		at Los Angeles (UCLA), Los Angeles, CA 90095 (jrbarnett@math.ucla.edu)}
	\footnotetext[2]{Department of Mathematics,
		Fordham University, Bronx, New York 10458 (wli198@fordham.edu)}
	\footnotetext[3]{Institute of Mathematics, 
		Alpen-Adria Universit\"at Klagenfurt, 
		Universit\"atsstrasse 65--67,
		9020 Klagenfurt, Austria 
		(elena.resmerita@aau.at)}
	\footnotetext[4]{Department of Mathematics, University of California
		at Los Angeles (UCLA), Los Angeles, CA 90095 (lvese@math.ucla.edu)}
	
	\renewcommand{\thefootnote}{\arabic{footnote}}
	\vspace*{-12pt}
	\begin{abstract}
		Recovering images corrupted by multiplicative noise is a well known challenging task. Motivated by the success of multiscale hierarchical decomposition methods (MHDM) in image processing, we adapt a variety of both classical and new multiplicative noise removing models to the MHDM form. On the basis of previous work, we further present a tight and a refined version of the corresponding multiplicative MHDM. We discuss  existence and uniqueness of solutions for the proposed models, and additionally, provide convergence properties. Moreover, we present a discrepancy principle stopping criterion which prevents recovering  excess noise in the multiscale reconstruction. Through comprehensive numerical experiments and comparisons, we qualitatively and quantitatively evaluate the validity of all proposed models for denoising and deblurring images degraded by multiplicative noise. By construction, these multiplicative multiscale hierarchical decomposition methods have the added benefit of recovering many scales of an image, which can provide features of interest beyond image denoising. 
		
	\end{abstract}

	\vspace{0.5cm}

	\section{Introduction}
	While the literature on denoising and deblurring images affected by additive noise is quite rich, the study of  images corrupted by multiplicative noise still requires attention. This type of noise is inherent  in radar, synthetic-aperture radar (SAR) and ultrasound images, cf. \cite{goodman1976, burckhardt1978}. Our aim is to contribute to the topic not only by removing such  noise from images, but also by proposing multiscale decomposition strategies for those images, similar to the contributions of \cite{tad_nez_ves1,tad_nez_ves2,mod_nac_ron,li_res_ves} in case of additive corruption. More precisely,  this work expands multiscale methods to the multiplicative-noise domain. While we apply such techniques to foundational methods  from \cite{rud_osh94, auber_aujol_2008,  osher_shi_multiplic}, the resulting procedures could translate to other approaches aimed at multiplicative  corruption. The importance of providing decompositions of  medical images that separate the coarse and fine  scales  has been clearly highlighted in the case of  image registration when significant levels of noise are involved \cite{multi_registration} (see also \cite{mod_nac_ron, deb_guy_ves}). The reader is referred also to \cite{bijaoui1997multiscale, rue1996multiscale} that address  astronomical imaging and the need of recovering objects of very different sizes.
	
	In a multiplicative noise problem, a clean image $z:\Omega\subseteq\mathbbm{R}^2 \to \mathbbm{R}$ is degraded by multiplicative noise  $\eta:\Omega \to \mathbbm{R}$ of mean 1 and (possibly) blurred by an ill-posed, linear, bounded operator $T:L^2(\Omega)\to L^2(\Omega)$, where $\Omega$ is a domain in $\R^2$. These together form the degraded data 
	\begin{equation}\label{model}
	f^\delta = (Tz)\cdot \eta,
	\end{equation}
	where $\delta$ is a parameter relating to the size of the noise, referred to  hereafter as  the noise level. Throughout this work, we denote by $f\in L^2(\Omega)$ the \textit{exact data} satisfying $Tz=f$.
	
	A natural way to approach multiplicative noise is to manipulate the problem into a familiar form and apply existing techniques. At least in the pure denoising case, the logarithm transforms the problem $f^\delta=z\cdot \eta$ to an additive noise system $\log(f^\delta) = \log(z) +\log(\eta)$ for which a plethora of denoising methods exist. Indeed, this idea has been tried---the reader can find in  \cite{auber_aujol_2008}  a discussion of a log-additive model. However, as pointed out in \cite{auber_aujol_2008}, blindly applying the $\log$ transform and employing algorithms for additive noise removal do not necessarily provide reasonable reconstructions, because the reconstruction means are often much smaller than those of the original images. This is due to the primary assumption of additive-noise methods, namely, zero-mean noise. To be more precise, by Jensen's inequality one has $\log(\mathbb{E}[\eta])\geq \mathbb{E}[\log(\eta)]$. If a restoration $u$ of the system $\log(\fd) = \log(u) + \log(\eta)$ is found under the assumption that $\mathbb{E}[\log(\eta)]=0$, then $\log(\mathbb{E}[\eta])\geq 0.$ Consequently, one has $\mathbbm{E}[\eta]\geq 1$ with strict inequality as soon as there is any noise, which is a contradiction to $\eta$ having mean 1. One can estimate this change in expectation by expanding $\log(\eta)$ about $\mathbb{E}[\eta]$,
	$$ \mathbb{E}[\log(\eta)] \approx \log(\mathbb{E}[\eta]) - \frac{\mathbb{V}[\eta]}{2\mathbb{E}[\eta]^2},$$
	whenever the distribution of $\eta$ allows such quantities to be defined. Such restorations $u$ will satisfy $\fd \approx u\cdot \eta$, implying $\mathbbm{E}[u]\mathbbm{E}[\eta] \approx \mathbb{E}[\fd]$, and necessarily $\mathbbm{E}[u] \lessapprox \mathbbm{E}[\fd] = \mathbbm{E}[z]$, indicating a shift in the reconstruction mean $\mathbbm{E}[u]$ from the image mean $\mathbbm{E}[z]$. Therefore, designing novel algorithms which address directly the multiplicative noise  is highly desirable.
	
	Let us review below several variational models for restoring images corrupted by multiplicative noise. Rudin and Osher \cite{rud_osh94} introduced in 1994 the following model for image denoising by imposing constraints on the mean and the variance of the noise,
	\begin{equation}\label{eq:RLO}
	\min_u \left\{ TV(u)+\lambda\int_\Omega \bigg(\frac{f^\delta}{u} -1\bigg)^2\right\},
	\end{equation}
	where $TV$ is the total variation and the minimization is performed in the space of bounded variation functions $BV(\Omega)$. Note that problem \eqref{eq:RLO} is well-defined when $f^\delta\in L^\infty(\Omega)$ and $\inf_\Omega f^\delta>0$ (see \cite{chambolle_multiplicative, cha_lio95}), and 
	the unique minimizer $u$ verifies $\inf_\Omega f^\delta\leq u\leq \sup_\Omega f^\delta$ a.e.

	In 2008, Aubert and Aujol \cite{auber_aujol_2008} proposed minimizing the energy 
	\begin{equation}\label{eq:AA}
	E(u) = TV(u) + \lambda\int_\Omega \left( \log(u) + \frac{f^\delta}{u}\right)
	\end{equation} 
	over the set $\{u\in BV(\Omega): u>0\}$  for denoising images degraded by a gamma-law speckle noise, with $f^\delta>0$ as well. We will call this the AA model. The authors demonstrated that minimizers of \eqref{eq:AA} exist, however, the data fidelity term $\int_\Omega \left( \log(u) + \frac{f^\delta}{u}\right)$ is only strictly convex for $u\in(0, 2f^\delta)$ a.e., and not globally convex, so the minimization problem may not have a unique solution. They also noted that \eqref{eq:AA} can be extended to deblurring by involving an appropriate operator $T$,
	\begin{equation*}
	E(u) = TV(u) + \lambda\int_\Omega \left( \log(Tu) + \frac{f^\delta}{Tu}\right).
	\end{equation*}
	
	Concurrently, Shi and Osher \cite{osher_shi_multiplic} introduced two multiplicative noise removal models. The first one, which looks for
	\begin{equation}\label{eq:osher_shi}
	\arg\min_u\left\{ TV(u) + \lambda \int_\Omega \left( a \frac{f^\delta}{u} + \frac{b}{2} \bigg(\frac{f^\delta}{u}\bigg)^2 + c \log(u) \right) \right\},
	\end{equation}
	is a more general AA formulation which can be reduced to \eqref{eq:AA} by setting $b=0$ and $a=c$. Again, the fidelity term is not globally convex. To address this, Shi and Osher let $w=\log(u)$ within the fidelity term of \eqref{eq:osher_shi} and replaced $TV(u)$ with $TV(w)$, thus producing the second model which is convex (in $w$),
	\begin{equation}\label{eq:osher_shi_convex}
	\arg\min_w \left\{TV(w) + \lambda \int_\Omega \left( a f^\delta e^{-w} + \frac{b}{2} (f^\delta)^2 e^{-2w}+c w \right) \right\}.
	\end{equation}
	Having solved the now convex minimization problem for $w$, the true image estimate can be recovered by $u=e^w$. It is worth emphasizing that this partial transformation, which replaces $TV(u)=TV(e^w)$ with $TV(w)$, shifts the regularization to the logarithm of the image intensity. Consequently, the majority of smoothing is performed on image intensities near 0, while larger intensities are much less smoothed.
	
	There have been several extensions of the works 
	\cite{auber_aujol_2008, osher_shi_multiplic} which enforce convexity of the objective functional or  tackle the efficient computation of the minimizers. For instance, the authors in \cite{huang_2009} studied \eqref{eq:osher_shi_convex} with $a=c=1$ and $b=0$, splitting the regularizing and fidelity terms, and adding a quadratic fitting term.
	A primary reason for the formulation in \cite{huang_2009} is the numerical efficiency in solving the minimization with an iterative alternating scheme. 
	
	Rather than transform $w=\log(u)$ to gain convexity, Dong and Zeng \cite{dong_2013} introduced an additional quadratic penalty term to the AA model
	\begin{equation}\label{eq:dong_2013}
	E_T(u) := \lambda TV(u) + \int_{\Omega} \left( \log(Tu) + \frac{f^\delta}{Tu} \right) +\alpha \int_\Omega \left( \sqrt{\frac{Tu}{f^\delta}} -1\right)^2,
	\end{equation}
	thus ensuring convexity of the fidelity term for $\alpha\geq \frac{2\sqrt{6}}{9}$, as well as coercivity of the objective functional for the more general problem of deblurring. Hereafter, we refer to \eqref{eq:dong_2013} as the DZ model. We mention also the interesting approach for multiplicative noise removal in \cite{steidl_teuber}, that uses a data fidelity which is typical for eliminating Poisson noise, and incorporates total variation or nonlocal means as regularizers. Additionally, in recent years there have been new approaches for removing multiplicative noise from images with or without blur. In \cite{ullah2017new} and \cite{TGV_multiplicative}, the authors made use of a fractional-order total variation and a total generalized variation
	penalty, respectively. The paper \cite{wu2020convex} considered a convex scheme for structured multiplicative noise, \cite{liu2016modified} proposed an improved algorithm for the DZ model \cite{dong_2013}, and \cite{zhang2022image} adapted Euler's elastica to the multiplicative noise problem. The reader is referred further to the introduction and the included references on the multiplicative noise topic in \cite{dar_men_res22}. There are also methods addressing  denoising of color images degraded by speckle noise, which employ a total variation function adapted to red-green-blue (RGB)  and hue-saturation-value (HSV) images (see \cite{ullah2017new} and \cite{ wang2021color}, respectively). Studies on using convolutional neural networks for  speckle noise removal  can be found in \cite{ learned_models, nao2022speckle}.
	
	As mentioned above, our aim goes beyond the need of reconstructing images corrupted by multiplicative noise. That is, we focus also on obtaining decompositions of such images along several scales in a variational manner. To this aim, we start  by recalling the approach  by Tadmor, Nezzar and Vese (TNV) \cite{tad_nez_ves1, tad_nez_ves2}, that introduced a hierarchical decomposition based on the Rudin and Osher's (RO) model \eqref{eq:RLO}. Let us emphasize the role of such a decomposition in image restoration. For simplicity, we consider here the case of additive noise removal (see again \cite{tad_nez_ves1, tad_nez_ves2}),  starting from the Rudin, Osher, Fatemi (ROF) model 
	\begin{equation}\label{eq:rof}
	\min_{u}\{\lambda_0\|Tu-f^\delta\|^2+TV(u)\}.
	\end{equation}
	It is not easy to determine an appropriate parameter $\lambda_0$ to ensure that the cartoon (the main features of the image) is well extracted and also the image texture is well preserved while removing the noise. The advantage of the hierarchical process is that it enables separation of noise and image texture in increasingly refined scales by updating parameters, since the texture can be seen as cartoon at finer scales. As a result, the method provides an approximation of the original image $f$ by a sum of image components, that is $f\approx \sum u_j$. As explained in \cite{tad_nez_ves1} (see also the more recent work \cite{multiscale_theory}), the approximation $\sum u_j$ obtained at the $k$-th hierarchical step involving the regularization parameter $\lambda_k$ does not necessarily coincide with the one-step ROF minimization corresponding to the parameter $\lambda_k$. This shows the versatile role of the hierarchical decompositions versus single-step variational models. 
	Motivated by stronger theoretical properties and better restoration effects, tight and refined versions of the multiscale hierarchical decomposition for denoising and deblurring images with additive noise were proposed in \cite{li_res_ves} (see also \cite{mod_nac_ron} for a more general tight version). Moreover, \cite{li_res_ves} proposed for the first time the discrepancy principle for early stopping in the original, tight and refined MHDM.

	In this study, we introduce, test, and provide convergence properties for several hierarchical decomposition procedures designed to recover structured and textured images with multiple scales, when affected by multiplicative noise. Specifically, we propose four multiscale hierarchical decomposition methods for multiplicative noise removal, called SO MHDM, AA MHDM, AA-log MHDM and TNV-log MHDM. Thus, we first formulate a direct MHDM extension of the Shi-Osher model \eqref{eq:osher_shi_convex}, which we abbreviate as SO MHDM, allowing us to adapt the summed-MHDM denoising techniques from \cite{li_res_ves} to the new data fidelity setting, which is no longer quadratic (see Remark \ref{rem:summed}). Secondly, we proceed similarly with 
	the AA model (AA MHDM, for short), and additionally introduce a penalty-modified adaptation of the AA model \eqref{eq:AA} (abbreviated as AA-log MHDM) which handles multiplicative gamma noise and blurring. Finally, we introduce a new variational model, that is a modified RO model, in which the  TV penalty is replaced by  TV(log). Then we derive its multiscale adaptation, thus yielding the so-called TNV-log MHDM. In order to promote  more details in the  reconstruction of the images perturbed by multiplicative noise, we introduce also tight and refined MHDM versions, and emphasize their effect on images with more texture.
	
	We expect that the proposed multiplicative multiscale hierarchical decomposition methods can be extended to applications beyond image restoration, such as image fusion \cite{cui2015detail, bavirisetti2018multi}, image representation \cite{tadmor2009multiscale}, image registration and inverse problems \cite{xu2014multiscale, mod_nac_ron}. 
	
	The current work is organized as follows. In Section \ref{sec:prelim}, we lay out the general strategy of hierarchical decomposition for multiplicative noise degraded images. We justify well-definedness, convergence properties and stopping rules of such schemes in Sections \ref{sec:definedness} and \ref{conv_MHDM}. Tight and refined modifications of the recovery schemes are analyzed in Section \ref{sec:extensions}. We propose several numerical discretizations of the methods in Section \ref{sec:num_schemes}, present  detailed numerical results  in Section \ref{sec:numerical_results}, and point out the robustness of the proposed procedures, as well as the advantages of using one method or another, depending on the structure of the given image.
	\section{Preliminaries}\label{sec:prelim}
	In the multiplicative denoising problem, recovering the true image $z$ in $BV(\Omega)$ amounts to solving the equation
	\begin{equation*}
	\fd = (Tz)\cdot \eta
	\end{equation*}
	in a stable way, where $z$ is assumed to contain features at different scales, as happens for example, in natural images.
	Our aim is to derive multiscale hierarchical decomposition methods for images affected by multiplicative noise, inspired by the idea developed in \cite{tad_nez_ves1, tad_nez_ves2}.
	
	We first briefly recall the setup from \cite{tad_nez_ves1} for images perturbed by additive noise. Let $\lambda_0$ be a positive number and $u_0$ be a solution of the ROF problem \eqref{eq:rof}.
	Define the  sequence $(u_k)\subset BV(\Omega)$ such that
	\begin{equation}\label{iter1}
	u_k\in\arg\min_{u}\{\lambda_k\|Tu-v_{k-1}\|^2+TV(u)\},
	\end{equation}
	with $\lambda_k=2^k\lambda_0$ and $v_{k-1}=f^\delta-\sum_{j=0}^{k-1}Tu_j$, and thus $f^\delta=Tu_0+Tu_1+\dots+Tu_{k-1}+v_{k-1}$. Equivalently, procedure \eqref{iter1} can be expressed as
	\begin{equation*}
	\min_{u}\{\lambda_k\|T(u+x_{k-1})-f^\delta \|^2+TV(u)\},
	\end{equation*}
	for $k\geq 0$, where $\displaystyle{x_{k-1}=\sum_{j=0}^{k-1}u_j}$ and $x_{-1}=0$ (see also \cite{mod_nac_ron}).
	Convergence rates  of $(Tx_k)$ to $f$ have been analyzed in \cite{mod_nac_ron,li_res_ves,tad_nez_ves1,tad_nez_ves2}, while  improved versions have been introduced and studied in \cite{mod_nac_ron,li_res_ves,tang_he}.
	
	For images degraded by multiplicative noise,  the only multiscale hierarchical decomposition we know about is the one from \cite{tad_nez_ves1} and \cite{tad_nez_ves2}, which uses an increasing weighting parameter $\lambda_k$ in the iteration-adapted  Rudin-Osher model \eqref{eq:RLO}.  Namely,
	one starts with
	\begin{equation*}
	u_0 \in\arg\min_u \left\{ \lambda_0 \int_\Omega \bigg(\frac{f^\delta}{u}-1\bigg)^2  + TV(u)\right\},
	\end{equation*}
	where  $\lambda_0$ is a positive parameter, and proceeds further with a similar minimization problem by doubling $\lambda_0$ \sout{the parameter} and considering the new residual $\fd/u_0$ which might contain more features of the original image, and so on. The minimizers $u_k $ obtained iteratively as
	\begin{equation*}
	u_k \in\arg\min_u \left\{ \lambda_k \int_\Omega \bigg(\frac{f^\delta}{uu_0\cdots u_{k-1}} -1\bigg)^2  + TV(u)\right\}
	\end{equation*}
	for $k\geq 0$ (with  $u_{-1}=1$) are well-defined \cite{cha_lio95} and can be characterized as shown in  \cite{tad_nez_ves2}. 
	
	We will work with a general data fidelity term in order to provide analysis in a unifying setting. 
	Assume that $J:L^2(\Omega)\to [0,\infty]$ is a proper function and $H$ is a non-negative data fitting term to be specified later.
	Let $u_k$ be defined as follows: 
	\begin{equation}\label{multiplic}
	u_k\in\arg\min_{u} E_k(u),\,\mbox{with}\quad E_k(u)= \lambda_kH(\fd,T(ux_{k-1}))+J(u),
	\end{equation}
	where  
	$\displaystyle{x_{k-1}=\prod_{j=0}^{k-1}u_j}$, $x_{-1}=1$ and $ {\lambda_{k+1}}={2\lambda_{k}}$, if $k\geq 1$. For example, choosing the data fidelity
	\begin{equation}\label{quadratic_data}
	H(\fd, Tu)=\left\|\frac{\fd}{Tu} -1\right\|^2
	\end{equation}
	yields the Rudin-Osher variational method \eqref{eq:RLO} for deblurring images, while
	\begin{equation}\label{IS_data}
	H(\fd, Tu)=\int_{\Omega}\left(\frac{\fd}{Tu} +\log(Tu)-\log(\fd)-1\right)
	\end{equation}
	is the Itakura-Saito divergence that leads to the Aubert-Aujol model. Note that this divergence is the Bregman distance associated with the  $-\log (u)$ Burg entropy, thus being nonnegative due to the convexity of the entropy. We will mention later more properties of $H$ that will be helpful in the analysis regarding convergence of $(Tx_k)$ to the exact data $f$. 
	More properties of the  multiplicative MHDM schemes introduced in this work, e.g. error estimates, will be shown when using penalty functionals $J$ satisfying
	\begin{equation}\label{properties_multiplic}
	J(uv)\leq J(u)+ J(v),\quad\quad J(u)=J\left(\frac{1}{u}\right),\quad\quad J(1)=0,
	\end{equation}
	for any $u,v\in dom\,J=\{u\in L^2(\Omega): J(u)<\infty\}$. An example of such a function is $J(u)=\varphi(\log(u))$, where $\varphi$ is a seminorm (e.g., the total variation or the $*$-norm).
	
	For the moment, we assume that minimizers $u_k$ in \eqref{multiplic} exist, and instead focus on the analysis of the multiscale decomposition method. Note that existence results will be pointed out for the particular denoising models we deal with in Section \ref{sec:definedness}, while the deblurring models (that is $T\neq I$) will be considered in more detail in our future research.
	\begin{remark}\label{rem1} Let us discuss the choice of the data-fidelity term in \eqref{multiplic}. The first iterate $u_0$ is just a minimizer of
		$\lambda_0H(\fd,Tu)+J(u).$ When searching for $u_1$, we can consider two possibilities. The first one consists of looking for $u_1$ such that the misfit between $T(u_1u_0)$ and  $f^\delta$ becomes smaller than the one between $Tu_0$ and  $f^\delta$, and corresponds to the choice $\lambda_1H(\fd,T(uu_0))+J(u)$ used in \eqref{multiplic}. Thus, the clean data $f$ will be approximated by  $T(u_0u_1\dots u_k)$. The second possibility addresses the ``new" data $f^\delta/Tu_0$ and amounts to finding $u_1$ as a minimizer of $\lambda_1H(\fd/Tu_0,Tu)+J(u).$ In this case, it is desired that the product $Tu_0Tu_1\dots Tu_k$ converges in some sense to  $f$. Our work focuses on the first version, since it is hoped that the product $u_0u_1\dots u_k$ might approximate the true image $z$ in both the denoising and deblurring case.
	\end{remark}
	
	\begin{remark} The Itakura-Saito divergence occuring in the AA-model has the interesting property of being scale invariant in the following sense: $H(\lambda u,\lambda v)=H(u,v)$ for any $\lambda>0$.  Therefore, in the denoising case, it holds that $H(\fd,uu_0)=H(\fd/u_0,u)$, showing that the two approaches from Remark \ref{rem1} coincide. The same holds for $H$ used in the Rudin-Osher model. 
	\end{remark}
	\begin{remark}\label{rem:summed} For clarity, we will at times refer to hierarchical decompositions which break an image down into a sum  $\sum_j u_j$ as summed-MHDM (like the those studied in \cite{tad_nez_ves1,li_res_ves}). We introduce this vocabulary to distinguish from the decomposition techniques which use a multiplicative hierarchical representation $\prod_j u_j$ of an image.
	\end{remark} 
	
	\section{Well-definedness of several models for multiplicative noise removal}\label{sec:definedness}
	Recall we are focusing on multiscale hierarchical decompositions  applied to variational denoising models that address multiplicative noise. Before listing those models, we verify the following equivalence that will ensure well-definedness for some schemes of type \eqref{multiplic} involving particular penalties $J=TV(\log)$.
	
	\begin{proposition} \label{log-exp} The following minimization problems in $BV(\Omega)$
		\begin{align}\label{eq:min1}
		u^* \in \arg\min_{u} \left\{ E(u):= \lambda H(\fd,u) + TV(\log(u)) \right\}
		\end{align}
		and
		\begin{align}\label{eq:exp_min1}
		w^* \in \arg\min_{w} \left\{ \tilde E(w):= \lambda H(\fd,e^w) + TV(w) \right\}
		\end{align}
		are equivalent (that is, they have the same minimum values). {Moreover, the following holds: If $u^*$ is a minimizer of \eqref{eq:min1}, then $\log(u^*)$ minimizes   \eqref{eq:exp_min1}, and if $w^*$ is a minimizer of  \eqref{eq:exp_min1}, then $e^{w^*}$   minimizes \eqref{eq:min1}. }
	\end{proposition}
	\begin{proof}
		Note that, whenever $E$ and $\tilde E$ are defined, one has $E(u) = \tilde E(\log(u))$ and $\tilde E(w) = E(e^w)$, and furthermore, the minimum values of \eqref{eq:min1} and \eqref{eq:exp_min1} are finite. Indeed, one can easily substitute the constant functions $u=1$ or $w=0$ to get a finite energy. 
		To show that minimizers of $\tilde E$ lead to minimizers of $E$, let $w^*\in BV(\Omega)$ minimize \eqref{eq:exp_min1}. Since $e^{w^*}\in BV(\Omega)$ holds by a chain rule property (see \cite{vol67}), we propose this as a candidate minimizer of \eqref{eq:min1}. Indeed, suppose by
		contradiction that there exists $u\in BV(\Omega)$ so that 
		$$ E(u) < E(e^{w^*})= \tilde E(w^*).$$
		Since $E(e^{w^*})<\infty$, we have $E(u)<\infty$ and so $\log(u)\in BV(\Omega)$. Consequently, $\log(u)$ is feasible for $\tilde E$ and 
		$$\tilde E(\log(u)) = E(u) < \tilde E(w^*),$$
		a contradiction to the minimality of $\tilde E(w^*)$. We conclude $u^* = e^{w^*}$ is feasible and minimizes \eqref{eq:min1}.
		
		For the reverse implication, consider $u^*\in BV(\Omega)$ minimizing \eqref{eq:min1} and suppose there is a $w\in BV(\Omega)$ with 
		$$\tilde E(w) < E(u^*).$$
		But then, $e^w\in BV(\Omega)$ and consequently $E(e^w) = \tilde E(w) < E(u^*)$, a contradiction. Furthermore, $w^* = \log(u^*)$ is in $BV(\Omega)$ by the finiteness of $E(u^*)$, so $w^*$ minimizes $\tilde E$. \smallskip
	\end{proof}
	We focus on the following variational models, among which the TNV-log is based on a new energy functional. 
	This is the first work that considers and analyzes these multiscale hierarchical adaptations. 
	
	\textit{1. A particular Shi-Osher (SO) MHDM model:}
	One can  replace the total variation penalty in \eqref{eq:AA} by  $J(u)=TV(\log (u))$ and substitute $w=\log (u)$, thus obtaining
	the \emph{convex} optimization problem
	\begin{equation}\label{so}
	\min_{w} \left\{ TV(w) + \lambda_0 \int_\Omega \left( f^\delta e^{-w} + w \right) \right\}.
	\end{equation}
	This is \eqref{eq:osher_shi_convex}  for $a=c=1$ and $b=0$.	The paper \cite{jin_yang_2010} showed existence and uniqueness of the minimizer $w_0$ when the data $f^\delta\in L^\infty(\Omega)$ satisfy $\inf_\Omega f^\delta>0$. Moreover, the minimizer $w_0$  verifies $\inf_\Omega(\log (f^\delta))\leq w_0\leq \sup_\Omega(\log (f^\delta))$. 	
	We can now apply summed-MHDM, that is solving
	\begin{equation}\label{SO_MHDM}
	w_k= \arg\min_{w} \left\{\lambda_k\int_\Omega \left( f^\delta e^{-(y_{k-1}+w)}+y_{k-1}+w-\log (f^\delta)-1\right)+ TV(w)\right\},
	\end{equation}
	where 
	$y_{k-1}=\sum_{j=0}^{k-1}w_j$ for $k\in\N$, with $y_{-1}=w_{-1}=0$. Note that our data fidelity also incorporates the term $-\log (f^\delta)-1$ in order to build the Itakura-Saito divergence, which is non-negative.
	
	As in the case of $w_0$, existence and uniqueness can be shown for  $w_1$ (and for further iterations), since the updated data $f^\delta/e^{w_0}$ are also away from zero, and so on.
	
	\textit{2. AA MHDM model:}  It was shown in  \cite{auber_aujol_2008} that minimizers $u_0$ of the AA model \eqref{eq:AA}
	exist in $BV(\Omega)$ for data $f^\delta\in L^\infty(\Omega)$ which satisfy $\inf_\Omega f^\delta>0$. Moreover, any minimizer $u_0$  obeys $\inf_\Omega  f^\delta\leq u_0\leq \sup_\Omega f^\delta$.  In order to obtain existence of  $u_1$ and   of further MHDM iterates, one takes into account that $\fd/u_0$ belongs also to $L^\infty(\Omega)$ and verifies $\inf_\Omega f^\delta/u_0>0$. The generated AA MHDM scheme given by
	\begin{equation}\label{eq:AA_multiscale}
	u_k \in \arg\min_u \left\{\lambda_k \int_\Omega \left(\frac{f^\delta}{ux_{k-1}} + \log(ux_{k-1}) - \log(\fd) -1 \right) + TV(u)\right\}
	\end{equation}
	will briefly be discussed theoretically and numerically in the upcoming sections.
	
	\textit{3. The AA-log MHDM model:} One can employ directly the penalty  $J(u)=TV(\log (u))$ in the AA model,
	\begin{equation}\label{AAlog}
	\min_u\lambda\int_{\Omega}\left(\frac{\fd}{u} +\log(u)-\log(\fd)-1\right)+TV(\log(u)).
	\end{equation}
	
	Clearly, the substitution $w=\log (u)$ yields the SO model. By taking into account the latter and by applying Proposition \ref{log-exp}, problem \eqref{AAlog} has a unique minimizer. The MHDM problem 
	\begin{equation}\label{eq:AA_MHDM}
	u_k \in \arg\min_u \left\{ \lambda_k \int_\Omega \left(\frac{\fd}{ux_{k-1}} + \log(ux_{k-1}) - \log(\fd) -1 \right) + TV(\log(u)) \right\}
	\end{equation}
	is also well-defined in this case. Despite transforming into the convex SO model under the appropriate substitution, we include the AA-log method because it extends to deblurring, and in the presence of blur the log-transformation no longer produces a convex problem.
	
	
	\textit{4. The TNV-log model:} We propose a version of the Rudin-Osher minimization problem, where the penalty  $J(u)=TV(\log (u))$ is used instead of just $TV$. It reads  as
	\begin{equation}\label{eq:RO}
	\min_u \left\{ \lambda_0 \int_\Omega \bigg(\frac{f^\delta}{u}-1\bigg)^2  +TV(\log(u))\right\}. 
	\end{equation}
	Since the TNV method is the RO model based multiscale method, correspondingly, we call RO-log model's multiscale form the TNV-log model, given by
	\begin{equation}\label{eq:ARO_MHDM}
	u_k \in \arg\min_u \left\{ \lambda_k \int_\Omega\left( \frac{f^\delta}{ux_{k-1}} -1 \right)^2 + TV(\log(u)) \right\}.
	\end{equation}
	The existence of minimizers $u_0$ can be shown via Proposition \ref{log-exp} and the following result.
	
	\begin{proposition} Let $f^\delta\in L^\infty(\Omega)$ such that $\inf_\Omega f^\delta>0$. Then, there exists at least one solution $w \in BV(\Omega)$ of the problem
		\begin{equation}\label{eq:RO1}
		\min_{w} \left\{ \lambda_0 \int_\Omega \bigg(f^\delta e^{-w}-1\bigg)^2  +TV(w) \right\},
		\end{equation}
		such that $\inf_\Omega \log f^\delta \leq w\leq \sup_\Omega \log f^\delta$ a.e.
	\end{proposition}  
	
	\begin{proof} Let
		$h(x)=(ae^{-x}-1)^2, $ where $a>0$ and $x\in\R$.
		One can prove the result by following the techniques from \cite[Theorem 4.1]{auber_aujol_2008}, taking into account that the function $h$ is nonincreasing on $(-\infty, \log (a))$  and nondecreasing on $(\log (a),\infty).$
	\end{proof}
	As opposed to the situation of the AA model where the $w=\log(u)$ transformation produces a convex problem \eqref{so}, we do not focus on the form \eqref{eq:RO1} since it does not exhibit special properties, and in practice the recoveries are the same or slightly worse than those from the TNV-log model \eqref{eq:RO}. 
	
	\section{Convergence properties of the multiplicative MHDM}\label{conv_MHDM}
	We will consider  a general  data fidelity $H$ and a penalty, $J$, which for the moment does not necessarily satisfy \eqref{properties_multiplic}. Moreover, supposing that the general multiscale hierarchical decomposition schemes \eqref{multiplic} are well-defined (minimizers exist, but might not be unique), we focus on convergence properties of the corresponding iterates. 
	We assume in what follows that the given noisy data $f^\delta$ verify
	\begin{equation}\label{noise}
	H(\fd,f)\leq\delta^2,\quad\delta>0,
	\end{equation}
	where $f$ denotes the exact---that is, non-noisy but potentially blurred---data. Moreover,  existence of a clean image $z$ satisfying  $Tz=f$ and $J(z)<\infty$ is also assumed.
	
	The lemma below shows a couple of basic properties for procedure \eqref{multiplic} (including the SO model, after the logarithm substitution), whenever the iterates are well-defined.
	
	\begin{lemma}\label{lemma1} Assume that $J(1)=0$ and that the iterates $x_k = \prod_{j=0}^k u_j$ given by \eqref{multiplic} are well-defined.
		Then the following inequality holds for any $k\geq 0$,
		\begin{equation*}
		\lambda_kH(\fd, Tx_k)+J(u_k)\leq \lambda_kH(\fd, Tx_{k-1}),
		\end{equation*}
		and the residual $\displaystyle H(\fd, Tx_k)$ decreases for increasing $k$.
		If \eqref{noise} is additionally satisfied and $z/x_{k-1}\in dom\,J$ for any $k\geq 0$, then  \begin{equation}\label{multiplic_monotone2}
		\lambda_kH(\fd, Tx_k)+J(u_k)\leq  \lambda_k\delta^2+J\left(\frac{z}{x_{k-1}}\right)
		\end{equation}
		holds.
	\end{lemma}
	
	\begin{proof}
		According to \eqref{multiplic}, one has
		\[
		\lambda_kH(\fd,T(x_{k}))+J(u_k)
		\leq \lambda_kH(\fd,T(ux_{k-1}))+J(u),
		\]
		for any feasible $u$.
		Using $u=1$ and then $u=z/x_{k-1}$ in \eqref{multiplic}, one obtains the two inequalities for any $k\geq 0$. Clearly, the first one implies that $\displaystyle H(\fd, Tx_k)$ decreases.
	\end{proof}
	
	\begin{remark}\label{comments1}
		Note that the condition      $z/x_{k-1}\in BV(\Omega)$  holds when $x_{k-1}$ is bounded away from zero, since the product of the two bounded variation functions $z$ and $1/x_{k-1}$ has bounded variation, according to \cite{amb_dal}. Indeed, $1/x_{k-1}$ belongs to $BV(\Omega)$ based on the chain rule for $\varphi\circ x_{k-1}=1/x_{k-1}$, since $\varphi(s)=1/s$  is Lipschitz when $s$ is bounded away from zero (see \cite{vol67}). Therefore, Lemma \ref{lemma1} works for the corresponding AA and RO models. Moreover, it is also applicable to the $\log$ models approached in Section \ref{sec:definedness}  due to \eqref{properties_multiplic} for $J=TV(\log)$, as $J(z/x_k)\leq J(z)+J(1/x_k)=J(z)+J(x_k)<\infty$. Last but not least, recall that well-definedness of $x_k$   is ensured in all these models when $T=I$ and the data $\fd\in L^{\infty}$ satisfy $\inf_{\Omega}\fd>0$.
	\end{remark}
	
	Actually, one can show additional convergence properties for the multiplicative  MHDM if the penalty $J$ has the properties  \eqref{properties_multiplic}.\medskip
	
	\begin{proposition}\label{estimate} If  \eqref{properties_multiplic} and \eqref{noise} are satisfied, and the iterates $x_k$ given by \eqref{multiplic} are well-defined, then the following estimate holds for any $k\geq 0$,
		\begin{equation}\label{residual_exact}
		H(\fd,T x_k)\leq\delta^2+ \frac{2J(z)}{(k+1)\lambda_0}. 
		\end{equation}

	\end{proposition}
	
	\begin{proof} Since $J$ satisfies \eqref{properties_multiplic} and $u_k=x_k/x_{k-1}$, one has for any $k\geq 0$,
		\[
		J(z/x_k)-J(z/x_{k-1})\leq  J(u_k).
		\]
		This inequality combined with \eqref{multiplic_monotone2} yields
		\begin{equation}\label{eq1}
		H(\fd, Tx_k)+\frac{1}{\lambda_k}J(z/x_k)\leq\delta^2+ \frac{2}{\lambda_k}J(z/x_{k-1})=\delta^2+\frac{1}{\lambda_{k-1}}J(z/x_{k-1}).
		\end{equation}
		By writing \eqref{eq1} for indices $0,1,...,k$ and summing up, one has for any $k\geq 0$,
		\begin{equation*}
		(k+1)H(\fd, Tx_k)+\frac{1}{\lambda_k}J(z/x_k)\leq \sum_{j=0}^k H(\fd, Tx_j)+\frac{1}{\lambda_k}J(z/x_k) \leq (k+1)\delta^2+ \frac{2}{\lambda_{0}}J(z),
		\end{equation*}
		where the left inequality follows from the monotonicity of the data fidelity term cf. Lemma \ref{lemma1}, and the right one follows from $x_{-1}=1$. This yields \eqref{residual_exact}.
	\end{proof}%
	Clearly, inequality \eqref{residual_exact} holds for the  $\log$ approaches in Section \ref{sec:definedness}, as explained in Remark \ref{comments1}, but not necessarily for the  AA and RO models.\medskip

	\textbf{Summed-MHDM for non-quadratic data fidelity}
	The work \cite{li_res_ves} provided error estimates for MHDM  in case of quadratic data-fidelity. Fortunately, the proof techniques can be similarly employed in the case of  non-quadratic data-fidelities $H(\fd, e^w)$ as long as the existence of minimizers $w_k$ is guaranteed. Hence,  the following result  holds for the SO MHDM defined by \eqref{SO_MHDM} (compare to \cite[Proposition 3.1]{li_res_ves}).
	
	\begin{proposition} Let $f^\delta\in L^\infty(\Omega)$ be such that  $\inf_\Omega f^\delta>0$. Then the  data-fidelity $H$ is monotonically decreasing for increasing $k$ and
		\begin{equation*}
		H(f^\delta,  e^{w_k})\leq \delta^2+\frac{2TV(\log(z))}{\lambda_0(k+1)},\,k\in\N.
		\end{equation*}
		
	\end{proposition}
	
	\subsection{Discrepancy principle stopping rule}
	
	Computing too many multiscale hierarchical iterations  can result in getting back more and more noise in the reconstructed image. Therefore, stopping the procedure  early enough is necessary. In view of this, we propose a stopping rule for \eqref{multiplic} and show convergence properties.
	Let us define the following stopping index,
	\begin{equation}\label{discrepancy_index}
	k^*(\delta):=\max\{k\in\N: H(f^\delta, Tx_k)\geq \tau\delta^2\},\quad\mbox{for some}\,\tau>1.
	\end{equation}
	As shown below, this index exists and convergence of the data fidelity to zero is guaranteed.
	
	\begin{proposition}\label{residual_conv}
		Assume that \eqref{properties_multiplic} and  \eqref{noise} are satisfied, and the iterates $x_k$ given by \eqref{multiplic} are well-defined. Then 
		the stopping index  \eqref{discrepancy_index} is finite.
		If $(k^*(\delta))$ is unbounded as $\delta\to 0$, then  $\displaystyle{\lim_{\delta\to 0}H(\fd,Tx_{k^*(\delta)})=0}$ holds.
	\end{proposition}
	
	\begin{proof}
		By writing \eqref{residual_exact} for $k=k^*(\delta)$ and using \eqref{discrepancy_index}, it follows that
		\begin{equation*}
		\tau \delta^2\leq \delta ^2+ \frac{2}{\lambda_{0}(k^*(\delta)+1)}J(z)
		\end{equation*}
		and thus, the stopping index is finite:
		\begin{equation*}
		k^*(\delta)\leq\frac{2J(z)}{\lambda_0(\tau-1)\delta^2}-1.
		\end{equation*}
		If $(k^*(\delta))$ is unbounded, then \eqref{residual_exact} written for $k=k^*(\delta)$ implies  $\displaystyle{\lim_{\delta\to 0}H(\fd,Tx_{k^*(\delta)})=0}$.
	\end{proof}
	
	\subsection{Convergence of multiplicative MHDM for particular models}\label{sec:particular}
	
	This subsection deals with convergence of the MHDM iterates for  the particular  models considered in the current study. Note that the residual $H$ converges to zero when the procedure is stopped earlier at $k^*(\delta)$ cf. \eqref{discrepancy_index}, as seen in the previous subsection. We now analyze the implications of this  convergence   
	in case of the two  data fidelities employed in the proposed MHDM, namely  the quadratic term of the RO model and the Itakura-Saito distance. As opposed to the MHDM concerning additive noise in images, where convergence is shown with respect to the $L^2$ norm, we can prove  only pointwise convergence on subsequences a.e. for the MHDM corresponding to multiplicative noise.
	
	\begin{proposition} Assume that \eqref{properties_multiplic} and  \eqref{noise} are satisfied, and the iterates $x_k$ given by \eqref{multiplic} are well-defined, whenever the data fidelity $H$ is defined by  \eqref{quadratic_data} or \eqref{IS_data}. If $(k^*(\delta))$ is unbounded as $\delta\to 0$, then $(Tx_{k^*(\delta)})$ converges a.e. to $f$ on a subsequence. In particular for the denoising case, one has a.e. convergence of $(x_{k^*(\delta)})$ on a subsequence to the true image.
	\end{proposition}
	
	\begin{proof} According to Proposition \ref{residual_conv}, one has $\displaystyle{\lim_{\delta\to 0}H(\fd,Tx_{k^*(\delta)})=0}$. If the data fidelity is given by \eqref{quadratic_data},
		then  $\left(\frac{f^\delta}{Tx_{k^*(\delta)}}\right)$ converges strongly to $1$  in $L^2(\Omega)$. This yields a.e. convergence  of $(\frac{f^\delta}{Tx_{k^*(\delta)}})$ to 1 on a subsequence, thus a.e.\, convergence of $(Tx_{k^*(\delta)})$ to $f$ on a subsequence.  
		Now consider $H$ given by \eqref{IS_data}. Then the convergence of the residual (cf. Proposition \ref{residual_conv}) implies that the positive sequence $(d(\fd,Tx_{k^*(\delta)}))$ converges to zero in the $L^1(\Omega)$ norm, where
		$d(\fd,Tx_{k^*(\delta)})=\frac{\fd}{Tx_{k^*(\delta)}}+\log (Tx_{k^*(\delta)})-\log (\fd)-1$.  Consequently, it converges a.e. to zero on a subsequence.  It follows that the sequence $
		(Tx_{k^*(\delta)})$ is (a.e.) pointwise bounded in $[0,\infty)$, otherwise a subsequence would diverge to $+\infty$, which would contradict $d(\fd,Tx_{k^*(\delta)})\to 0$. Therefore, $(Tx_{k^*(\delta)})$ converges on a subsequence to some nonnegative function $g$ a.e., implying a.e. convergence of $(d(\fd,Tx_{k^*(\delta)}))$ to $d(f,g)$. Uniqueness of the limit yields $d(f,g)=0$ a.e., that is $g=f$ a.e., due to the strict convexity of the Burg entropy which defines the (pointwise) Itakura Saito  distance $d$.
	\end{proof}
	
	\section{Extensions of the multiplicative MHDM}\label{sec:extensions}
	\subsection{A tight multiplicative MHDM}\label{sec:tighter_MHDM}
	In this section, we adapt to the multiplicative noise case the  tight hierarchical decomposition method \cite{mod_nac_ron} proposed in the additive noise context. That tight version incorporated an additional penalization, namely on the entire approximation $(x_k)$, in order to obtain better convergence properties of $(x_k)$. Since this section follows the structure of the tight MHDM in the case of additive noise \cite{li_res_ves}, we introduce the tight method in the new setting by omitting proof details.
	
	Let $(a_k)$ be a  sequence of nonnegative numbers such that for any $k\geq 1$,
	\begin{equation}\label{an}
	\lim_{k\to\infty}a_k=0\quad \mbox{and}\quad a_k\leq a_{k-1}.
	\end{equation}
	Set $\lambda_0$ to be a positive number and let $(\lambda_k)\subset (0,\infty)$ verify the following relaxed inequality
	\begin{equation}\label{sequence}
	2\lambda_{k}\leq\lambda_{k+1},\quad k\geq 0,
	\end{equation}
	rather than the equality $2\lambda_{k}=\lambda_{k+1}$. Finally, determine $u_k\in BV(\Omega)$ as a solution of 
	\begin{equation*}
	\min_u F_k(u),\quad \mbox{with}\quad F_k(u)=\lambda_k H(\fd,T(ux_{k-1}))+\lambda_ka_kJ(ux_{k-1})+J(u),
	\end{equation*}
	with, as before, $\displaystyle{x_{k-1}=\prod_{j=0}^{k-1}u_j}$, $x_{-1}=1$. The tight formulation, then, is augmented by a new penalization term $\lambda_k a_k J(ux_{k-1})$.
	
	\begin{remark} The tight versions of   the  denoising models presented in Section \ref{sec:definedness} are also well-defined (similar arguments).
	\end{remark}
	
	Under the assumptions of Lemma \ref{lemma1}, one can derive similarly  the following inequalities,
	\begin{equation}\label{iter4_tight}
	\lambda_k H(\fd,Tx_{k})+\lambda_k a_kJ(x_k)+J(u_k)\leq \lambda_k H(\fd,Tx_{k-1})+\lambda_k a_k J(x_{k-1}),
	\end{equation}
	\begin{equation*}
	\lambda_k H(\fd,Tx_{k})+\lambda_k a_kJ(x_k)+J(u_k)\leq\lambda_k a_kJ(z)+ J(z/x_{k-1}) + \lambda_k \delta^2, \quad k\geq 0.
	\end{equation*}
	Note that \eqref{iter4_tight} yields the decreasing monotonicity of $ H(\fd,Tx_{k})+a_kJ(x_k)$, which is a type of residual in the tight method. If we further require
	\begin{equation}\label{series}
	\displaystyle{\sum_{k=0}^\infty a_k<\infty}
	\end{equation}
	and define the stopping index also by a discrepancy rule
	\begin{equation}\label{discrepancy_tight}
	k^*(\delta):=\max\{k\in\N: H(\fd,Tx_{k})+a_kJ(x_k)\geq \tau\delta^2\},\quad\mbox{for some}\,\tau>1,
	\end{equation}
	then the results below can be established in a similar manner to the ones for the multiplicative MHDM when  $J$ verifies \eqref{properties_multiplic}.
	
	\begin{proposition}\label{estimate_tight_noise}
		Let conditions  \eqref{properties_multiplic},
		\eqref{noise}, \eqref{an} and  \eqref{sequence}  be satisfied. Then the following estimate holds for any  $k\geq 0$,
		\begin{equation*}
		H(\fd,Tx_{k})+a_kJ(x_k)\leq \delta^2+ \left(\sum_{j=0}^k a_j\right)\frac{J(z)}{k+1}+ \frac{2J(z)}{(k+1)\lambda_0}.
		\end{equation*}
		Moreover, if  \eqref{series} is verified, then the stopping index defined by \eqref{discrepancy_tight} is finite. Additionally,
		\begin{enumerate}
			\item If $(k^*(\delta))$ is unbounded, then $\displaystyle{\lim_{\delta\to 0}H(\fd,Tx_{k^*(\delta)})=0}$ and $\displaystyle{\lim_{\delta\to 0}a_{k^*(\delta)}J(x_{k^*(\delta)})=0}$.
			\item If the stopping index is chosen as $\displaystyle{k^*(\delta)\sim \frac{1}{\delta^2}}$, then
			$$H(\fd,Tx_{k^*(\delta)})+a_{k^*(\delta)}J(x_{k^*(\delta)})=O({\delta^2}).$$
		\end{enumerate}
	\end{proposition}
	By  adapting the techniques from  \cite[Section 4]{li_res_ves} to the multiplicative noise case with the help of  the  condition $\limsup_{k\to\infty}\frac{2^k}{\lambda_k a_k}=0$, one can show  $\displaystyle{J(x_{k^*(\delta)})\to J(z)}$, demonstrating that the recoveries have the same level regularity as the clean image. Compare also  to \cite[Theorem 2.5]{mod_nac_ron} which addresses the tight summed-MHDM. Additionally, the convergence in the sense of Subsection \ref{sec:particular} holds.
	
	\subsection{A refined multiplicative MHDM}\label{sec:refined_MHDM}
	In order to promote specific properties of the $u_k$ components, we propose a multiplicative counterpart of the refined method introduced in \cite{li_res_ves}. Thus, we allow the penalization on the hierarchical component to be a functional different from $J$, that is different from $TV$ or $TV(\log)$. Although it can vary in every iteration as stated in \cite{li_res_ves}, we consider it fixed (for fixed $J$) hereafter and denote it by $R$. 

	In particular, we require $R:L^2(\Omega)\to\R\cup\{\infty\}$ to be a seminorm which is weakly lower semicontinuous and verifies the following inequality for some $c>0$:
	\begin{equation*}
	R(u)\leq cJ(u), \forall u\in dom\,R.
	\end{equation*}
	Construct a  sequence $(u_k)\subset BV(\Omega)$ with $u_k$ as a solution of
	\begin{equation*}
	\min_{u}F_k(u),\quad \mbox{with}\quad F_k(u)=\lambda_k H(\fd,T(ux_{k-1}))+\lambda_ka_kJ(ux_{k-1})+R(u), 
	\end{equation*}
	where $\lambda_k$ and $a_k$ are defined as in the tight formulation. One can similarly derive the estimate
	\begin{equation*}
	H(\fd,Tx_{k})+a_kJ(x_k)\leq \delta^2+ \left(\sum_{j=0}^k a_j\right)\frac{J(z)}{k+1}+ \frac{2R(z)}{(k+1)\lambda_0},
	\end{equation*}
	as well as the same convergence results under the same assumptions, in addition to the ones above for $R$. An improved behavior (as compared to the tight and the regular MHDM versions) will be shown  numerically by considering $R=\|\cdot\|_*$ or $R=\|\log(\cdot)\|_*$ when  $J=TV$ or $J=TV(\log)$, respectively. 
	%
	%
	\section{Numerical schemes for multiplicative MHDM minimiza- tion
	} \label{sec:num_schemes}
	Here we introduce numerical discretizations for the three classes of MHDM problems we consider: Shi-Osher adaptations, AA-like models, and TNV inspired methods. 
	\subsection{Shi-Osher model adaptations}
	\subsubsection{Shi-Osher MHDM}
	
	\textbf{Discretization of Euler-Lagrange equations:}\\
	We develop a numerical scheme for the Shi-Osher (SO) model adapted to multiscale hierarchical decomposition (MHDM). Our goal is to recover $w_k$ which satisfies (\ref{SO_MHDM}).
	
	That is, given a partial reconstruction $y_{k-1}:=\sum_{j=0}^{k-1} w_j$, we seek a sufficiently regular $w_k$---as imposed by $TV(\cdot)$---so that the sum $w_k+y_{k-1}$ fits $f^\delta $ according to the data fidelity term. 
	Using a gradient descent
	scheme to solve the associated Euler-Lagrange equation for (\ref{SO_MHDM}) with Neumann boundary conditions, we can numerically determine $w_k$ by running
	\begin{equation}\label{eq:GD_multiscale}
	\begin{cases}\frac{\partial w}{\partial t} =  \text{div}\left( \frac{\nabla w}{|\nabla w|}\right) - \lambda_k(1-f^\delta e^{-(w+y_{k-1})}) \ \text{ in } \Omega, \\
	\frac{\partial w}{\partial \vec{n}} = 0 \ \text{ in } \partial \Omega,
	\end{cases}
	\end{equation}
	to equilibrium, followed by updating $y_k = w_k + y_{k-1}$. Having recovered $y_k$, we subsequently obtain the reconstruction $x_k$ via an exponential transform. We omit the discretization for  SO MHDM, SO MHDM tight (Sec.~\ref{sec:SO_tight}) and SO MHDM refined (Sec.~\ref{SO_Refine}), as the transformed problems are now summed-MHDM procedures for which discretizations can be found in \cite{li_res_ves}. See Subsection \ref{sec:inits} for initializations of each scheme.
	
	Let \texttt{ShiOsher}($f, y_{k-1} ,\Delta t, \lambda_k, \epsilon,\texttt{maxIter}$) be the numerical solution to \eqref{eq:GD_multiscale} after running $n=\texttt{maxIter}$ times. Then the image restoration algorithm proceeds as follows:
	\begin{flushleft}
		\begin{framed}
			\textbf{Algorithm:} ShiOsher\\
			INPUT: noisy image $f^\delta = z\cdot \eta$, where $\eta$ is some multiplicative noise and $z$ is the original image. \\
			OUTPUT: $x_{\texttt{numScales}}$, an approximation to $z$. \\
			\begin{itemize}[nosep,after=\strut, leftmargin=1em, itemsep=3pt]
				\item Set $y_{-1}=0$, $\lambda_0=0.01$ and $\epsilon=0.01$ (or some small constant). 
				\item Choose \texttt{maxIter}.\\
				For $k=0,1,2,\dots, \texttt{numScales}$ do%
				\begin{itemize}
					\item Set: $w_k$ = ShiOsher($f^\delta , y_{k-1} ,\Delta t, \lambda_{k}, \epsilon,\texttt{maxIter})$%
					\item Update: $y_k = w_k + y_{k-1}$%
					\item Update: $\lambda_{k+1} = 2\lambda_{k}$%
				\end{itemize}
				\item Return: $x_{\texttt{numScales}} = e^{y_{\texttt{numScales}}}$%
			\end{itemize}
		\end{framed}
	\end{flushleft}%
	\textbf{ADMM for Shi-Osher MHDM:}
	In addition to the Euler-Lagrange approach, we consider the popular alternating direction method of multipliers (ADMM) for the convex optimization problem obtained by the Shi-Osher formulation. Recall, given $y_{k-1}$ and $\lambda_k$, we solve (\ref{SO_MHDM})
	to form the multiscale reconstruction $x_k=e^{y_k}$ with $y_k = \sum_{j=0}^k w_j$. We split the problem into subproblems, minimizing the data fidelity term $\lambda_k \int f^\delta e^{-(w+y_{k-1})} + (w+y_{k-1})$ and the regularizing term $TV(w)$ separately, subject to the condition these minimizers match. This gives the ADMM formulation
	\begin{align}
	\theta^{j+1} &= \arg\min_\theta \lambda_k \int \left( f^\delta e^{-(\theta+y_{k-1})} + (\theta+y_{k-1}) \right)+ \frac{\rho}{2}\| \theta - \psi^j  +  \vartheta^j\|^2_2 \label{eq:admm_primal},\\
	\psi^{j+1} &= \arg\min_{\psi} TV(\psi) + \frac{\rho}{2} \| \theta^{j+1} - \psi +  \vartheta^j\|^2_2 \label{eq:admm_psi},\\
	\vartheta^{j+1} &=  \vartheta^j + \theta^{j+1} - \psi^{j+1} \label{eq:admm_residual} , 
	\end{align}
	which proceeds iteratively in $j$, forming solution $w_k=\theta^\infty$. Here, $\rho$ is a constant parameter of the scheme. The stopping condition for ADMM is determined by some tolerance $\epsilon>0$ and is satisfied whenever 
	$$ \max\left\{\|\theta^{j+1}-\theta^j\|^2, \| \vartheta^{j+1}- \vartheta^j\|^2, \|\psi^{j+1} - \psi^j\|^2, \|\theta^{j+1} - \psi^{j+1}\|^2  \right\} < \epsilon. $$
	
	To perform the minimizations in \eqref{eq:admm_primal}, we use Newton's iteration. For \eqref{eq:admm_psi}, we use an exact total variation minimization routine \cite{Chambolle2009OnTV} provided at \url{http://www.cmap.polytechnique.fr/~antonin/software/}.

	\subsubsection{Shi-Osher Tight MHDM}\label{sec:SO_tight}
	\textbf{Discretization of Euler-Lagrange equations:}
	
	For the tight SO scheme, we consider the modified objective function 
	\begin{align}\label{eq:osher_tight}
	w_k = \arg\min_w \left\{TV(w) + \lambda_k a_k TV(w+y_{k-1}) + \lambda_k \int \left(f^\delta e^{-(w+y_{k-1})} + (w+y_{k-1}) \right)\right\}.
	\end{align}
	Notice that the only alteration from the standard Shi-Osher model is the additional $TV(w+y_{k-1})$ term, so the resulting Euler-Lagrange equations will be modified solely by this term. Additionally, the boundary condition will require $\vec{n}\cdot\nabla(w+y_{k-1})=0$. However, since $y_{k-1}=\sum_{j=0}^{k-1} w_j$ where $\vec{n}\cdot \nabla w_j=0 $ on the boundary, we need only to impose $\vec{n}\cdot \nabla{w} = 0$ on $\partial \Omega$.

	Consequently, the Euler-Lagrange equation for \eqref{eq:osher_tight} after considering  artificial time is
	\begin{align}\label{eq:osher_tight_GD}
	\begin{cases} \frac{\partial w}{\partial t} = \text{div}\left( \frac{\nabla w}{|\nabla w|}\right) + \lambda_ka_k\text{div}\left(\frac{\nabla(w+y_{k-1})}{|\nabla(w+y_{k-1})|} \right)- \lambda_k(1-fe^{-w-y_{k-1}})  \ \text{ in } \Omega, \\
	\frac{\partial w}{\partial \vec{n}} = 0 \ \text{ in } \partial \Omega.
	\end{cases}
	\end{align} 
	
	Compared with the Shi-Osher MHDM, the Shi-Osher Tight MHDM has the same initialization and parameters except $\lambda_{k+1} = 3\lambda_{k}$, with   $a_0=1$, $a_k = \frac{a_0}{(k+1)^{3/2}}$.

	\textbf{ADMM for Shi-Osher tight MHDM:}
	Given $y_{k-1}$, $\lambda_k$, $ a_k$, the ADMM  proceeds as follows: 
	\begin{align}
	\theta^{j+1} &= \arg\min_\theta \lambda_k \int \left( f^\delta e^{-(\theta+y_{k-1})} + (\theta+y_{k-1}) \right) + \frac{\rho}{2}\| \theta - \frac{\psi_1^j + \psi_2^j}{2} +  \vartheta^j\|^2_2 \label{eq:tight_admm_primal},\\
	\psi_1^{j+1} &= \arg\min_{\psi} \lambda_k a_k TV(\psi+y_{k-1}) + \frac{\rho}{2} \| \theta^{j+1} - \frac{\psi + \psi_2^j}{2} +  \vartheta^j\|^2_2 \label{eq:tight_admm_psi1} ,\\
	\psi_2^{j+1} &= \arg\min_{\psi} TV(\psi ) + \frac{\rho}{2} \| \theta^{j+1} - \frac{\psi^{j+1}_1 + \psi}{2} +  \vartheta^j\|^2_2 \label{eq:tight_admm_psi2}, \\
	\vartheta^{j+1} &=  \vartheta^j + \theta^{j+1} - \frac{ \psi_1^{j+1}+\psi_2^{j+1}}{2} \label{eq:tight_admm_residual}, 
	\end{align}
	iteratively forming the solution $w_k= \theta^{\infty}$. 
	We set the stopping criterion as before and approach \eqref{eq:tight_admm_primal} and  \eqref{eq:tight_admm_psi2} as above. For \eqref{eq:tight_admm_psi1} we use $g=\psi+y_{k-1}$ and rescale the problem. This gives
	\begin{align*}
	g^{j+1}=&\arg\min_g \frac{4 a_k\lambda_k}{\rho} TV(g) + \frac{1}{2}\| 2\theta^{j+1} +2  \vartheta^j + y_{k-1} - \psi_2^j - g\|^2\ ,
	\end{align*}
	and finally $\psi^{j+1}_1 = g^{j+1} - y_{k-1}$ can be recovered. 
	\subsubsection{Shi-Osher Refined MHDM} \label{SO_Refine}
	For the refined version, we consider
	\begin{align}\label{eq:osher_refined}
	w_k = \arg\min_w \left\{\|w\|_* + \lambda_k  a_k TV(w+y_{k-1}) + \lambda_k \int \left( f^\delta e^{-(w+y_{k-1})} + (w+y_{k-1}) \right)\right\},
	\end{align}
	where $\|w\|_*= \sup_{TV_\epsilon(\phi)\neq 0} \frac{\langle w, \phi\rangle}{TV_\epsilon(\phi)}$, and $TV_\epsilon(\phi) = \int_\Omega \sqrt{\epsilon^2 + |\nabla\phi|^2}$, with $\epsilon>0$. Thus,  the main modification is the $*$-norm term. Assuming $w$ is given, we can determine $\phi$ by studying the Euler-Lagrange equation associated with maximizing $ \langle w, \phi\rangle/TV_\epsilon(\phi)$ over $\phi$, as done in \cite{li_res_ves}. Likewise, given $\phi$, the Euler-Lagrange equation for the refined problem \eqref{eq:osher_refined} is modified from \eqref{eq:osher_tight_GD} only by the $*$-norm term.  Thus, we can find $w_k$ through the Euler-Lagrange equation for \eqref{eq:osher_refined}. More details of this procedure are provided in \cite{li_res_ves}. We solve for $\phi$ and $w$ by alternatingly time-stepping in each variable, using the same semi-implicit discretizations as done in the regular and tight formulations. Parameters are chosen as in the tight formulation,  while initializations can be found in Subsection \ref{sec:inits}.
	\subsection{AA based models}
	\subsubsection{AA MHDM}
	We extended the original AA model \eqref{eq:AA} to an MHDM method, as given in \eqref{eq:AA_multiscale}.  We use a semi-implicit method to solve the  Euler-Lagrange equation for \eqref{eq:AA} as formulated in \cite{auber_aujol_2008}. The discretization details are omitted here  due to similarity with the subsequent AA-log schemes. 
	\subsubsection{AA-log MHDM with $TV(\log(u))$ penalty term}
	The AA-log model directly addresses images corrupted by blur and multiplicative noise, since now with blur, the substitution $w=\log u$ no longer produces a convex problem as in the SO formulation.
	We seek a multiscale solution $x_k = \prod_{j=0}^k u_j$, where each $u_k$ satisfies \eqref{eq:AA_MHDM}.
	
	Using the Euler-Lagrange equation associated with \eqref{eq:AA_MHDM}, we obtain a time dependent PDE to determine $u_k$,
	\begin{align}\label{eq:AA_EL_GD}
	\begin{cases} 
	\frac{\partial u}{\partial t}= \text{div}\left( \frac{\nabla u}{|u||\nabla u|}\right) +\frac{|\nabla u| }{u|u|} - \lambda_k T^*\left( \frac{1}{T(ux_{k-1})} - \frac{f^\delta }{[T(ux_{k-1})]^2} \right)x_{k-1}, & \text{ on } \Omega\\
	\nabla u \cdot \vec{n} =0, \ &\text{ on } \partial \Omega.
	\end{cases}
	\end{align}
	A numerical scheme for \eqref{eq:AA_EL_GD} proceeds in much the same manner as for the SO models, which comprises of composing forward and backward finite difference operators for computing the divergence-of-gradient terms, and uses centered differences for first derivative terms. Then, we isolate the terms linear in $u^n_{ij}$ and exchange them for $u^{n+1}_{ij}$ to make a semi-implicit scheme:
	\begin{align}\label{eq:AA_scheme2}
	u^{n+1}_{ij} &= \frac{1}{1+\Delta td^{ij}_{u^n}}\cdot \bigg\{ u^n_{ij} -\Delta t\lambda_k T^*\left[\frac{1}{T(u^nx_{k-1})}-\frac{f^\delta}{[T(u^nx_{k-1})]^2}\right]_{ij}\cdot x_{k-1,ij}  \nonumber\\
	& + \Delta t \frac{\sqrt{D^0_x(u^n_{ij})^2 + D^0_y(u^n_{ij})^2}}{u^n_{ij}|u^n_{ij}|} + \Delta t\chi[d,i,j]^{u^n}_{u^n}\bigg\},
	\end{align} 
	where
	\begin{align}
	\chi[d,i,j]^{u^n}_{u^n} &:=d(u^n_{ij})u^n_{i+1,j} + d(u^n_{i-1,j})u^n_{i-1,j}+d(u^n_{ij})u^n_{i,j+1} + d(u^n_{i,j-1})u^n_{i,j-1},\label{eq:XI}\\
	d_{u^n}^{i,j} &:= 2d(u^n_{ij})+d(u^n_{i-1,j}) + d(u^n_{i,j-1}), \label{eq:d_u}\\
	d(u_{ij})&:= \frac{1}{|u_{ij}|\sqrt{\epsilon^2 + (D^+_x u_{ij})^2 + (D^+_yu_{ij})^2}}. \label{eq:d_of_u}
	\end{align}
	
	Here $D^{+}_x$ and $D^{+}_y$ are the standard forward differences, and $D^0_x$ and $D^0_y$ the centered differences in the first and second coordinates, respectively.

	To initialize, we set $x_{-1,ij}=1$ and choose $u^0$ according to Subsection \ref{sec:inits}. Denoting the solution of (\ref{eq:AA_scheme2}) after $n=\texttt{maxIter}$ iterations by $\texttt{AAlog\_blur}(f^\delta,x_{k-1},\Delta t, \lambda_k,  a_k, T, \epsilon, \texttt{maxIter})$, we present the basic workflow for using the AA-log model. 
	\begin{flushleft}
		\begin{framed}
			\textbf{Algorithm:} AAlog\_blur\\
			INPUT: noisy image $f^\delta = Tz\cdot \eta$, where $\eta$ is multiplicative noise and $T$ is a blurring kernel. \\
			OUTPUT: $x_{\texttt{numScales}}$, a multiscale approximation to $z$.
			\begin{itemize}[nosep,after=\strut, leftmargin=1em, itemsep=3pt]
				\item Set: $x_{-1}=1$, $\lambda_0=0.01$, and $\epsilon=0.01$ (or some small constant). 
				\item Choose: \texttt{maxIter}.\\
				For $k=0,1,2,\dots, \texttt{numScales}$ do
				\begin{itemize}
					\item Set $u_k$ = AAlog\_blur($f^\delta,x_{k-1},\Delta t, \lambda_k,  a_k, T, \epsilon, \texttt{maxIter}$)
					\item Update: $x_k = u_k\cdot x_{k-1}$
					\item Update: $\lambda_{k+1} = 2\lambda_{k}$. 
				\end{itemize}
				\item Return: $x_{\texttt{numScales}}$
			\end{itemize}
		\end{framed}
	\end{flushleft}
	\subsubsection{AA-log Tight MHDM}
	
	The tight version modifies the objective function by adding an additional regularizing term 
	\begin{equation} \label{eq:AA_tight}
	u_k\in\arg\min_u TV(\log(u)) +\lambda_k a_kTV(\log(ux_{k-1}))+ \lambda_k \int_\Omega \left( \log(T(ux_{k-1})) + \frac{f^\delta}{T(ux_{k-1})}\right).
	\end{equation}
	The Euler-Lagrange equation for \eqref{eq:AA_tight} gives the time dependent PDE \eqref{eq:AA__tight_GD} for $u_k$, which we can run to equilibrium:
	\begin{align}\label{eq:AA__tight_GD}
	\frac{\partial u}{\partial t} &= -\lambda_k T^*\left( \frac{1}{T(ux_{k-1})} - \frac{f^\delta }{[T(ux_{k-1})]^2} \right)x_{k-1}\\ 
	&+ \text{div}\left( \frac{\nabla u}{|u||\nabla u|}\right) +\frac{|\nabla u| }{u|u|}+\lambda_k a_kx_{k-1} \left[ \text{div} \left(\frac{\nabla(ux_{k-1})}{|ux_{k-1}||\nabla(ux_{k-1})|}\right) + \frac{|\nabla(ux_{k-1})|}{ux_{k-1}|ux_{k-1}|}\right],\nonumber\\
	\nabla u \cdot \vec{n} &=0, \ \text{ on } \partial \Omega. \nonumber
	\end{align}

	To discretize \eqref{eq:AA__tight_GD}, we introduce the notation $Z^n = u^nx_{k-1}$ and follow the same semi-implicit strategy from before (recalling (\ref{eq:XI}, \ref{eq:d_u}, \ref{eq:d_of_u})), giving
	\begin{align*}
	u^{n+1}_{ij} &= \frac{1}{1+\Delta td_{u^n} + \Delta t \lambda_k a_k x^2_{k-1,ij}d^{i,j}_{Z^N} }\cdot \Bigg\{ u^n_{ij} - \Delta t \lambda_k T^*\left[\frac{1}{T(u^nx_{k-1})}-\frac{f^\delta}{[T(u^nx_{k-1})]^2}\right]_{ij}\!\!\!\cdot x_{k-1,ij} \nonumber\\
	&+ \Delta t\left( \chi[d,i,j]_{u^n}^{u^n} + \frac{\sqrt{D^0_x(u^n_{ij})^2 + D^0_y(u^n_{ij})^2}}{u^n_{ij}|u^n_{ij}|} \right)\nonumber\\
	&+
	\Delta t\lambda_k a_kx_{k-1,ij}\left( \chi[d,i,j]^{Z^n}_{Z^n}+\frac{\sqrt{D^0_x(Z^n_{ij})^2 + D^0_y(Z^n_{ij})^2}}{Z^n_{ij}|Z^n_{ij}|} \right)\Bigg\}.
	\end{align*}
	
	The AA-log Tight MHDM has the same initialization and parameters as the AA-log MHDM, except $ a_0=1$, $ a_k = \frac{  a_0}{(k+1)^{3/2}}$, and $\lambda_{k+1} = 3\lambda_{k}$. 
	
	\subsubsection{AA-log Refined MHDM}
	The refined version reads as follows,
	\begin{equation} \label{eq:AA_refined}
	u_k\in \arg\min_u\|\log(u)\|_* +\lambda_k a_kTV(\log(ux_{k-1}))+ \lambda_k \int_\Omega \left( \log(T(ux_{k-1})) + \frac{f^\delta}{T(ux_{k-1})}\right).
	\end{equation}

	Thus, the Euler-Lagrange equation for \eqref{eq:AA_refined} is modified from the tight formulation only by the $\|\cdot \|_*$ term. Moreover,  the discretization is determined by the same process as for the AA-log tight method combined with alternating time-stepping with a test-function $\phi$, as done in the SO refined MHDM (compare also to the refined version of the summed-MHDM in \cite{li_res_ves}).
	We omit further details, due to the similarity to the previous schemes.
	
	\subsection{TNV-log Models}
	\subsubsection{TNV-log}
	The TNV-log model \eqref{eq:ARO_MHDM} is the TNV method modified with a $TV(\log(u))$ penalty.
	The resulting dynamical PDE from the Euler-Lagrange equations for \eqref{eq:ARO_MHDM} is 
	\begin{align}\label{eq:TNV_log_EL}
	\frac{\partial u}{\partial t} &= \text{div}\left(\frac{\nabla u}{|u||\nabla u|}\right) + \frac{|\nabla u|}{u|u|} - 2\lambda_k T^*\left[\frac{f^\delta}{[T(ux_{k-1})]^2} \left(1-\frac{f^\delta}{T(ux_{k-1})}\right) \right] x_{k-1}\ &\text{in } \Omega,\\
	\nabla u\cdot \vec{n} & =0 \ &\text{in } \partial \Omega.\nonumber
	\end{align}
	
	This, differing from the AA-log MHDM by only its fidelity term, takes on a similar discretization and algorithm, with updates for $u^{n+1}$ given by
	\begin{align}\label{eq:ARO_scheme}
	u^{n+1}_{ij} &= \frac{1}{1+\Delta td^{ij}_{u^n}}\cdot \bigg\{ u^n_{ij} -2\Delta t\lambda_k T^*\left[\frac{f^\delta}{[T(u^nx_{k-1})]^2} \left(1-\frac{f^\delta}{T(u^nx_{k-1})}\right) \right]_{ij}\cdot x_{k-1,ij} + \nonumber\\
	& \quad\quad \Delta t \frac{\sqrt{D^0_x(u^n_{ij})^2 + D^0_y(u^n_{ij})^2}}{u^n_{ij}|u^n_{ij}|} + \Delta t\chi[d,i,j]^{u^n}_{u^n}\bigg\}.\nonumber
	\end{align}
	\subsubsection{TNV-log Tight} 
	Likewise, with the TNV-log tight formulation the Euler-Lagrange equation is the same as that for the AA-log tight MHDM scheme \eqref{eq:AA__tight_GD}, except with the appropriate fidelity term swapped for TNV, as was done in \eqref{eq:TNV_log_EL}. The discretization is akin to that of the AA-log tight MHDM, and we omit further details.

	\subsection{Initializations}\label{sec:inits}
	We take a moment to discuss initializing our MHDM schemes. Each of the proposed iterative schemes minimizes an energy as shown in (\ref{multiplic}), starting with some initialization $u^0$ which, hopefully, is close to the minimizer. In this work, we propose two types of choices with the following motivation. 
	
	\textbf{Fidelity minimizing initializations:}
	
	One approach is to choose $u^0$ to minimize the fidelity $H(\fd, Tu^0x_{k-1})$ without regard to the penalty term, which leads to non-constant initializations. For the SO MHDM schemes, if we aim to minimize $\int \left(f^\delta e^{-(w^0+y_{k-1})} + (w^0+y_{k-1}) \right) $, then the optimal initialization is $w^0 = \log(\fd) - y_{k-1}$. 
	For the AA MHDM, AA-log MHDM and TNV-log schemes, the optimal initialization $u^0$ minimizing $H(\fd,Tux_{k-1})$ is $u$ such that $Tux_{k-1} = \fd$, which amounts to solving a deblurring problem. For simplicity, we will just use $u^0= \fd/x_{k-1}$.
	
	\textbf{Penalty minimizing initializations:}
	
	On the other hand, we can choose  $u^0$ to minimize the penalty $J(\cdot)$. For $J(\cdot) = TV(\log(\cdot))$ or $TV(\cdot)$, these would be constant functions whose values we can optimally choose to minimize the remaining data fidelity $ H(\fd,Tux_{k-1})$. For the tight and refined schemes, the additional terms $ a_k\lambda_k TV(\log(u^0 x_{k-1}))$ and $\|\log(u^0)\|_*$ do not necessarily vanish as was the case before. In the tight schemes, however, this does not affect the choice of a constant $u^0$ since $TV(\log(u^0x_{k-1})) = TV(\log(u^0)+\log(x_{k-1})) = TV(\log(x_{k-1}))$. However, $\|\log(u^0)\|_*$ is finite only if $\int_\Omega \log(u^0) = 0$ (see Sec. 3.1.4, Lemma 1 in \cite{vese2016variational}), which would force $u^0=1$ for a constant initialization, with no choice on optimizing the fidelity further. 
	
	We let  $w^0 = \log\left(\frac{1}{|\Omega|}\int\fd e^{-y_{k-1}}\right)$ for the SO MHDM regular and tight schemes, and  $w^0 = 0$ for the refined one.
	In the AA MHDM and AA-log MHDM regular and tight schemes, we take  $u^0 = \frac{1}{|\Omega|} \int {\fd}/{Tx_{k-1}}$, while for the refined scheme we need $u^0 = 1.$ Finally,  we choose $u^0 = \frac{\|\fd/Tx_{k-1}\|^2_{L^2(\Omega)}}{\int \fd/Tx_{k-1}}$ for the TNV-log schemes.
	
	In practice, we use the \emph{penalty minimizing} initializations for the regular and tight SO MHDM, AA MHDM, AA-log MHDM and TNV-log schemes, while the \emph{fidelity minimizing} initialization are considered for the refined SO MHDM and AA-log schemes. In the case of images with blur, we find that the \emph{penalty minimizing} initialization $u^0 = 1$ works marginally better for the refined AA-log MHDM recovery. 
	
	\captionsetup[subfigure]{font = scriptsize}
	\captionsetup{font = small, labelfont={bf}}
	\section{Numerical Results}\label{sec:numerical_results}
	In this section we examine the numerical results of multiscale image restoration in the case of multiplicative noise, using the regular, tight and refined schemes for the SO MHDM, AA MHDM, AA-log MHDM and TNV-log models. We also compare these against the TNV multiscale method  \cite{tad_nez_ves2} and the DZ model \cite{dong_2013}.
	\begin{figure}[h]
		\centering
		\begin{subfigure}{0.24\textwidth}
			\includegraphics[width=\linewidth]{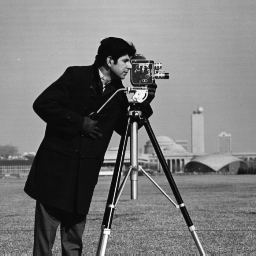}
			\caption{} \label{fig:cameraman}
		\end{subfigure}%
		\hspace*{\fill} 
		\begin{subfigure}{0.24\textwidth}
			\includegraphics[width=\linewidth]{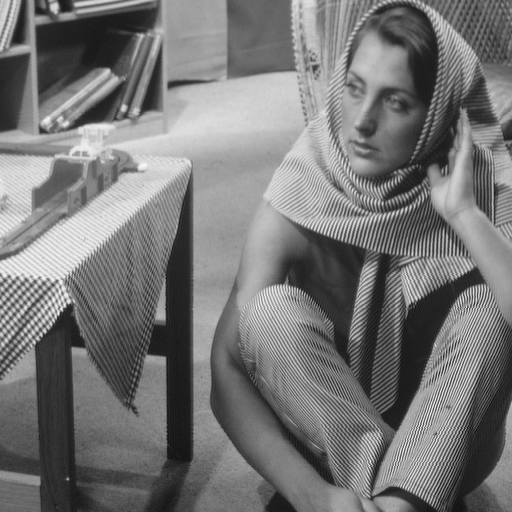}
			\caption{} \label{fig:barbara}
		\end{subfigure}%
		\hspace*{\fill}   
		\begin{subfigure}{0.24\textwidth}
			\includegraphics[width=\linewidth]{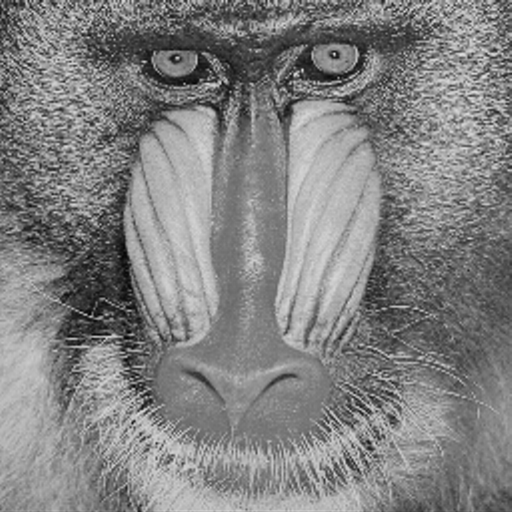}
			\caption{} \label{fig:mandril}
		\end{subfigure}%
		\hspace*{\fill}   
		\begin{subfigure}{0.24\textwidth}
			\includegraphics[width=\linewidth]{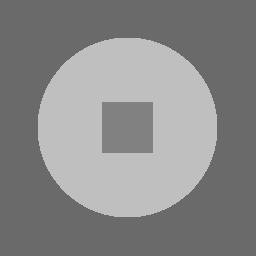}
			\caption{} \label{fig:disc_square}
		\end{subfigure}%
		\caption{Original images: (a) ``Cameraman", (b) ```Barbara", (c) ``Mandril", (d) ``Geometry".} \label{fig:original}
	\end{figure}%
	\captionsetup[subfigure]{labelformat=empty}
	\begin{figure}
		\centering
		\begin{subfigure}{0.24\textwidth}
			\includegraphics[width=\linewidth]{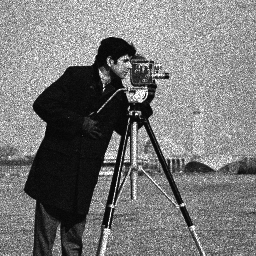}
			\caption{RMSE=27.00, SNR=13.98} \label{fig:cameraman_noisy}
		\end{subfigure}%
		\hspace*{\fill} 
		\begin{subfigure}{0.24\textwidth}
			\includegraphics[width=\linewidth]{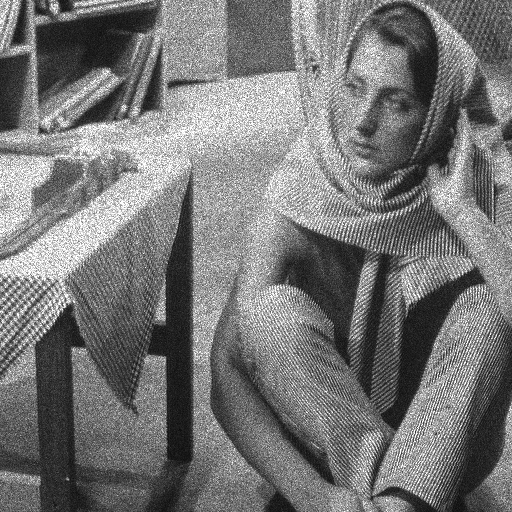}
			\subcaption{RMSE=26.05, SNR=13.99}\label{fig:barbara_noisy}
		\end{subfigure}%
		\hspace*{\fill}   
		\begin{subfigure}{0.24\textwidth}
			\includegraphics[width=\linewidth]{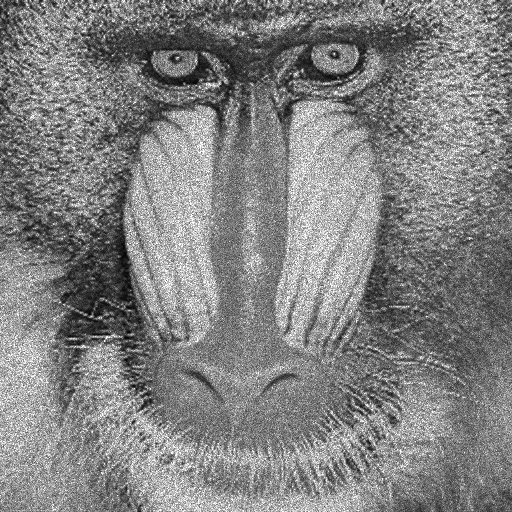}
			\caption{RMSE=27.12, SNR=13.97} \label{fig:mandril_noisy}
		\end{subfigure}%
		\hspace*{\fill}   
		\begin{subfigure}{0.24\textwidth}
			\includegraphics[width=\linewidth]{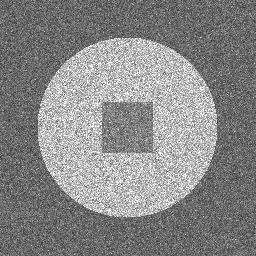}
			\caption{RMSE=28.34, SNR=12.96} \label{fig:disc_square_noisy}
		\end{subfigure}%
		\caption{Images with multiplicative gamma noise.} \label{fig:noisy}
	\end{figure}%
	We chose three natural grayscale images with edges and textures ranging from smooth to detailed, and one synthetic test image as shown in Fig.\,\ref{fig:original}. The models proposed in this work are aimed at removing multiplicative gamma noise and (possibly) blurring. Accordingly, we degrade the test images with gamma noise 
	$$ g(x;a) = \frac{a^a}{\Gamma(a)} x^{a-1} e^{-ax} \mathbbm{1}_{x\geq 0}\,,$$ with shape parameter $a=25$ (mean 1 and standard deviation $1/\sqrt{a}=0.2$), as shown in Fig.~\ref{fig:noisy}. Our choice of $a$ serves to compare with the original Aubert and Aujol paper \cite{auber_aujol_2008}. While the standard deviation for the Gamma noise used in \cite{auber_aujol_2008} is not given, this noise profile produces signal-to-noise ratios near the sample images used therein. Similar or lower noise levels are also used in \cite{ullah2017new}. Comprehensive comparisons are made at this noise level, but we also include some high and severe noise cases in Figures \ref{fig:AA_high_noise} and \ref{fig:mid_noise_compare}  which demonstrate the MHDM's ability to handle more aggressive corruption. 
	Restorations are evaluated on the root-mean-squared-error (RMSE) and signal-to-noise ratio (SNR) between recovered $x_{k}$ and original images $z$:  
	\begin{equation*}
	RMSE=\displaystyle\frac{\|x_{k}-z\|}{\sqrt{N}},~~
	SNR=10\times\log_{10}\left(\frac{\|z\|^2}{\|x_{k}-z\|^2}\right),
	\end{equation*}
	where $\|\cdot\|$ is the Euclidean norm and $N$ is the total number of pixels in the image. 
	We also examine how many multiscales are required for reconstruction, as well as  the effectiveness of the stopping criteria. 
	
	The MHDM recoveries in the following sections are all performed with the same parameter values (as much as the models allow). We choose $\lambda_0 = 0.01$, then  $\lambda_k=\lambda_0q^k$ with $q=2$ for the regular and $q=3$ for the tight/refined schemes, as they satisfy the convergence prerequisites given in Section \ref{conv_MHDM} and \ref{sec:extensions}.  For the tight and refined formulations, there is the additional parameter $ a_k=\frac{ a_0}{(1+k)^{3/2}}$ with $ a_0=1$. If blurring is considered, we use a $5\times 5$ Gaussian kernel $T$ with variance 2. The parameter $\lambda_k$ serves as the weight of the fidelity term, and larger $\lambda_k$ leads to more textural details in $u_k$. In general,  to avoid restoring noises in the image during the first few ranks of the hierarchy, a relatively small $\lambda_0$ should be chosen. As $\lambda_k$ gradually increases, the texture within the image is restored at finer and finer scales. The multiscale hierarchical nature of the methods also increases the robustness of the recoveries; that is to say, it dilutes the influence of parameter changes on the image restoration. For example, if we choose a smaller $\lambda_0$, an acceptable restoration can be obtained with more hierarchy ranks (larger $k$). The time step size $\Delta t=0.01$, gradient regularization $\epsilon=0.01$, and the maximum number of iterations \texttt{maxIters}=1000. Using SO MHDM (EL) as an example, we give Fig.\,\ref{fig:energy_k} to briefly mention the numerical convergence of gradient descent across multiscale indices $k$ for \texttt{maxIters}=1000, where $E_k(w^n)$ is given by the energy to be minimized in \eqref{SO_MHDM} evaluated at the $n$-th iteration $w^n$. Lower $k$ values may take more than $1000$ steps to fully converge (c.f. $k=6$ in Fig.\,\ref{fig:energy_k}), however, in aggregate across all $k$, convergence is achieved. Increasing \texttt{maxIters} or implementing a relative-energy-change exit condition has little impact on RMSE or SNR values at $k_{min}$, as shown in Table\,\ref{fig:energy_k}. We also observe similar behavior for the AA-log MHDM methods. Overall, the schemes exhibit robust nature due to the hierarchical construction of multiscale methods.
	
	\begin{figure}[h!]
		\centering
		\includegraphics[width=0.5\linewidth]{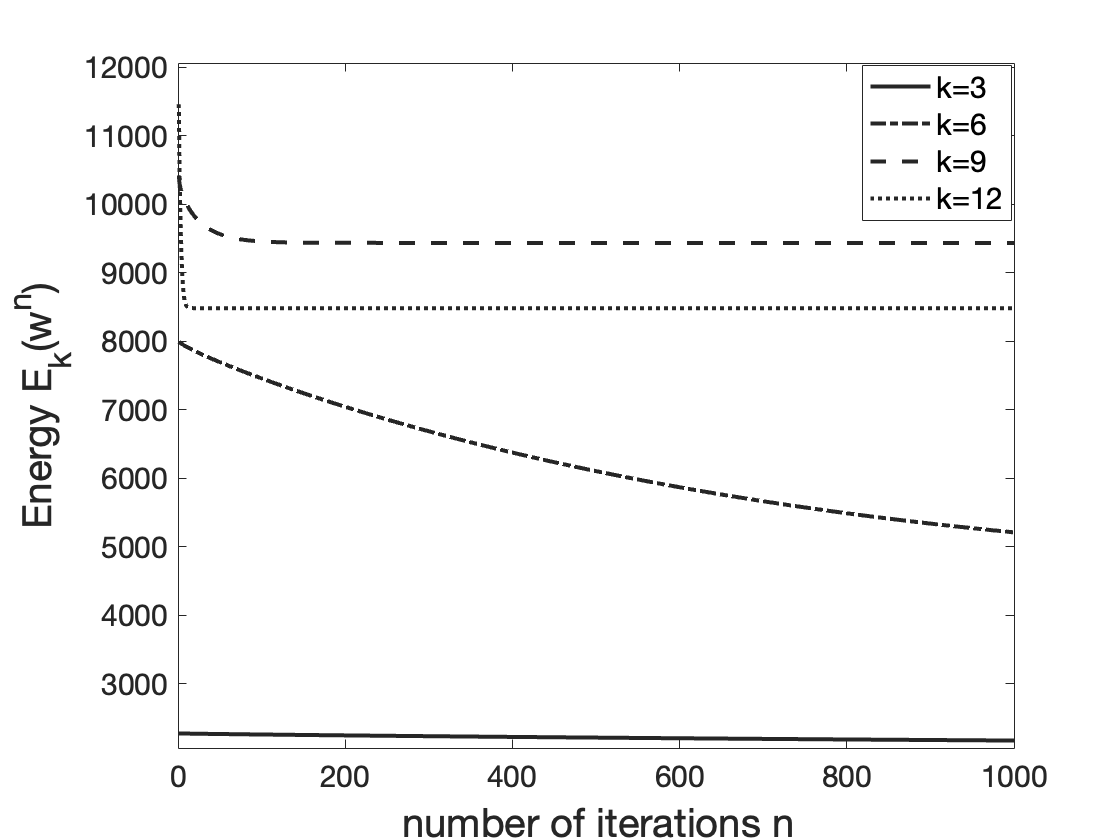}
		\caption{SO MHDM (cameraman image) energy versus iteration number during gradient descent, with multiscale numbers $k=3, 6, 9, 12$.}\label{fig:energy_k}
	\end{figure}%
	\begin{table}[h]
		\centering
		\caption{{$k_{min}$, RMSE and SNR values of SO MHDM (cameraman image) with different stopping criteria for gradient descent.}}
		\begin{tabular}{c|ccc}
			Stopping criteria & $k_{min}$ & RMSE at $k_{min}$  & SNR at $k_{min}$ \\\hline
			\texttt{maxIters}=100 & 9& 12.0223 & 21.0058 \\
			\texttt{maxIters}=1000 & 9 & 10.8598 &21.8892\\
			\texttt{maxIters}=10000 & 9 & 10.9102  & 21.8489 \\\hline
			\rule{0pt}{4ex}
			$\displaystyle\frac{|E_k(w^{n+1})-E_k(w^n)|}{E_k(w^{n+1})}<10^{-8}$& 9 & 10.9159 & 21.8444 
		\end{tabular}%
		\label{tab:Compare_maxiter}%
	\end{table}%
	
	\subsection{Shi-Osher models}
	Here and in the subsequent sections, to differentiate between the Euler-Lagrange PDE discretization and the ADMM recoveries, we will append (EL) or (ADMM) to the appropriate schemes, such as tight SO MHDM (ADMM) and refined SO MHDM (EL).
	
	The purpose of the MHDM recovery is to retain textures of the original images at different scales while eliminating noise. As an example of a typical recovery, we show the progression of the multiscales for the ``Cameraman" image in Fig.\,\ref{fig:cameraman_multiscales}. The importance of correctly selecting a stopping point is clear, since too many multiscales will recover the majority of the noise as finer and finer levels of texture are added back into the image, while too few leave the image without textural details. This can be seen visually in Fig.\,\ref{fig:cameraman_multiscales} and numerically in Fig.\,\ref{fig:combined_metrics}. The individual multiplicative scales are given in Fig.\,\ref{fig:cameraman_uk} and demonstrate how the images are built up. Recall each $x_k$ in Fig.\,\ref{fig:cameraman_multiscales} is the product $\prod_{j=0}^k u_j$, with each $u_j$ contributing features at different scales. Based on the multiplicative construction, light regions within $u_j$ (higher pixel values) are promoted, while darker ones (lower pixel values) are suppressed relative to middle-toned regions in the image. As one can see in  Fig.\,\ref{fig:cameraman_uk},  the multiplicative process segments the large scale ``cartoon" features at early scales before separating the texture and eventually noise within the image. Moreover, the proposed stopping index $k^*(\delta)=9$ aligns with a sensible contribution from $u_9$ which adds back sufficient detail within the cameraman's clothing and tripod before increasing the fine details---and noise---in the grass and background with $u_{j},\, j>9$. For a user with definite sense of the scale of features in the true data, the $u_j$ pieces could provide a visual method of choosing $k^*$, whereby one increases $k$ until $u_k$ is emphasizing details at the preferred size.
	
	Throughout the following discussion, we will look at both the multiscale index $k_{min}$ which minimizes RMSE and the proposed stopping index $k^*=k^*(\delta)$ defined in \eqref{discrepancy_index}, \eqref{discrepancy_tight}. 
	
	Recall that the noise parameter $\delta$ satisfies $H(f^\delta,Tz)\leq \delta^2$. In practice, we take $\delta^2 = H(f^\delta,Tz)$ for our numerical experiments. 
	
	In the case of SO MHDM, we transform the summed-MHDM reconstruction $y_k$ back to the image approximation $x_k = e^{y_k}$, and then we  compute $k^*$ accordingly. The stopping criterion is determined by choosing $k^*$ to be the penultimate multiscale to $H(f^\delta,x_{k})/H(f^\delta,z)$ dropping below the value 1, as shown in Fig.\,\ref{fig:combined_stopping} for the SO MHDM regular, tight and refined recoveries.%
	\captionsetup[subfigure]{labelformat=simple}
	\renewcommand\thesubfigure{(\alph{subfigure})}
	\begin{figure}[h!]
		\centering
		\includegraphics[width=0.75\linewidth]{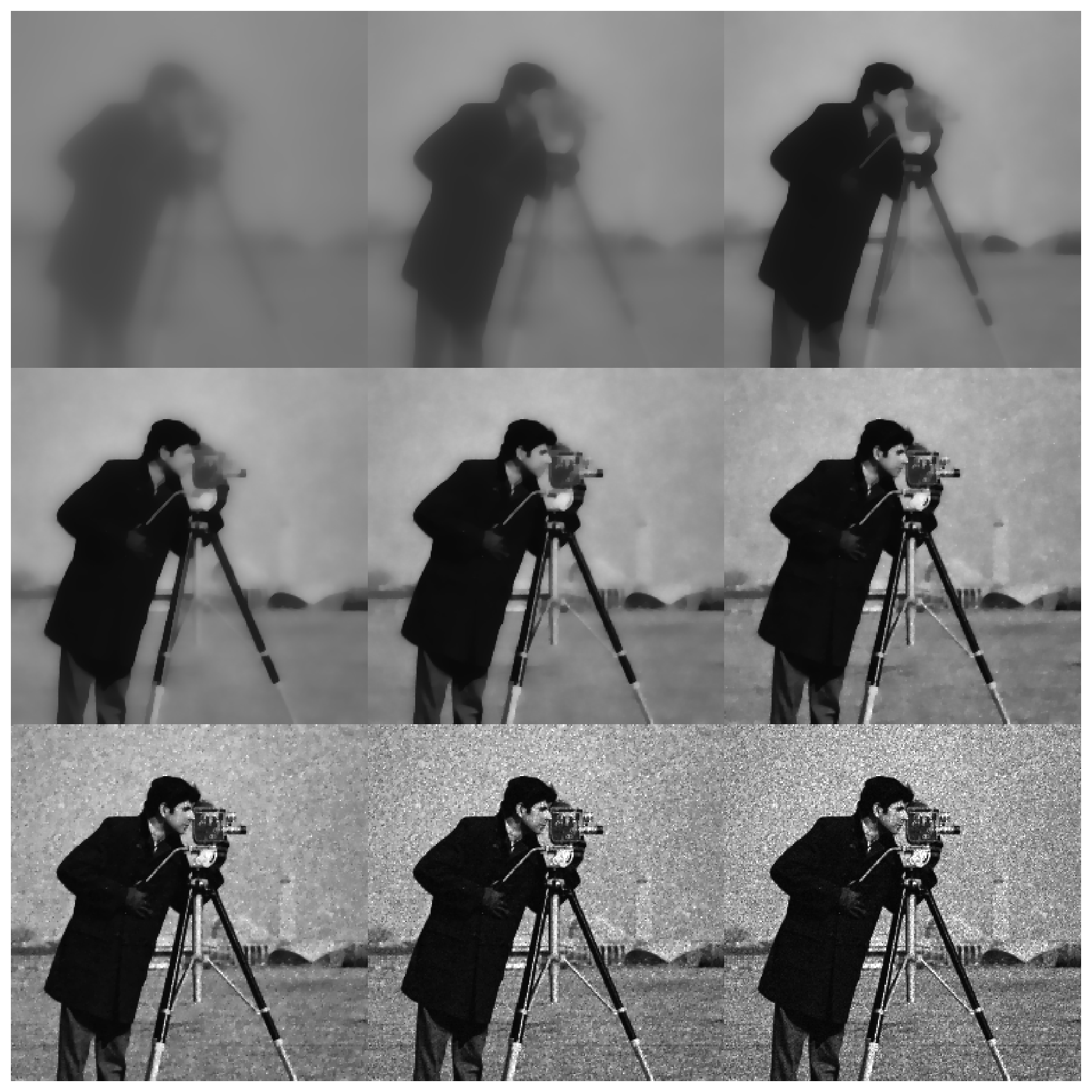}
		\caption{Multiscales $x_k$ for $k=4,5,\dots,12$ in the SO MHDM recovery of the cameraman image. In this example, our proposed stopping index $k^*(\delta)=9=k_{min}$ is optimal (Refer to the diamond-labeled curves in Fig.\,\ref{fig:metrics_SO_MHDM_combined}).} \label{fig:cameraman_multiscales}
	\end{figure}%
	\begin{figure}[h!]
		\centering
		\includegraphics[width=0.6\linewidth]{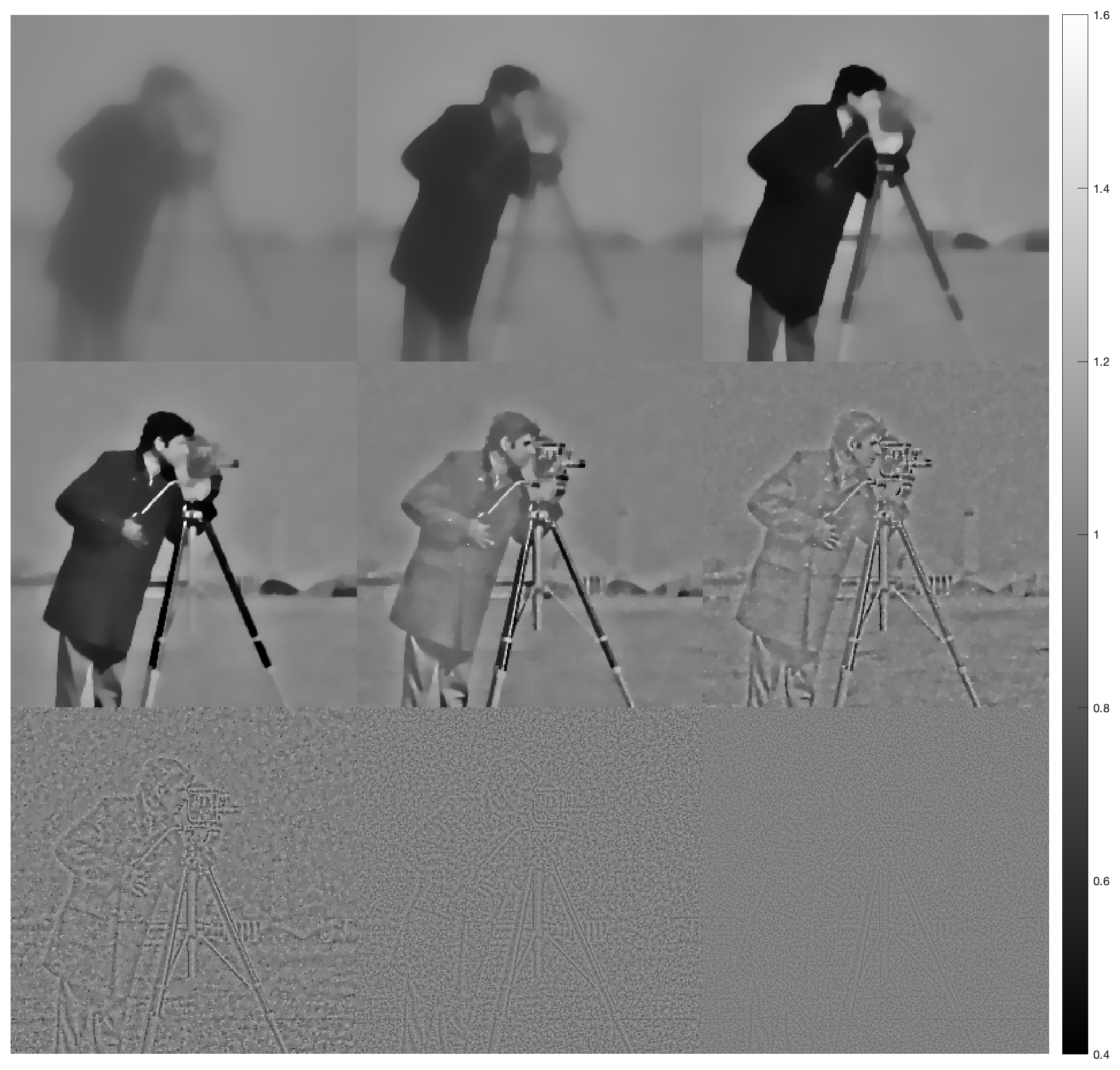}
		\caption{ Multiplicative scales $u_j=e^{y_j}$ from the SO MHDM recovery of the cameraman image. The product $\prod_{j=0}^k u_j=:x_k$ constructs the multiscales shown in Fig.\,\ref{fig:cameraman_multiscales}. These are displayed on the interval $[0.4,1.6]$ for increased contrast---true range $[0.22,1.63]$---and shown with a common legend.} 
		\label{fig:cameraman_uk}
	\end{figure}%
	\begin{figure}[h!]
		\centering
		\begin{subfigure}{0.60\textwidth}
			\includegraphics[width=0.99\linewidth]{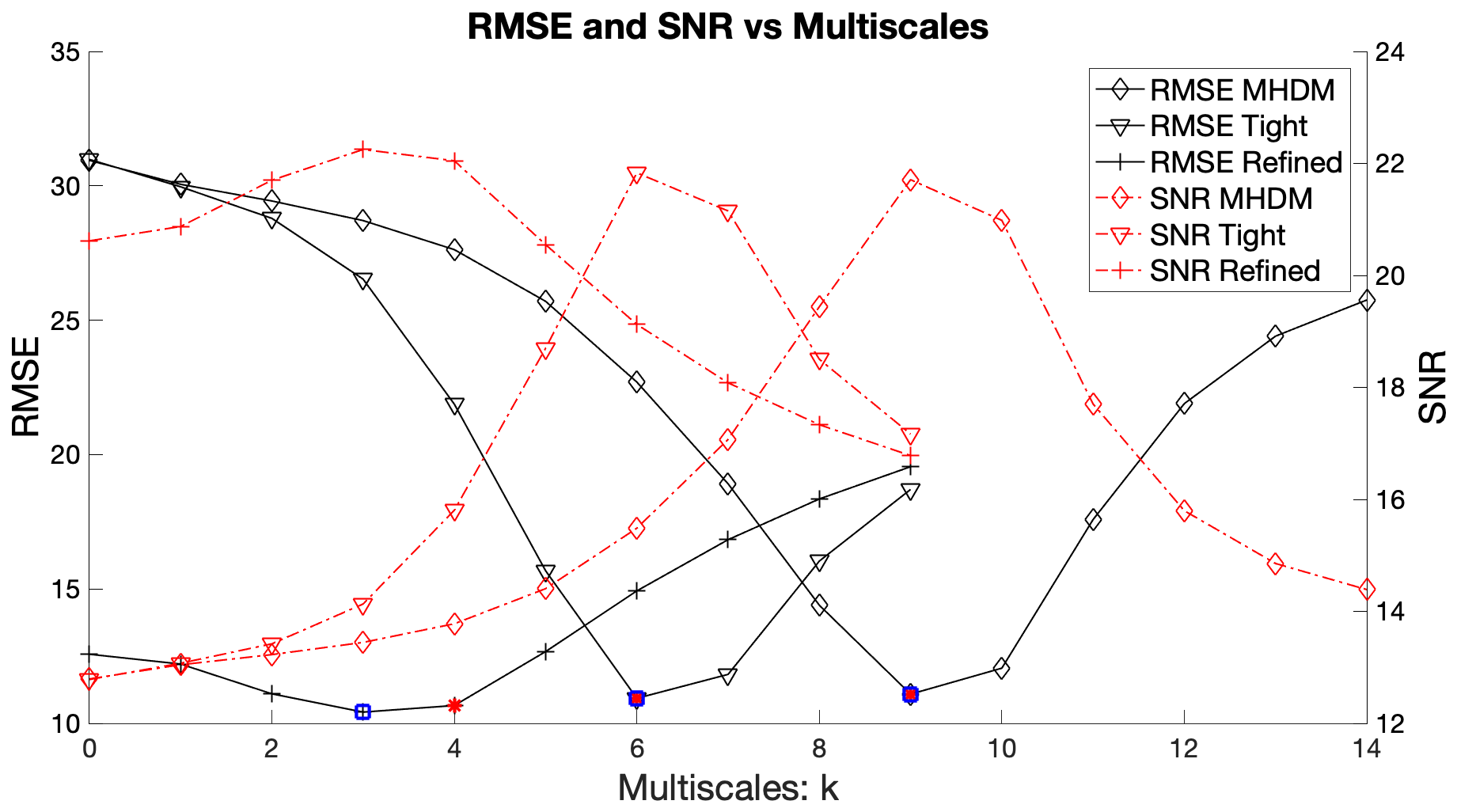}
			\caption{}\label{fig:combined_metrics}
		\end{subfigure}
		\begin{subfigure}{0.35\textwidth}
			\includegraphics[width=0.99\linewidth]{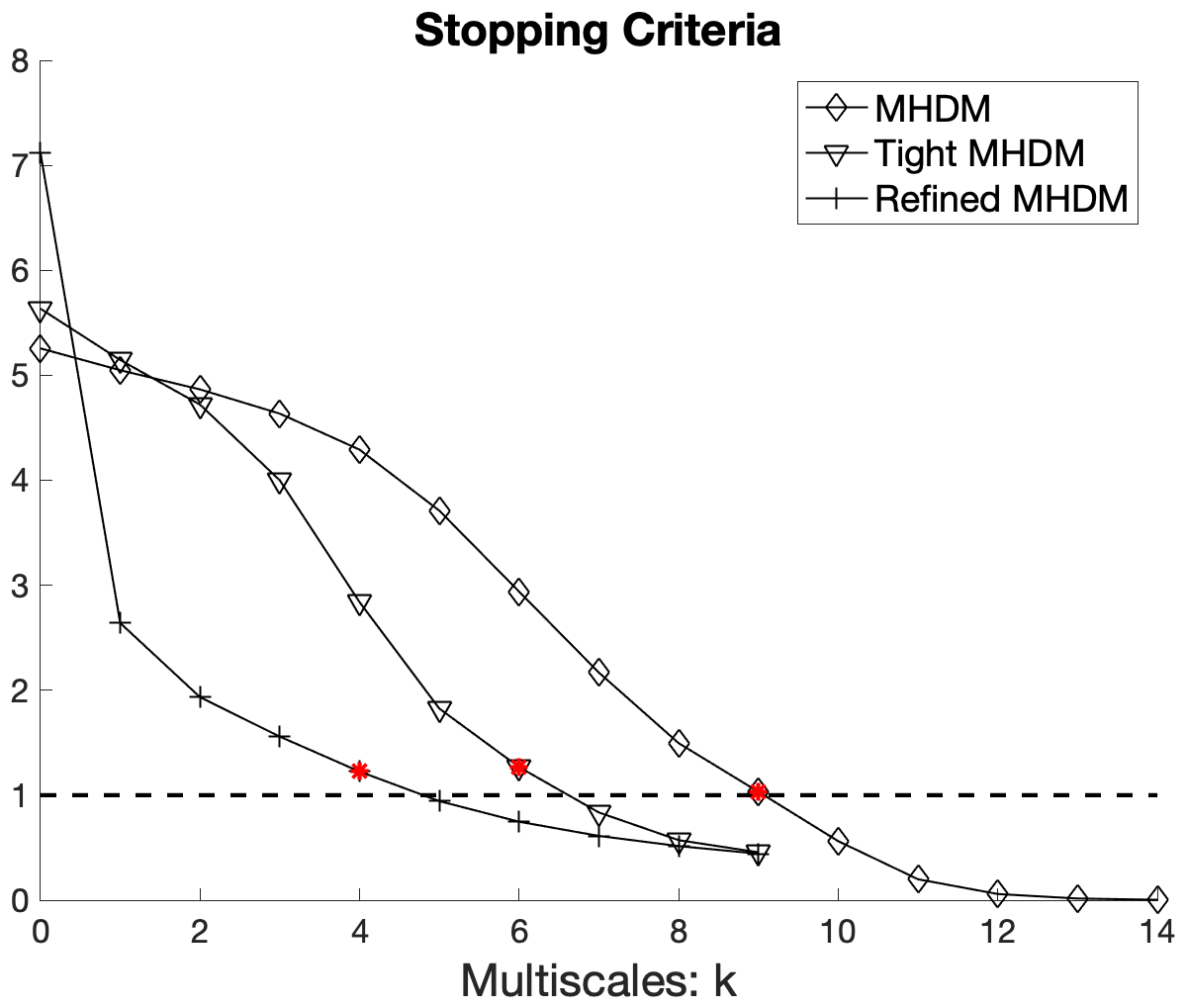}
			\caption{}\label{fig:combined_stopping}
		\end{subfigure}%
		\caption{(a) RMSE and SNR versus multiscales index across methods when restoring the ``cameraman" image. The optimal multiscale index $k_{min}$ for each method is shown as the blue square and the stopping criteria $k^*$ are given by red asterisks. (b) The stopping criteria $k^*$ (shown as red asterisks) are the maximal $k$ before $H(f^\delta,x_{k})/H(f^\delta,z)\geq \tau>1$ (for SO MHDM) or $(H(f^\delta,x_{k})+ a_k\lambda_kJ(x_k))/H(f^\delta,z)\geq \tau>1$ (for tight and refined SO MHDM) is no longer satisfied, as indicated by crossing under the horizontal dotted line in (b).} \label{fig:metrics_SO_MHDM_combined}
	\end{figure}%
	\captionsetup[subfigure]{labelformat=empty}
	\begin{figure}[h!]
		\centering
		\begin{subfigure}{0.24\textwidth}
			\includegraphics[width=\linewidth]{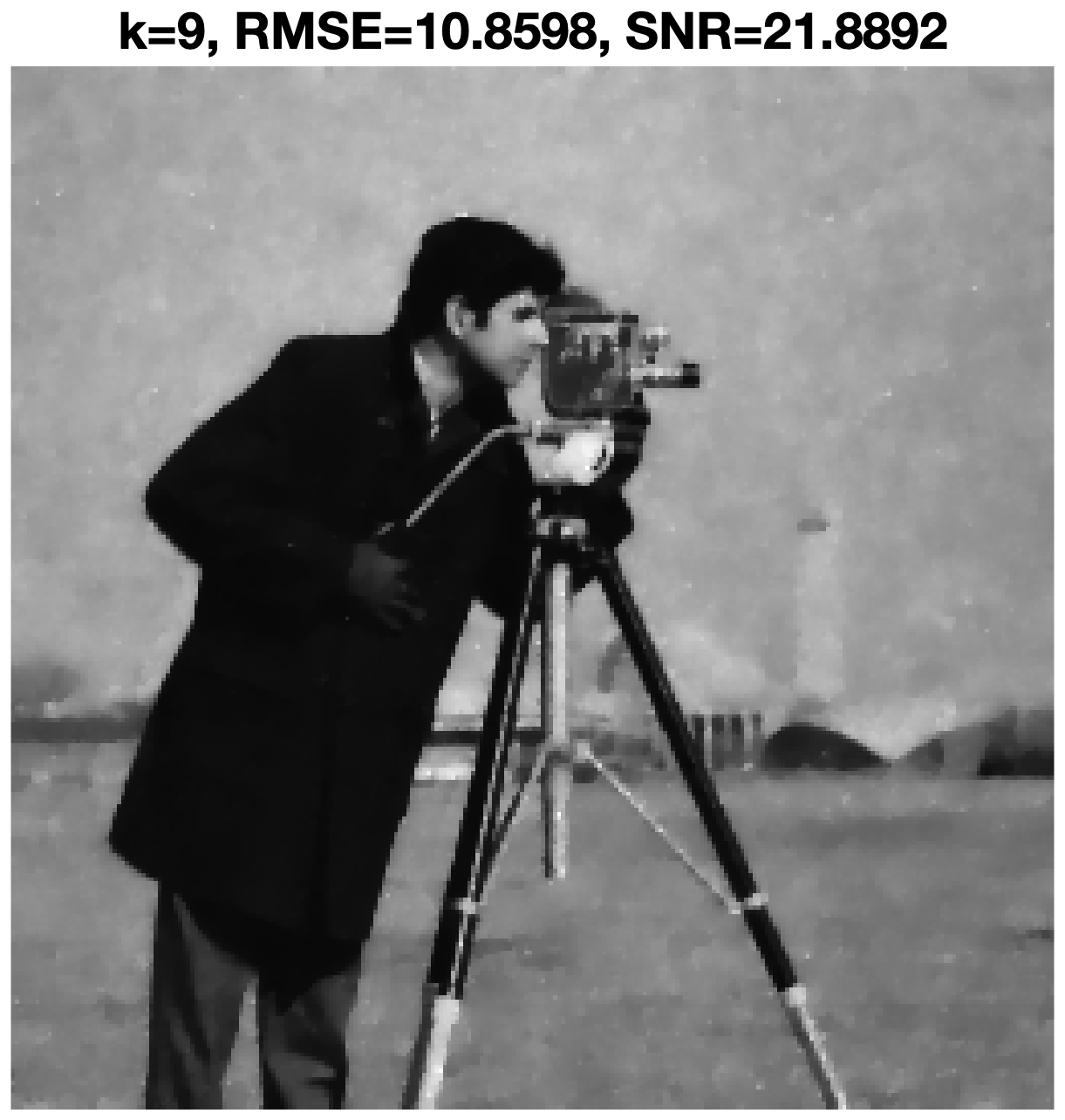}
		\end{subfigure}%
		\hspace*{\fill}
		\begin{subfigure}{0.24\textwidth}
			\includegraphics[width=\linewidth]{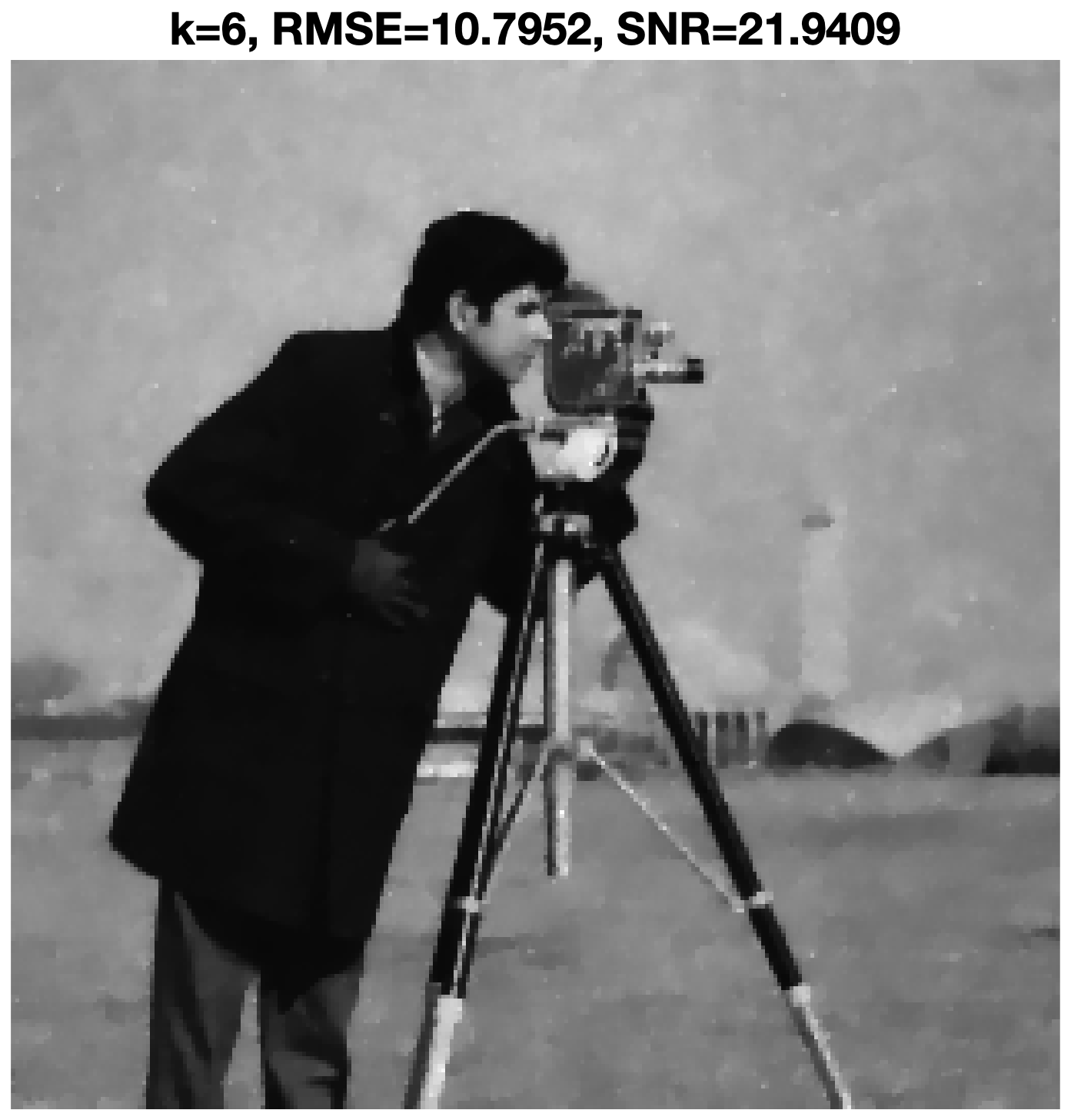}
		\end{subfigure}%
		\hspace*{\fill}
		\begin{subfigure}{0.24\textwidth}
			\includegraphics[width=\linewidth]{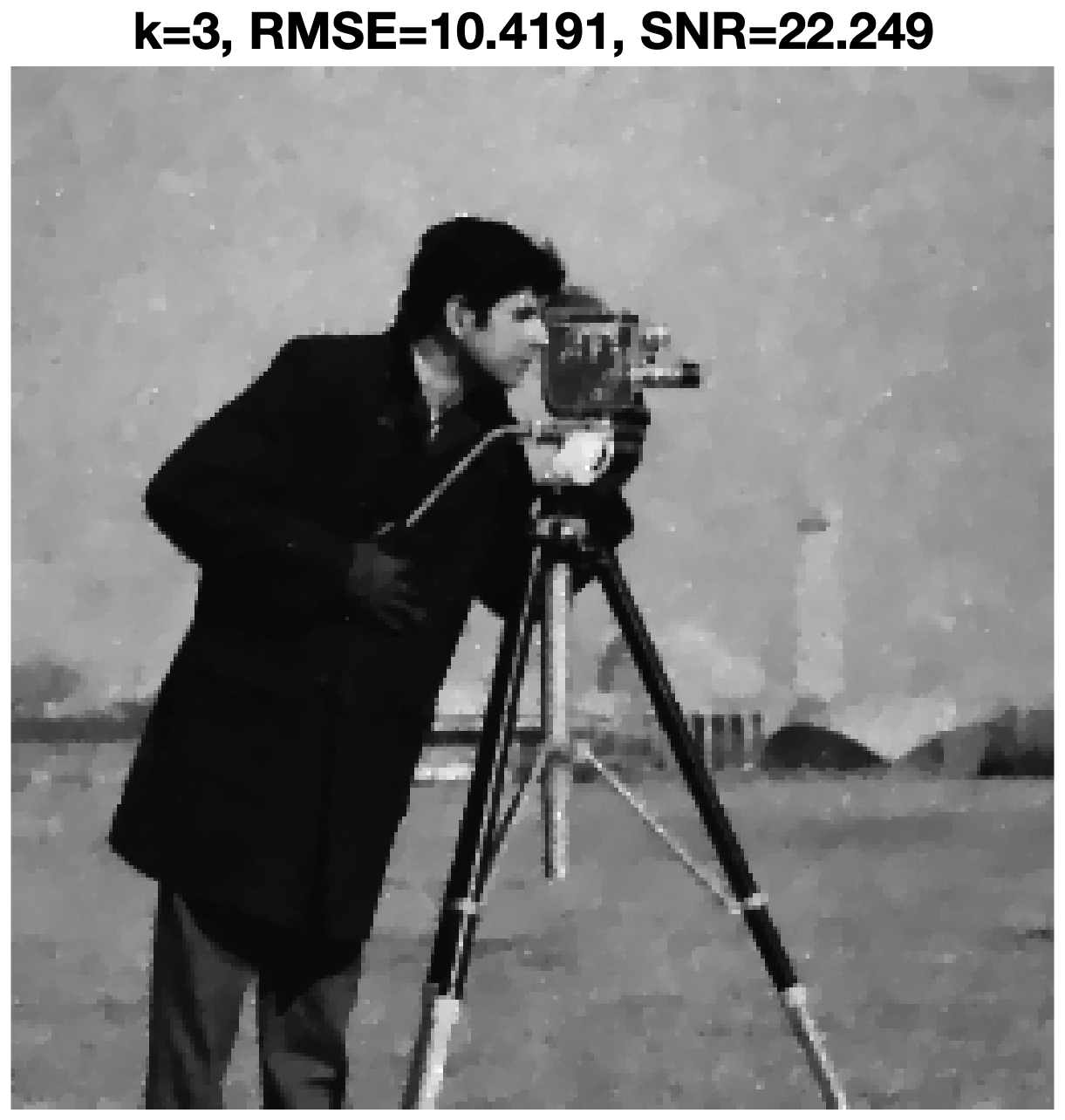}
		\end{subfigure}%
		\hspace*{\fill}
		\begin{subfigure}{0.24\textwidth}
			\includegraphics[width=\linewidth]{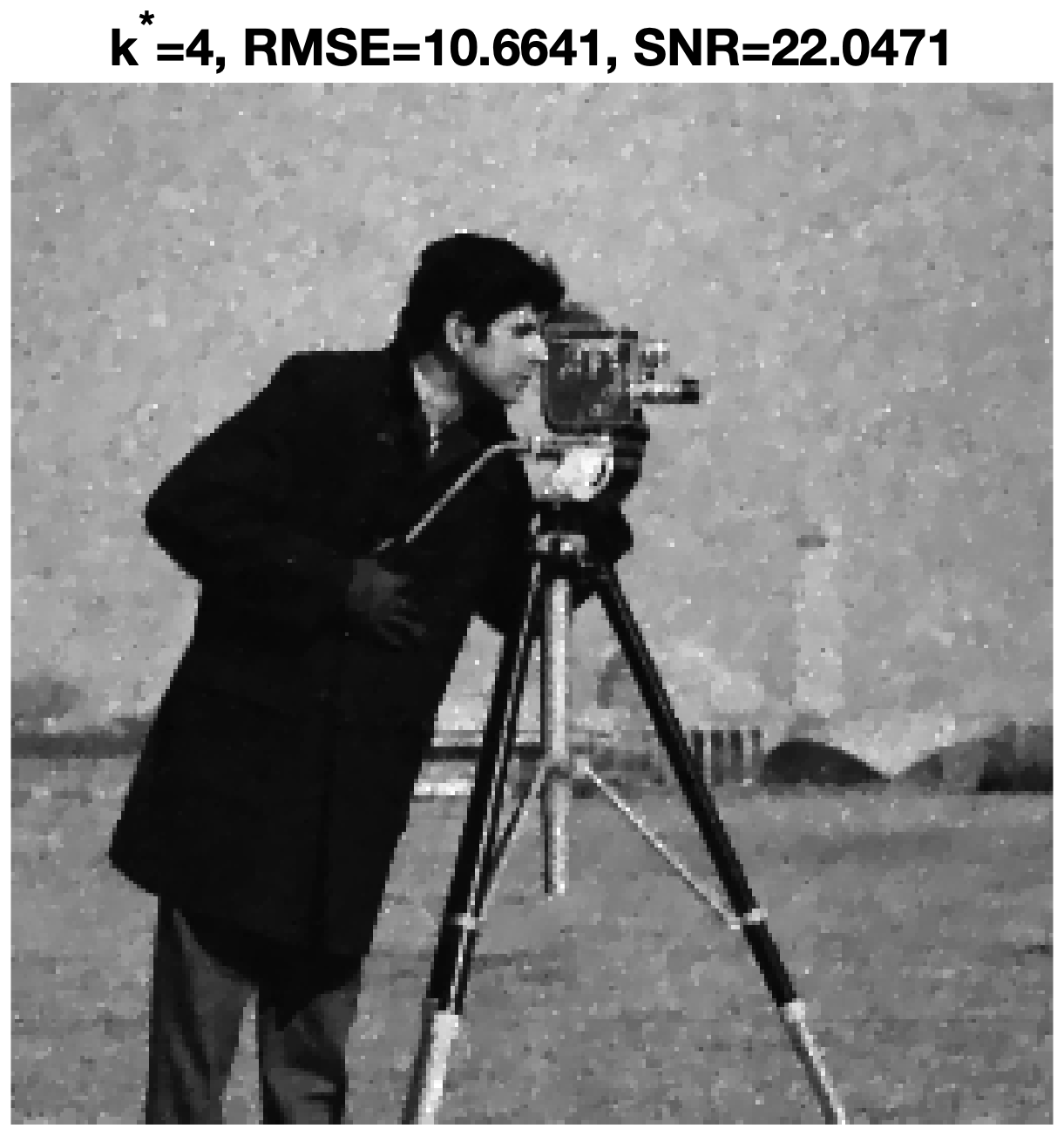}
		\end{subfigure}%
		\caption{SO MHDM (EL) image recoveries. From left to right: regular, tight, and refined restorations at $k_{min}$, and finally, the refined restoration at $k^*$. Recall that $k^*=k_{min}$ for regular and tight recoveries.}
		\label{fig:SOEL_cameraman}
	\end{figure}%
	In the ``Cameraman" restoration, one has $k^*=k_{min}$ for regular and tight SO MHDM, while $k^*(\delta)=k_{min}+1$ for the refined version. We emphasize that the tight and refined schemes require fewer multiscales to recover the images, and that Fig.~\ref{fig:combined_metrics} demonstrates the importance of the stopping criteria at preventing excess noise from being recovered. The SO MHDM (EL) recoveries are shown in Fig.\,\ref{fig:SOEL_cameraman}, where we observe slight improvements transitioning from the regular to tight and finally to refined methods. 
	%
	%
	\captionsetup[subfigure]{labelformat=empty}
	\begin{figure}
		\centering
		\begin{subfigure}{0.24\textwidth}
			\includegraphics[width=\linewidth]{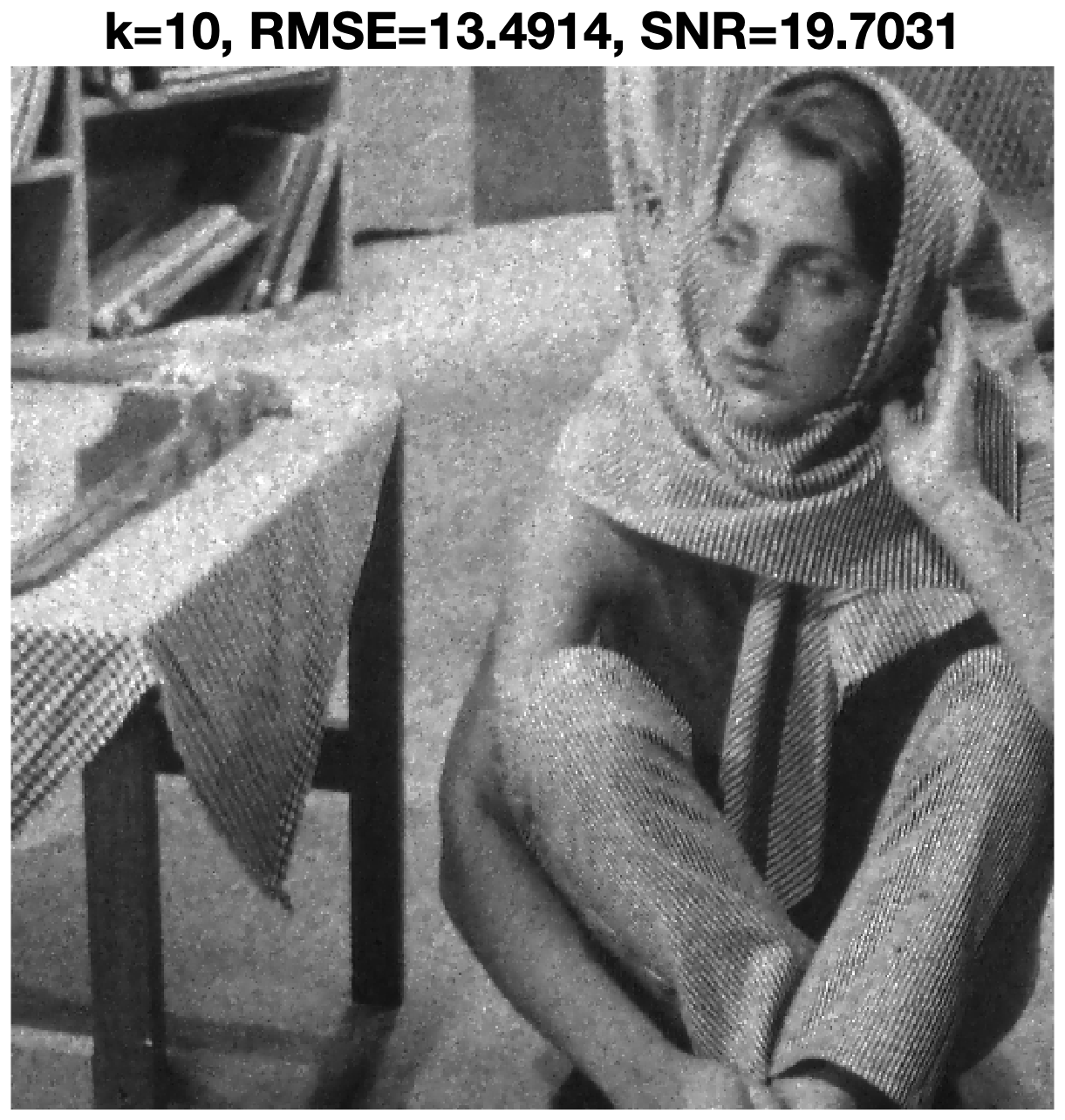}
		\end{subfigure}\hspace{2.3mm}%
		\begin{subfigure}{0.24\textwidth}
			\includegraphics[width=\linewidth]{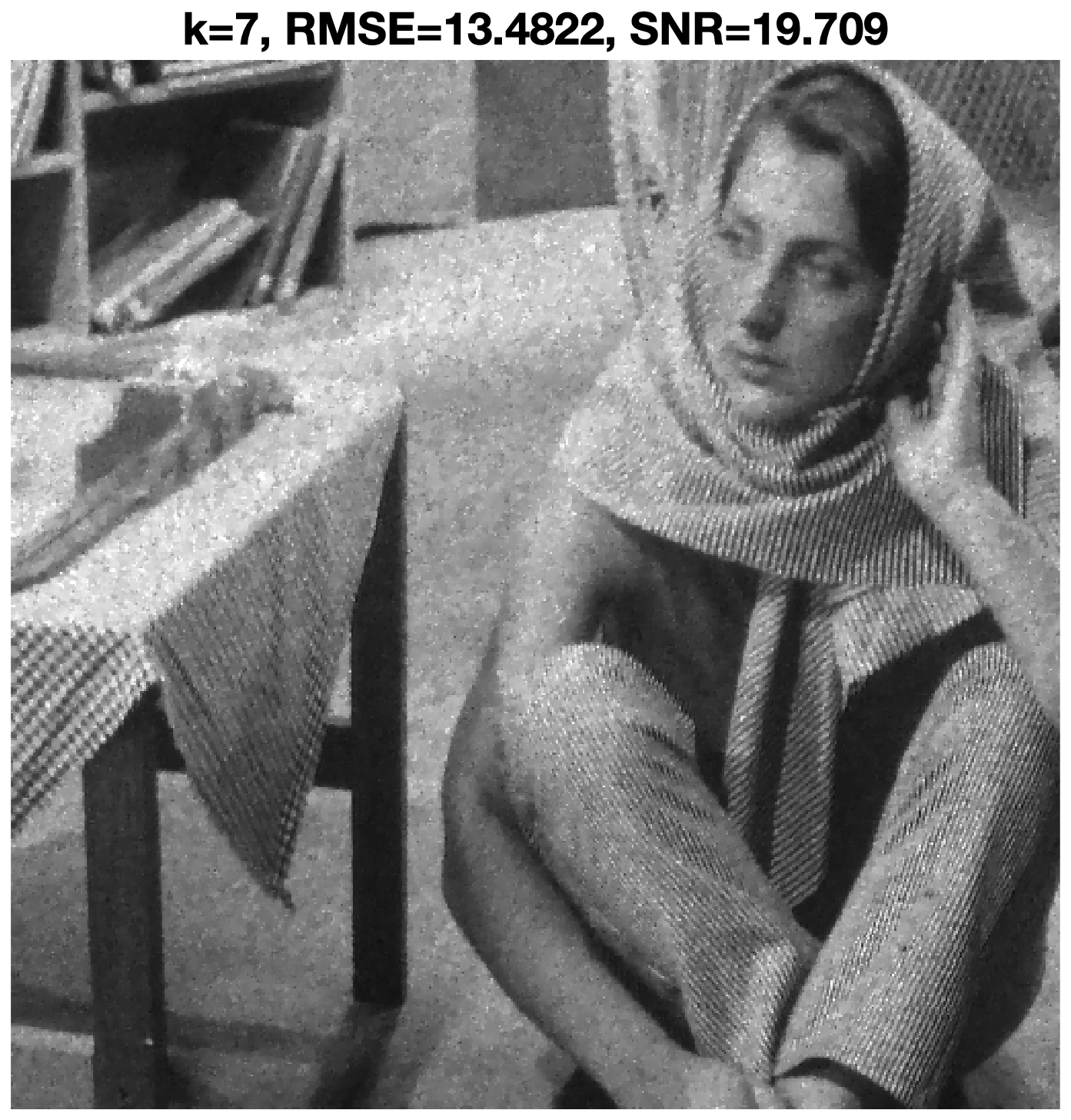}
		\end{subfigure}\hspace{2.3mm}%
		\begin{subfigure}{0.24\textwidth}
			\includegraphics[width=\linewidth]{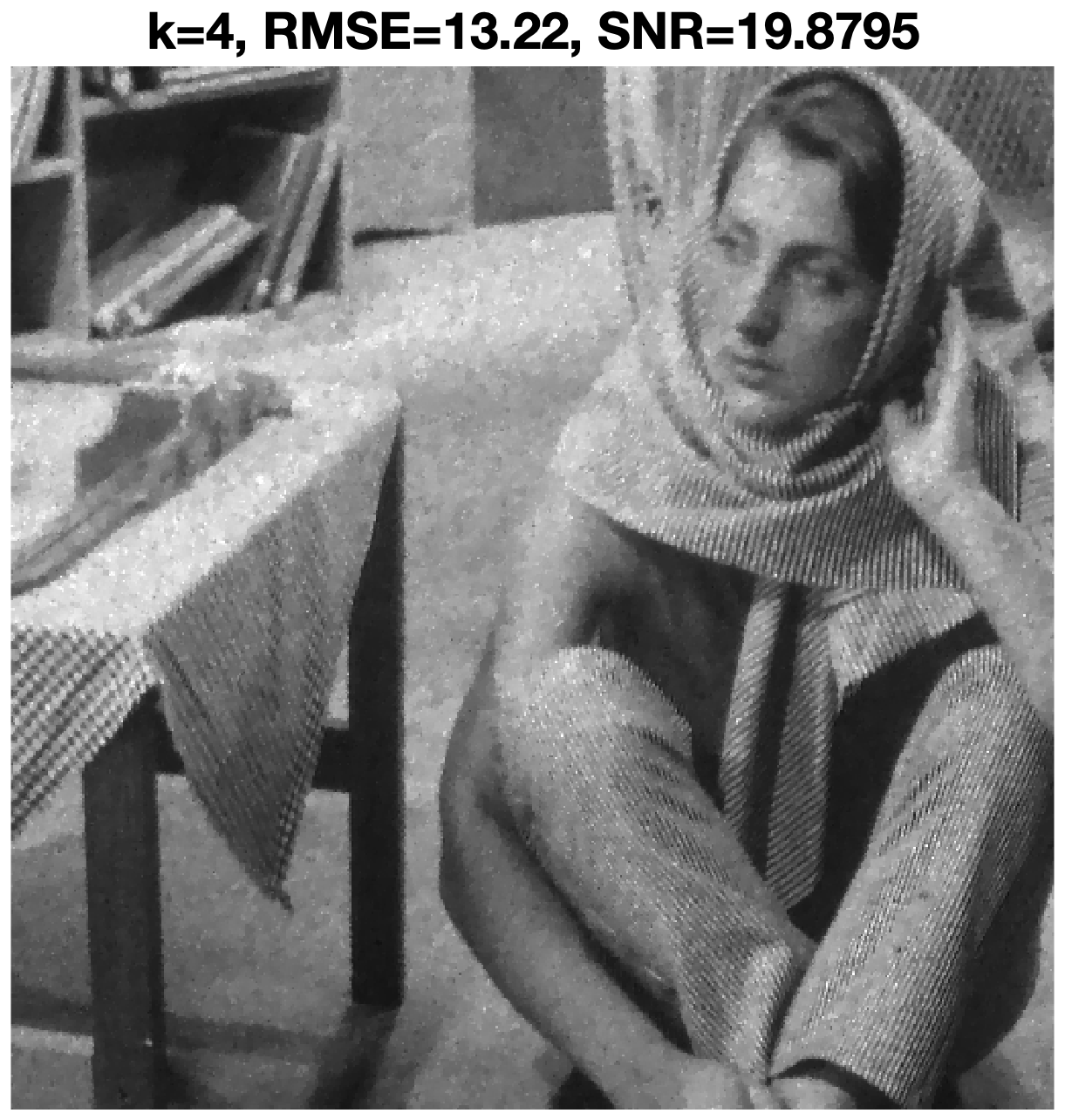}
		\end{subfigure}\\%
		\begin{subfigure}{0.24\textwidth}
			\includegraphics[width=\linewidth]{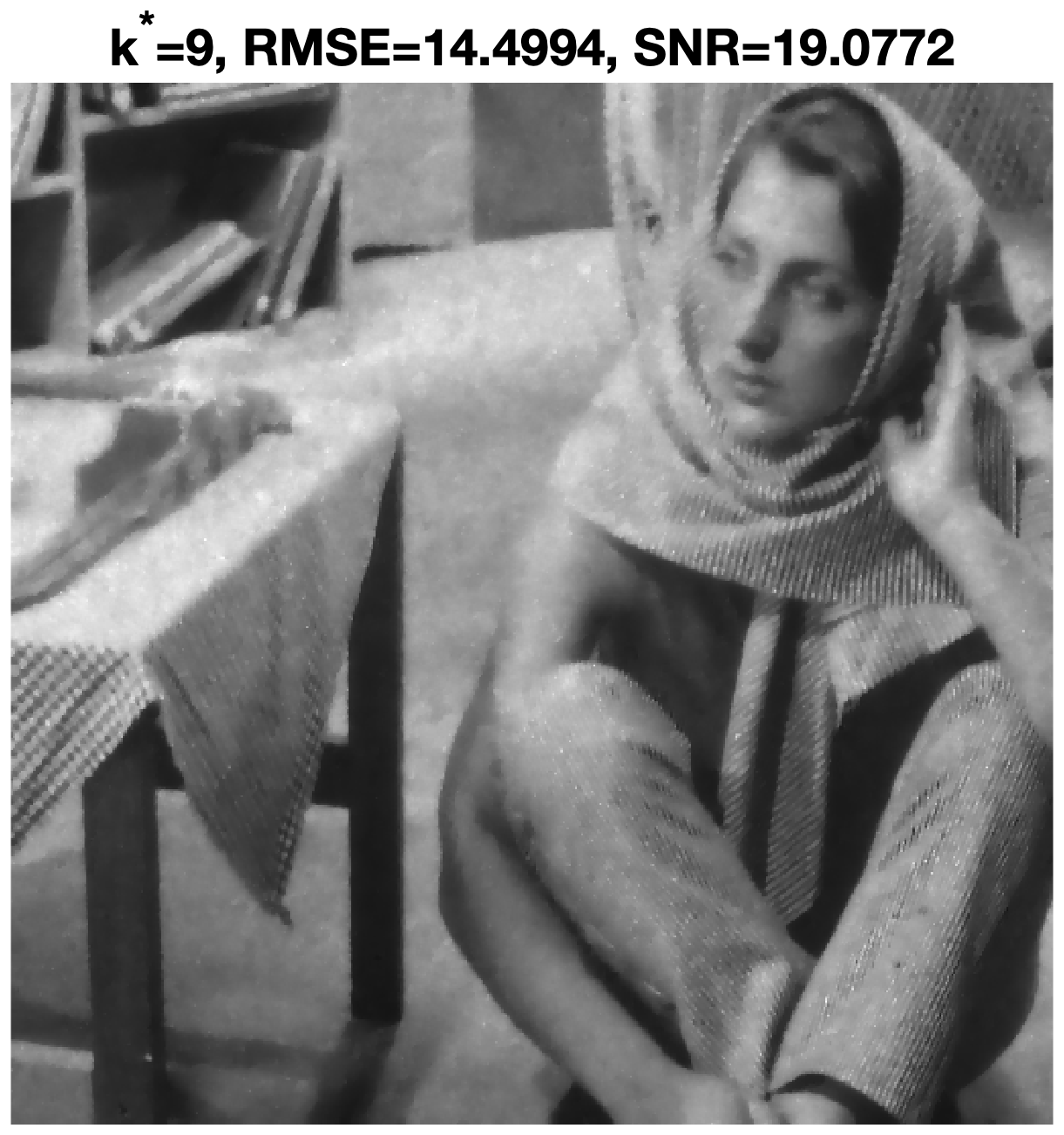}
		\end{subfigure}\hspace{2.3mm}%
		\begin{subfigure}{0.24\textwidth}
			\includegraphics[width=\linewidth]{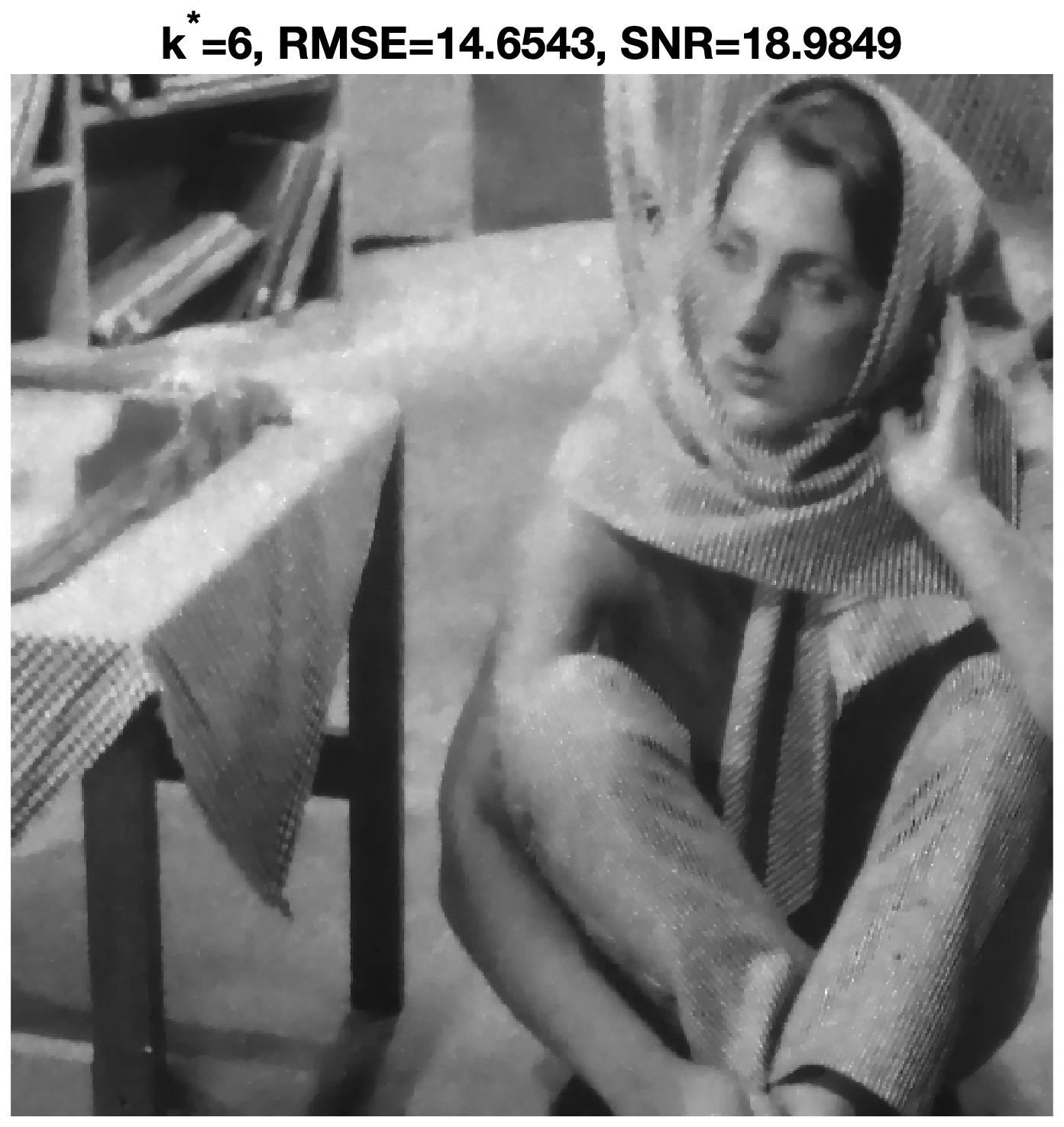}
		\end{subfigure}\hspace{2.3mm}%
		\begin{subfigure}{0.24\textwidth}
			\includegraphics[width=\linewidth]{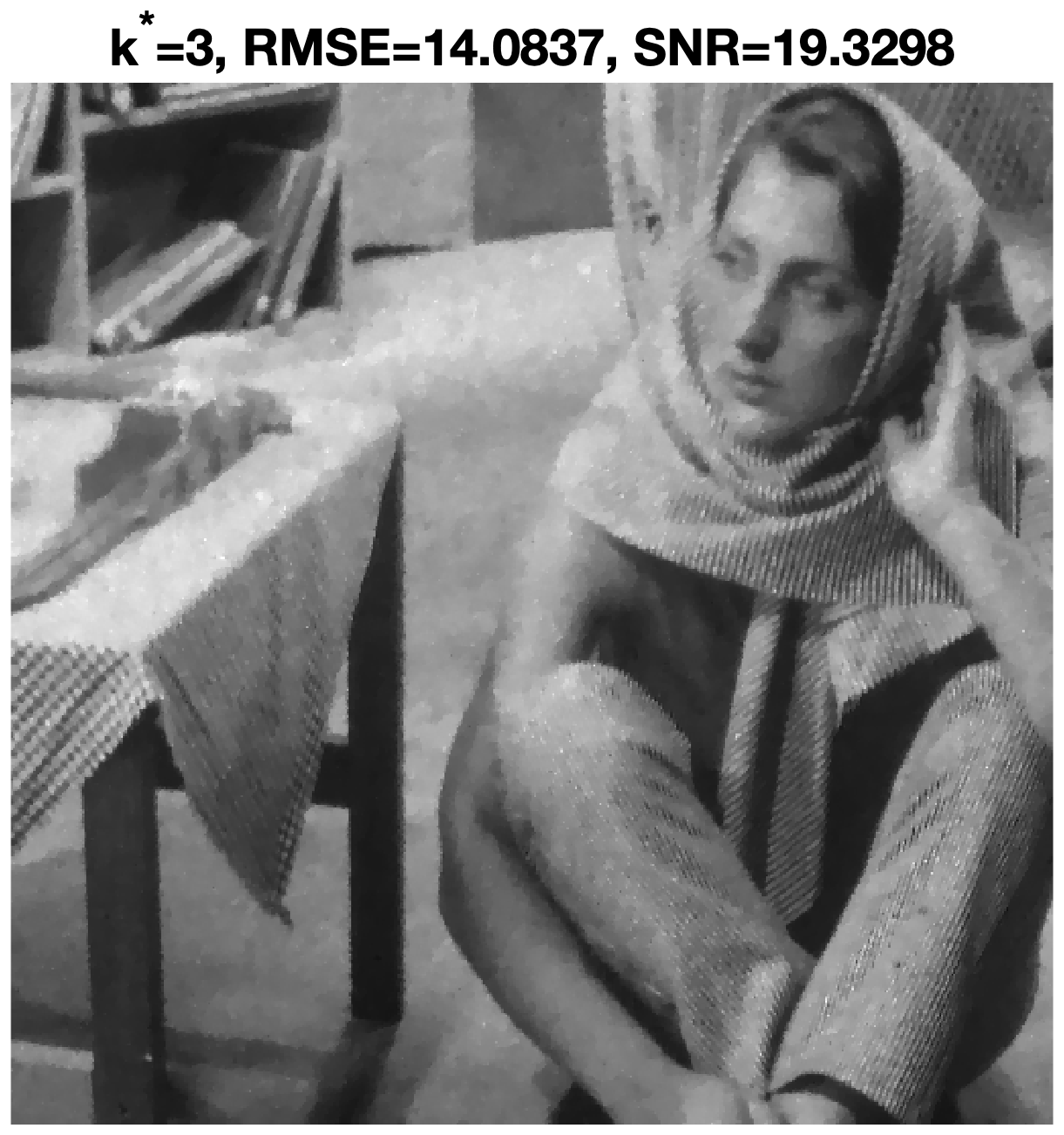}
		\end{subfigure}%
		\caption{From left to right: the regular, tight and refined SO MHDM (EL) recoveries. Row one is the $k_{min}$ restoration while row two is the $k^*$ recovery (if $k^*\neq k_{min}$).}
		\label{fig:restored_osher}
	\end{figure}%
	Fig.\,\ref{fig:restored_osher} exhibits a more textured image at the $k_{min}$ and $k^*(\delta)$ scales. Again, the number of multiscales required in the SO MHDM (EL) reconstructions decreases going from regular to tight and then to refined SO MHDM in a uniform manner, consistent with Fig.\,\ref{fig:metrics_SO_MHDM_combined}. We also see increasingly improved restorations when moving from regular to refined schemes. Notably, with this more textured image the stopping index $k^*(\delta)$---which is generally within one step of $k_{min}$---produces restorations which are visually very close to the optimal ones. 
	
	However, the difference between $k_{min}$ and $k^*(\delta)$ can significantly affect the restoration in some cases---see ``Geometry" in second column of Fig.\,\ref{fig:SO_EL_ADMM_restored}, which compares recoveries obtained from the SO MHDM ADMM and EL approaches. For this low texture image, the SO MHDM (ADMM) restorations are significantly better than the EL derived counterparts. We also take a moment to mention that the SO MHDM recoveries in Fig.\,\ref{fig:SO_EL_ADMM_restored} preserve the mean image intensity. For models that make use of the logarithm to transform multiplicative noise into additive noise, there is a downward shift in the mean intensity of the recoveries, as explained in \cite{auber_aujol_2008}. Our multiscale method effectively eliminates this mean intensity shift.
	\begin{figure}
		\begin{subfigure}{0.24\textwidth}
			\includegraphics[width=\linewidth]{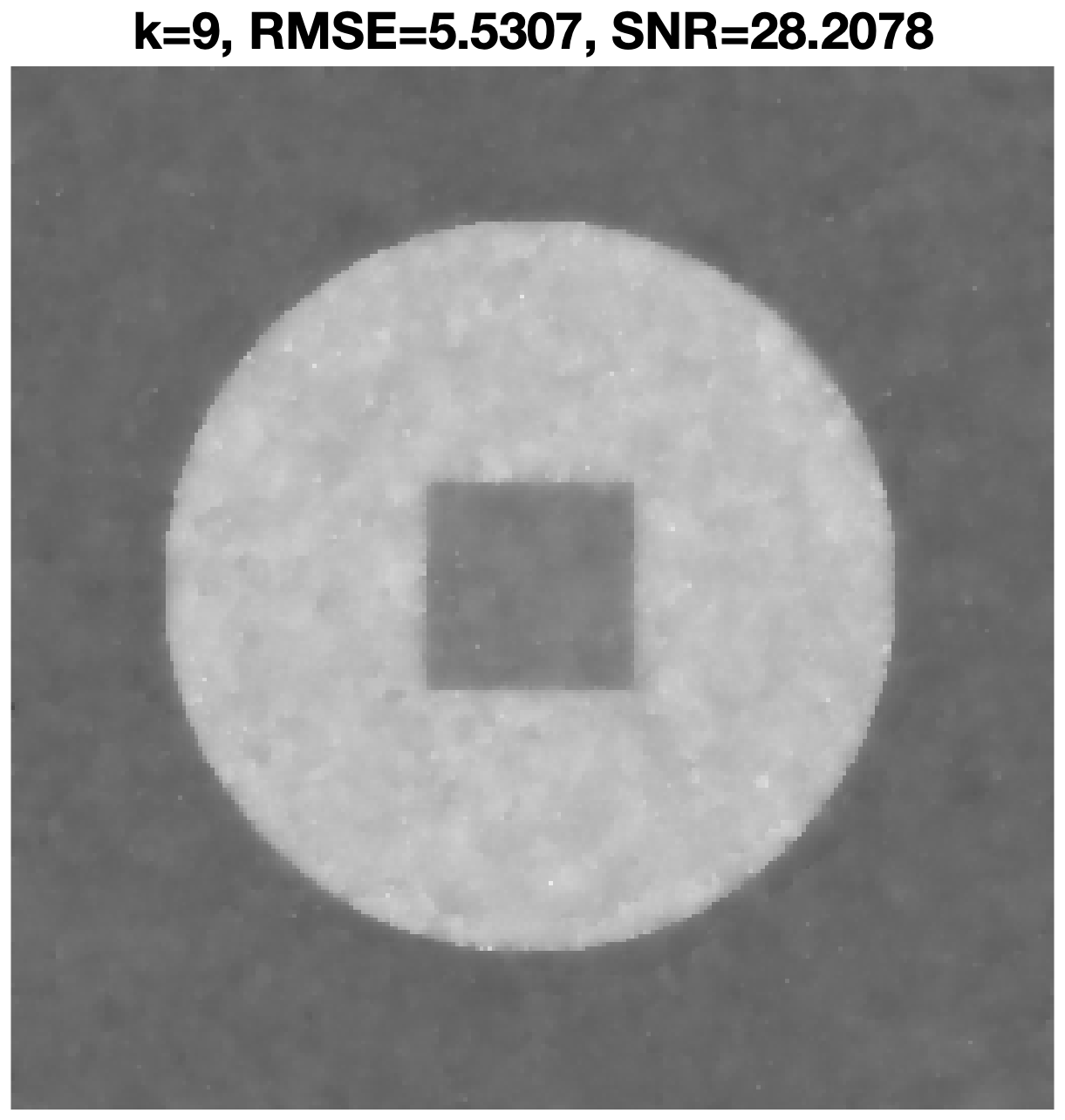}
		\end{subfigure}%
		\begin{subfigure}{0.24\textwidth}
			\includegraphics[width=\linewidth]{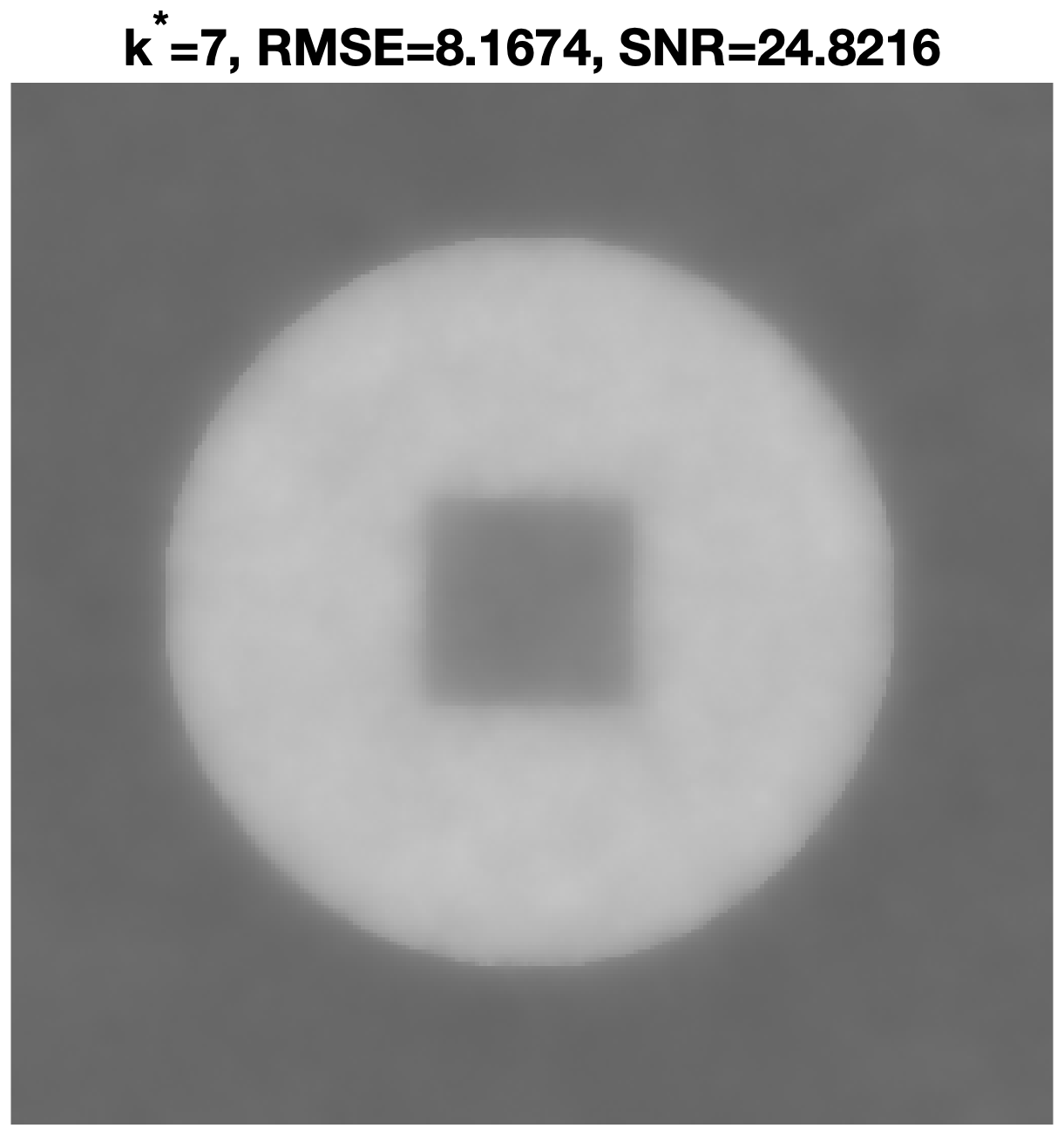}
		\end{subfigure}%
		\hspace*{\fill}   
		\begin{subfigure}{0.24\textwidth}
			\includegraphics[width=\linewidth]{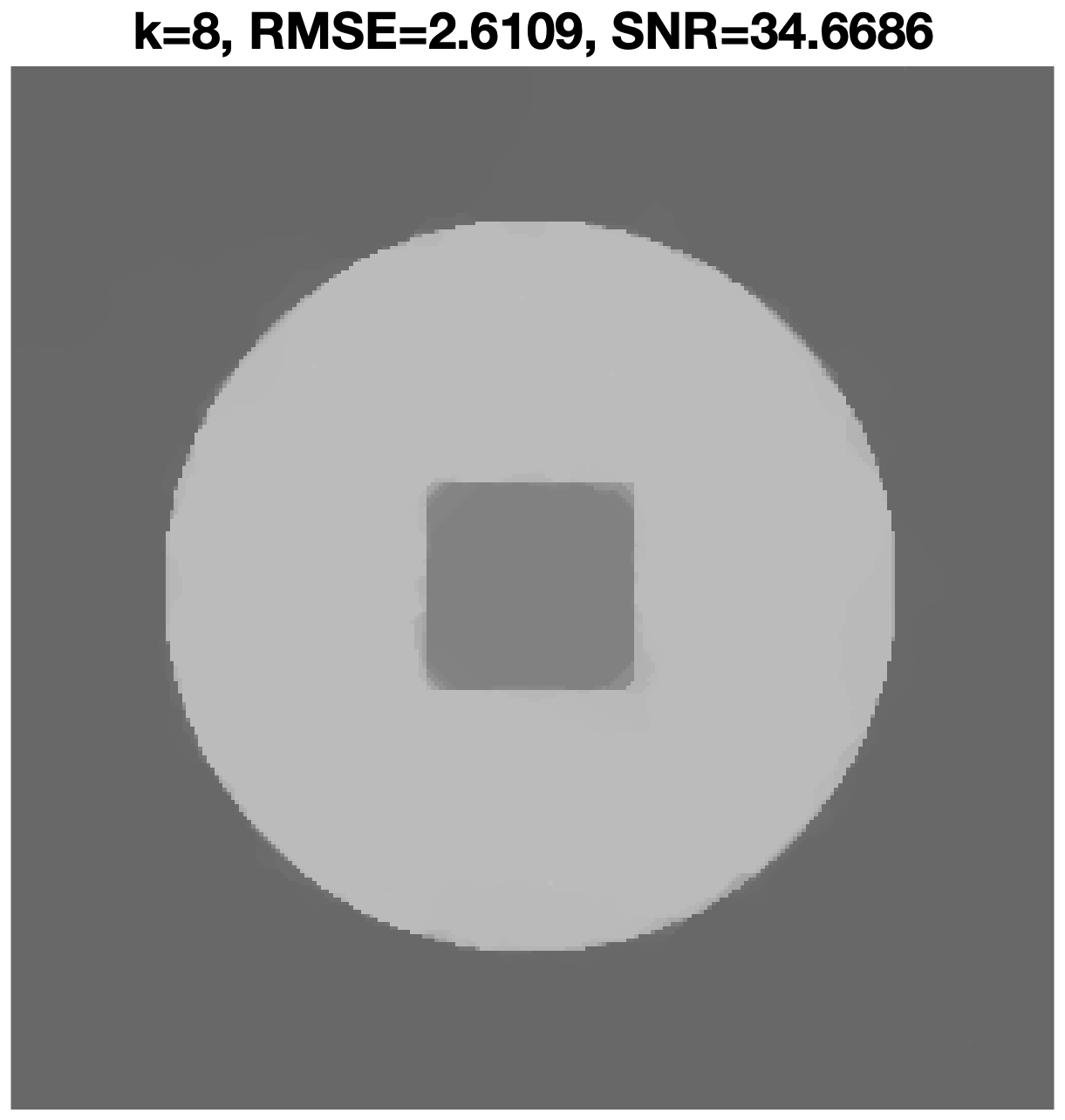}
		\end{subfigure}%
		\begin{subfigure}{0.24\textwidth}
			\includegraphics[width=\linewidth]{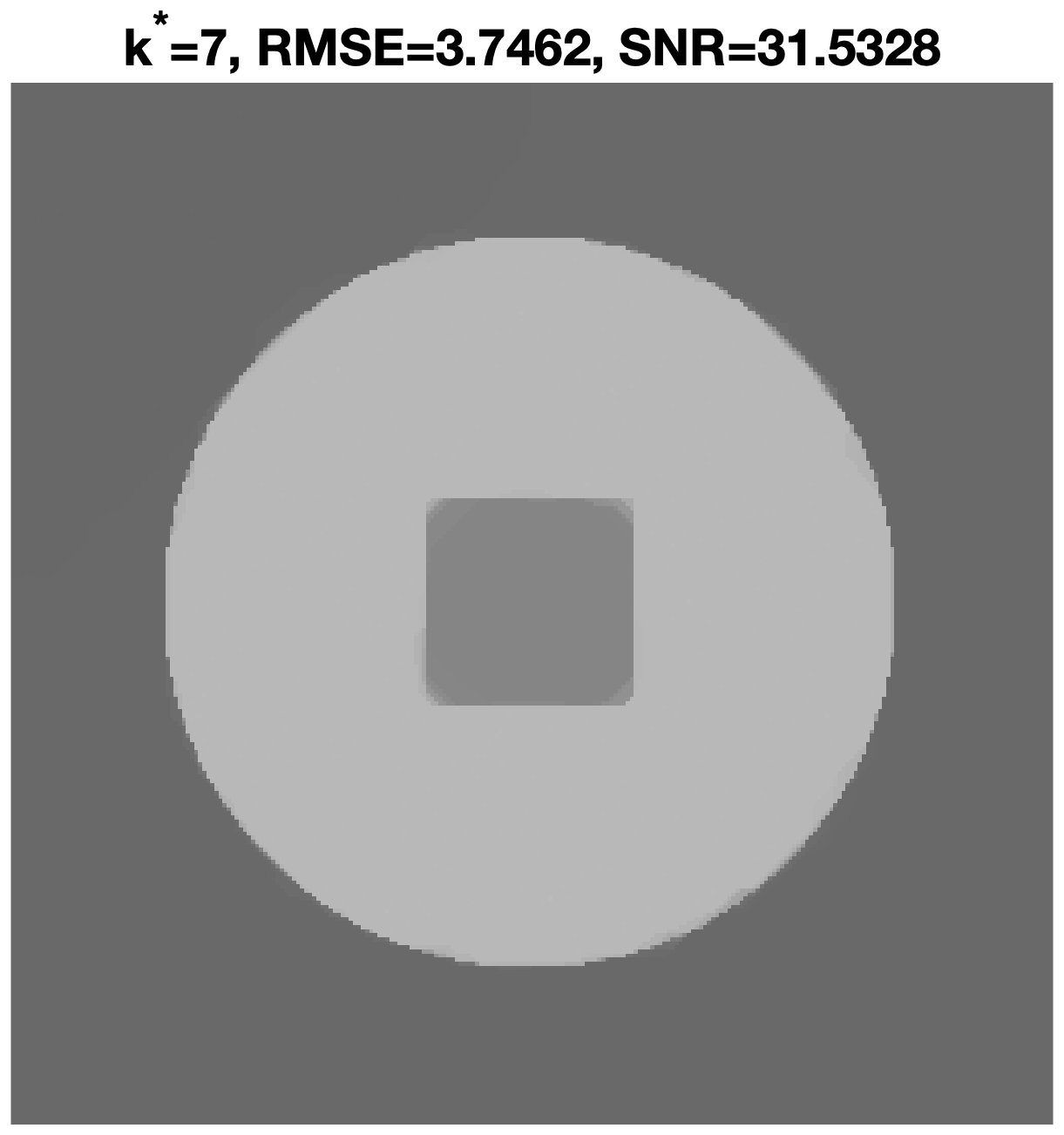}
		\end{subfigure}\\%
		\begin{subfigure}{0.24\textwidth}
			\includegraphics[width=\linewidth]{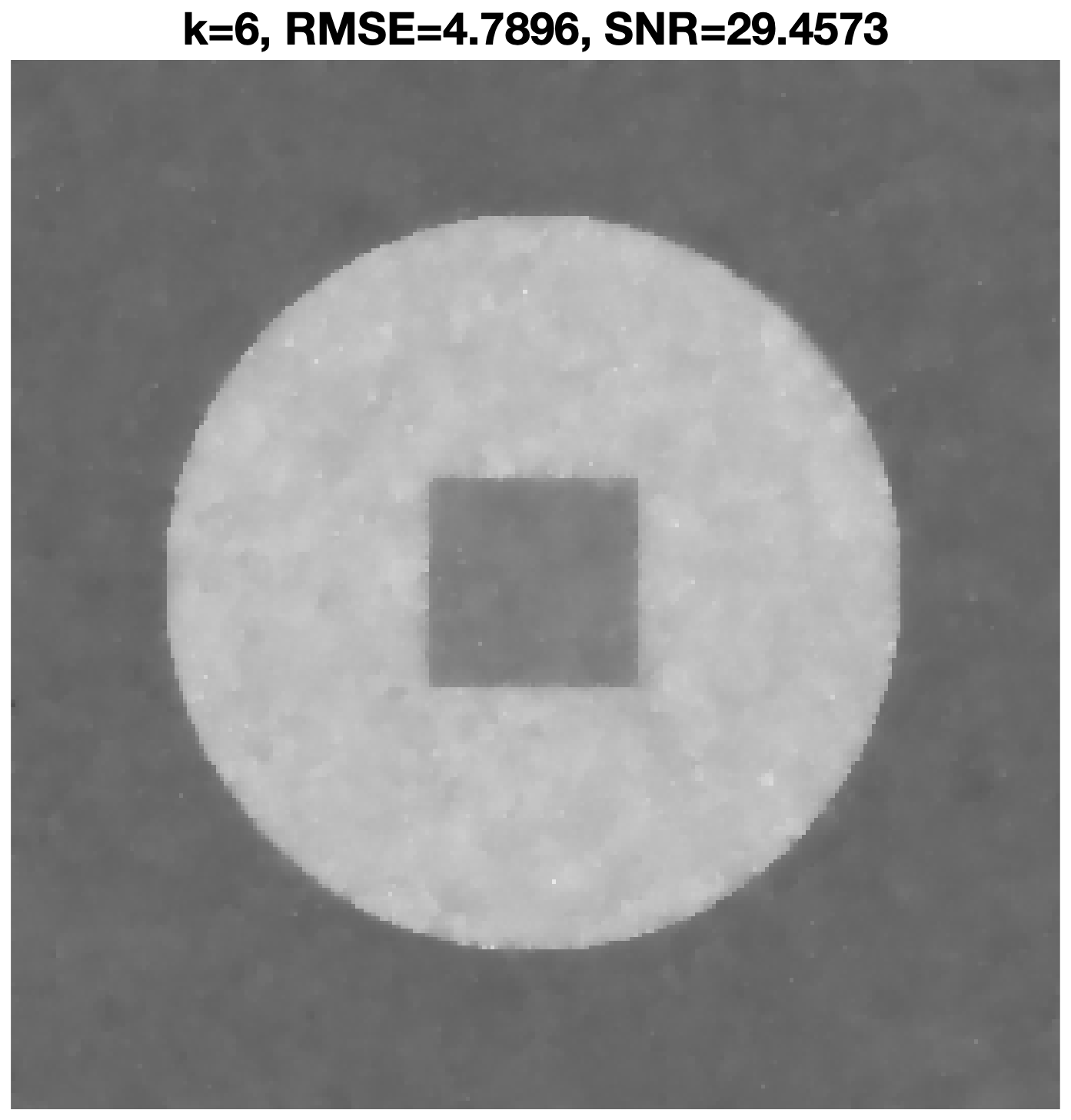}
		\end{subfigure}%
		\begin{subfigure}{0.24\textwidth}
			\includegraphics[width=\linewidth]{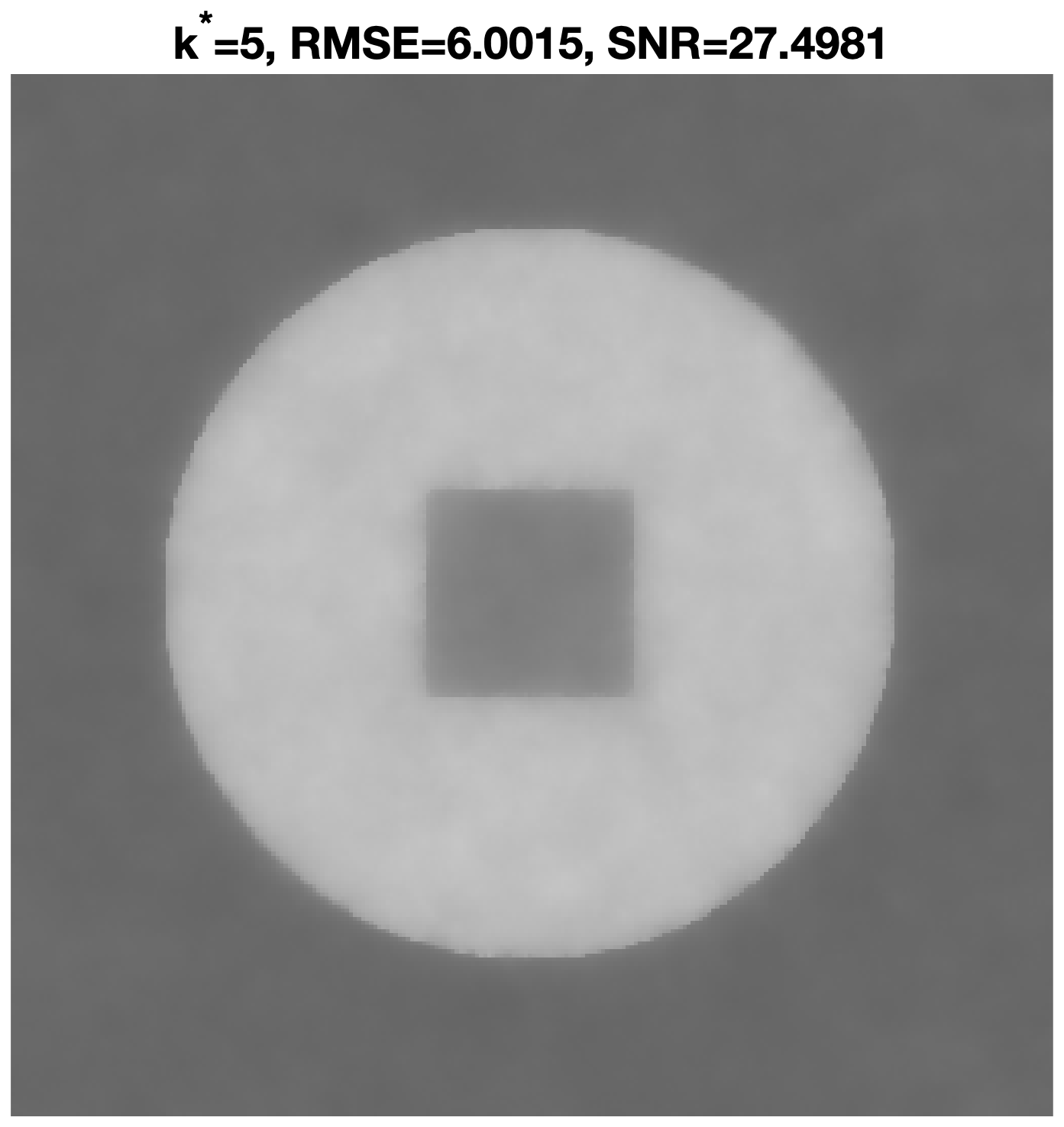}
		\end{subfigure}%
		\hspace*{\fill}   
		\begin{subfigure}{0.24\textwidth}
			\includegraphics[width=\linewidth]{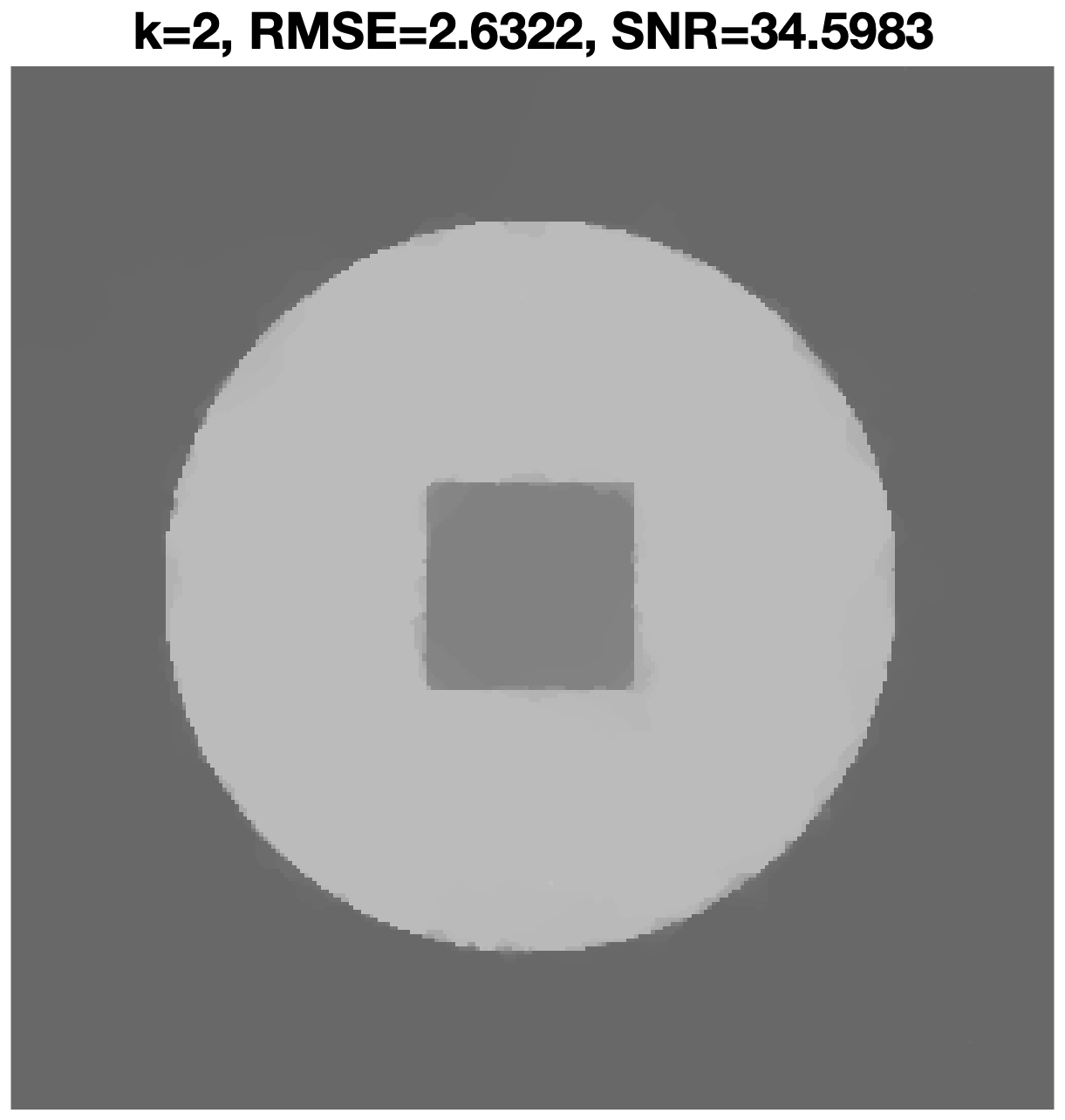}
		\end{subfigure}%
		\hspace*{0.24\textwidth}\\
		\begin{subfigure}{0.24\textwidth}
			\includegraphics[width=\linewidth]{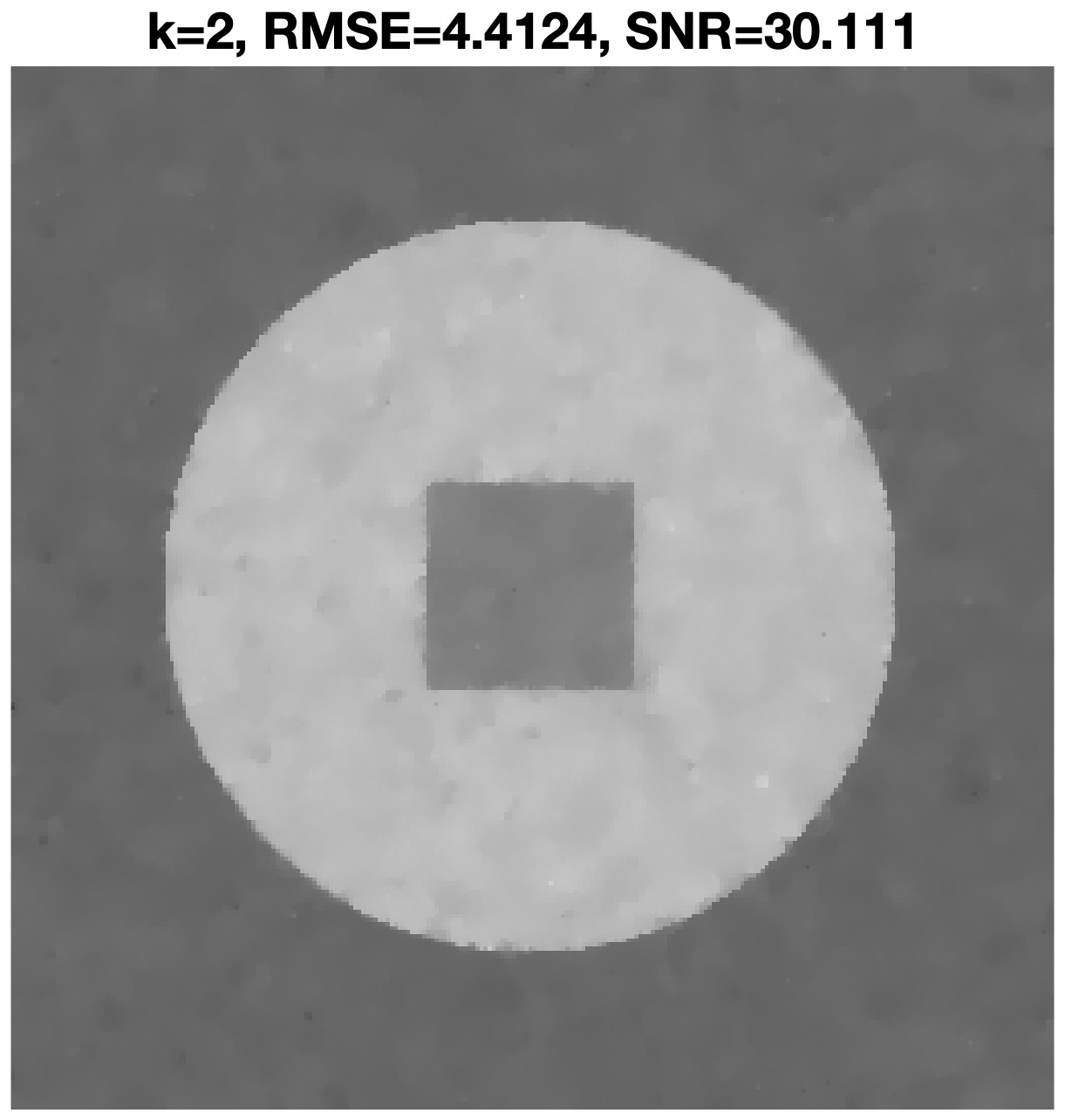}
		\end{subfigure}%
		\begin{subfigure}{0.24\textwidth}
			\includegraphics[width=\linewidth]{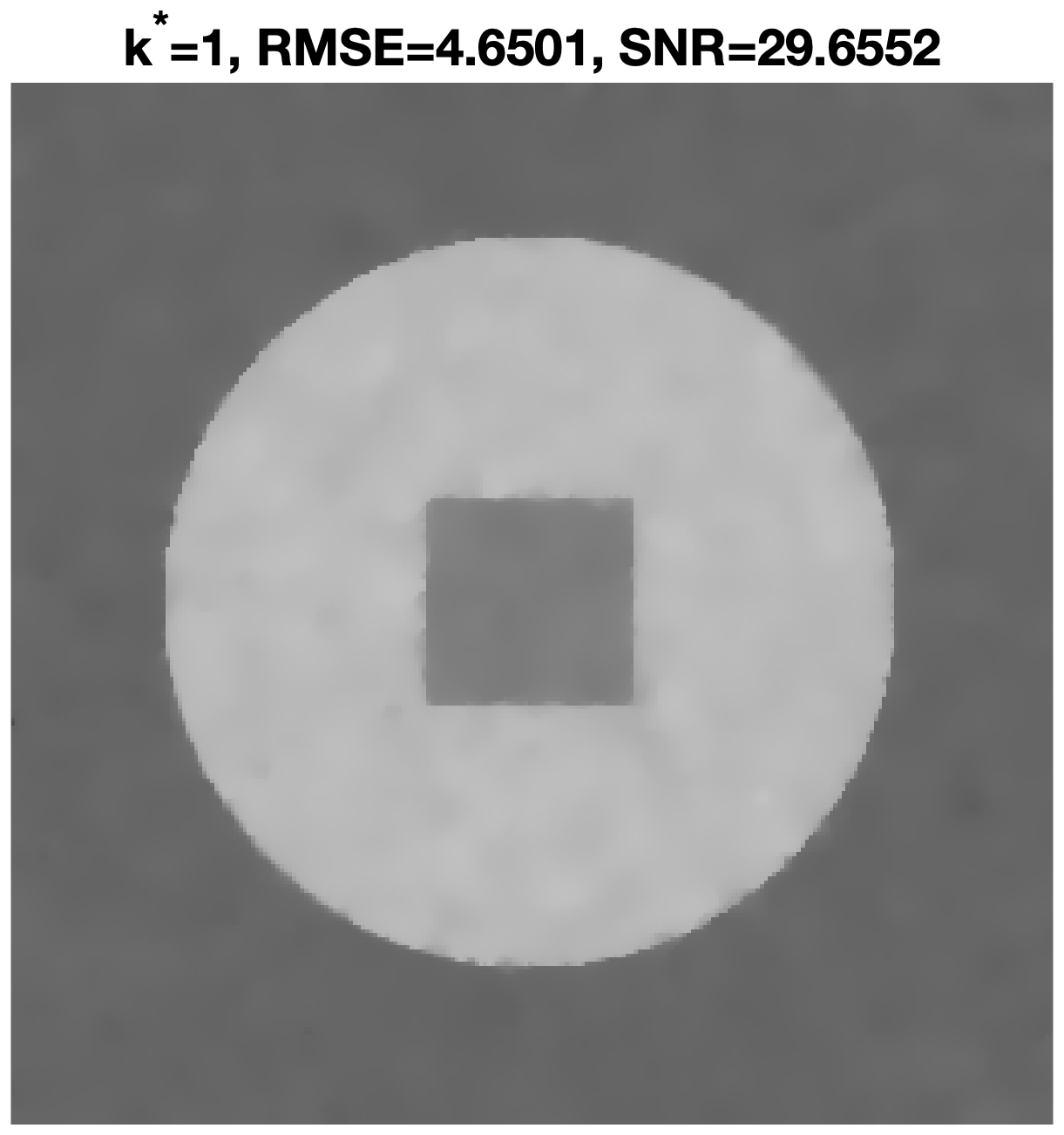}
		\end{subfigure}%
		\caption{SO MHDM EL (left) and ADMM (right) restored images. Rows 1, 2 and 3 are the regular, tight and refined model recoveries, respectively. Columns 1 and 3 are the $k_{min}$ restorations while columns 2 and 4 are the recoveries at the stopping criterion index $k^*$ (if $k^*\neq k_{min}$).} %
		\label{fig:SO_EL_ADMM_restored}
	\end{figure}%

	For the remaining higher-texture images, detailed recovery comparisons across all the SO MHDM models will be discussed in Section \ref{sec:comparisons}.  In summary, the best restorations (lowest {RMSE} and highest {SNR}) among the SO MHDM models are tight SO MHDM (ADMM) for `Cameraman' and `Geometry', and refined SO MHDM (EL) for `Barbara' and `Mandril'. This  confirms that the refined version is suitable for recovering  images with more texture. 
	%
	%
	\subsection{AA-log models}\label{subsec:aa-log}
	\subsubsection{Denoising}
	In contrast to the SO models, the AA MHDM and AA-log MHDM models address the noise directly, building a multiscale restoration through multiplicative decompositions. They are also well equipped to handle blurring, and perform comparably to SO in the regular and tight formulations. The AA MHDM method behaves similarly to the AA-log schemes (see Table \ref{tab:SNRcomparisons}), so the majority of discussion is dedicated to the latter. We note that with the same timestep as the AA-log scheme, a threshholding step to ensure that the iterates continue to satisfy $\inf_\Omega f^\delta \leq x_k \leq \sup_\Omega f^\delta$ helps with numerical stability. 
	
	In Fig.\,\ref{fig:AAlog_denoise_disk} we give the AA-log MHDM restorations, and note the reduced performance of the refined method, specifically on the smooth ``Geometry" image, is likely due to the weaker $ \|\log(u)\|_*$ penalty. 
	{As a result, it} may provide insufficient regularization to remove adequate noise, especially in smoother images. For the regular and tight formulations, the AA-log MHDM scheme produces better recoveries than its SO MHDM (EL) counterparts for the ``Geometry" image. We also include the tight $k^*$ recovery in Fig.\,\ref{fig:AAlog_denoise_disk}, which visually is a better restoration despite not obtaining the lowest RMSE.
	
	Worth remarking is the ability of the AA-log MHDM methods to recover corners and edges in the ``Geometry", even with very high noise levels. This is shown in Fig.\,\ref{fig:AA_high_noise}, which compares the original AA recovery to the AA-log MHDM tight restoration of a severely noise-degraded images (gamma noise $g(x;1)$, standard deviation 1), as 
	tested in Figures 2 and 3 from \cite{auber_aujol_2008}.  
	\begin{remark}
		As noted in \cite{tad_nez_ves2}, the presence of $\log(Tu)$ and $f^\delta/Tu$ terms in the AA-log MHDM models require the images to take strictly positive pixel values. Accordingly, images should be shifted away from zero, processed, and then shifted back appropriately when near-black pixels are expected. 	
	\end{remark}
	\begin{figure}[h]
		\centering
		\begin{subfigure}{0.24\textwidth}
			\includegraphics[width=\linewidth]{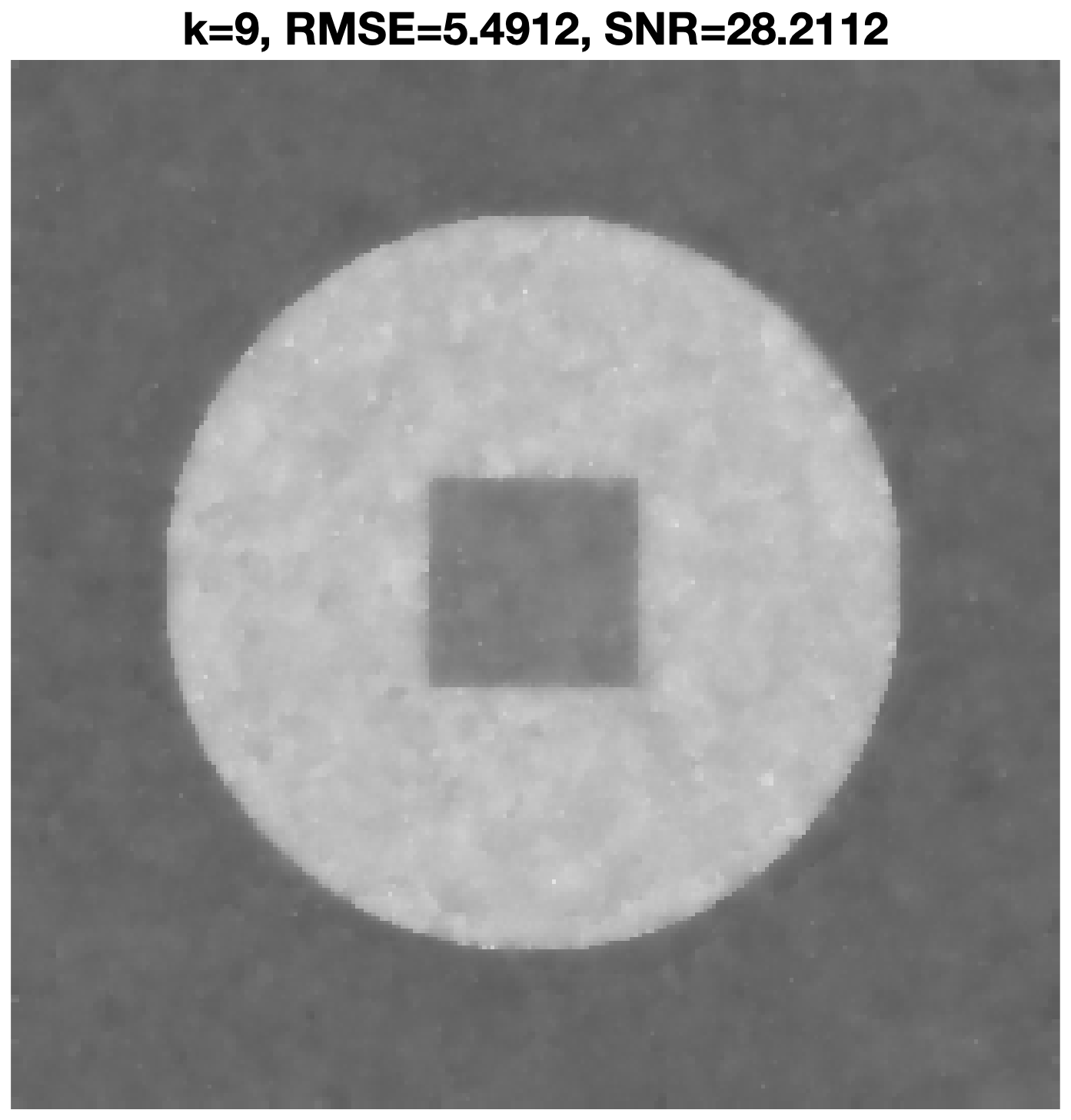}
		\end{subfigure}%
		\hspace{\fill}
		\begin{subfigure}{0.24\textwidth}
			\includegraphics[width=\linewidth]{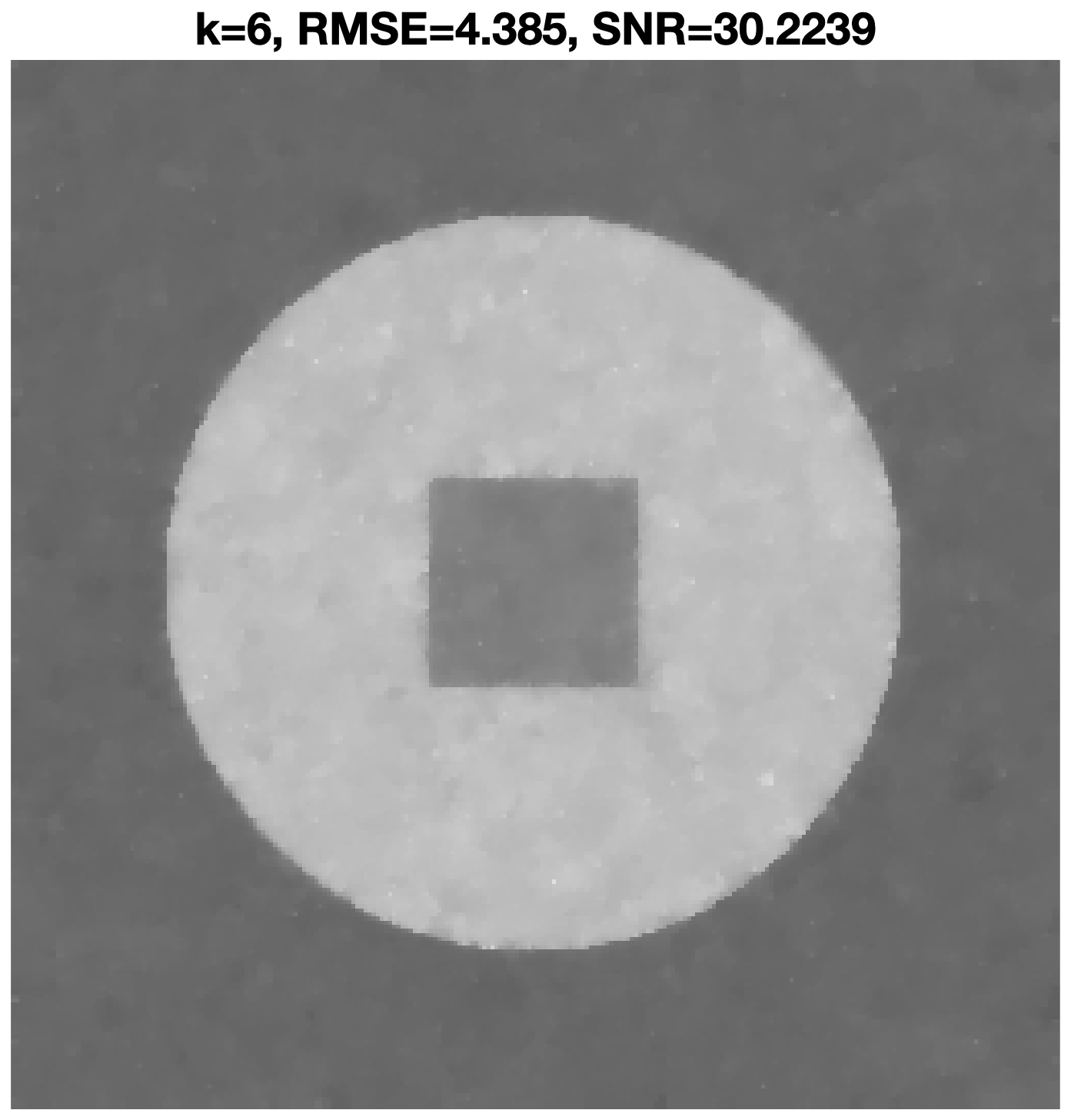}
		\end{subfigure}%
		\begin{subfigure}{0.24\textwidth}
			\includegraphics[width=\linewidth]{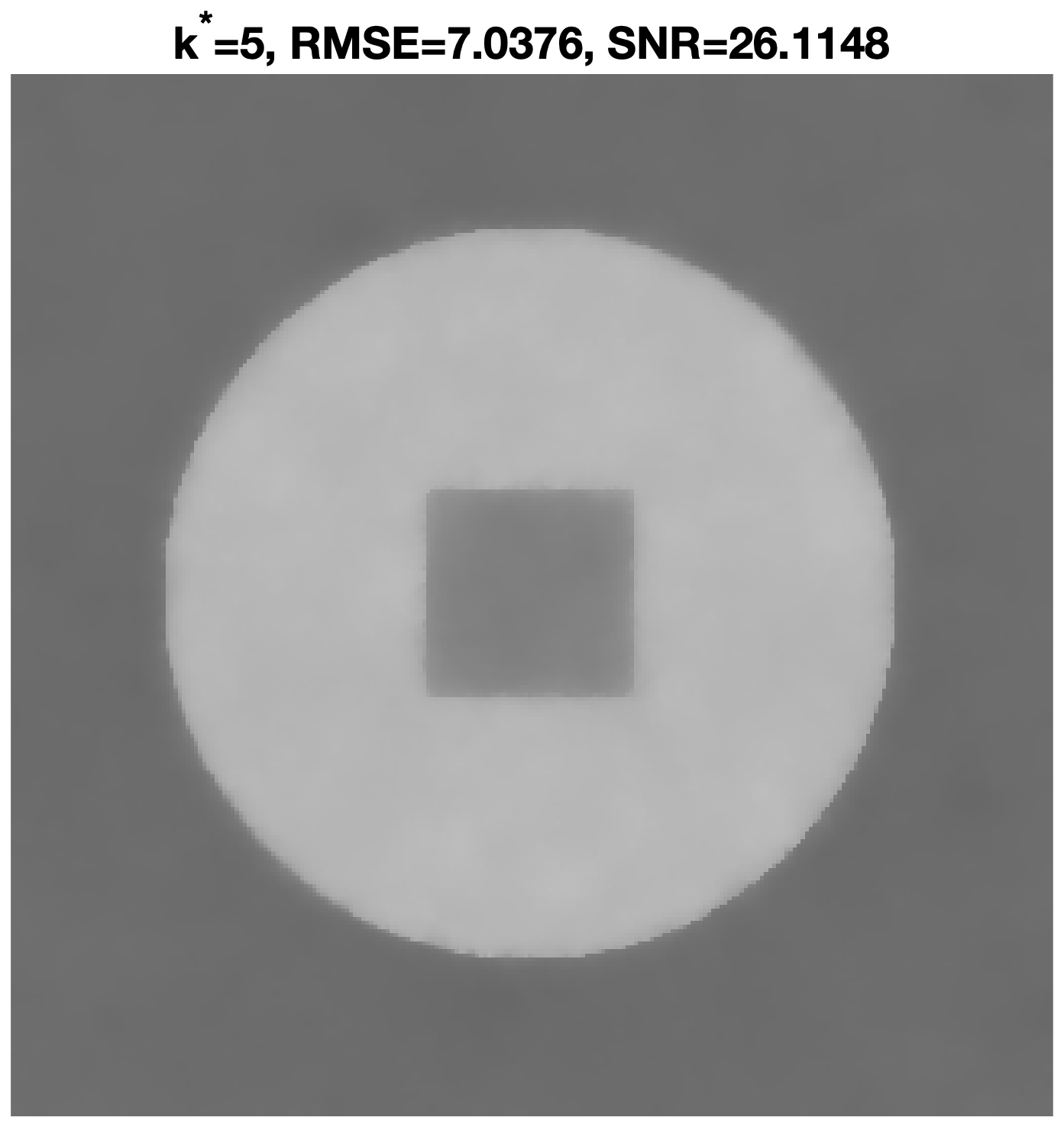}
		\end{subfigure}%
		\hspace{\fill}
		\begin{subfigure}{0.24\textwidth}
			\includegraphics[width=\linewidth]{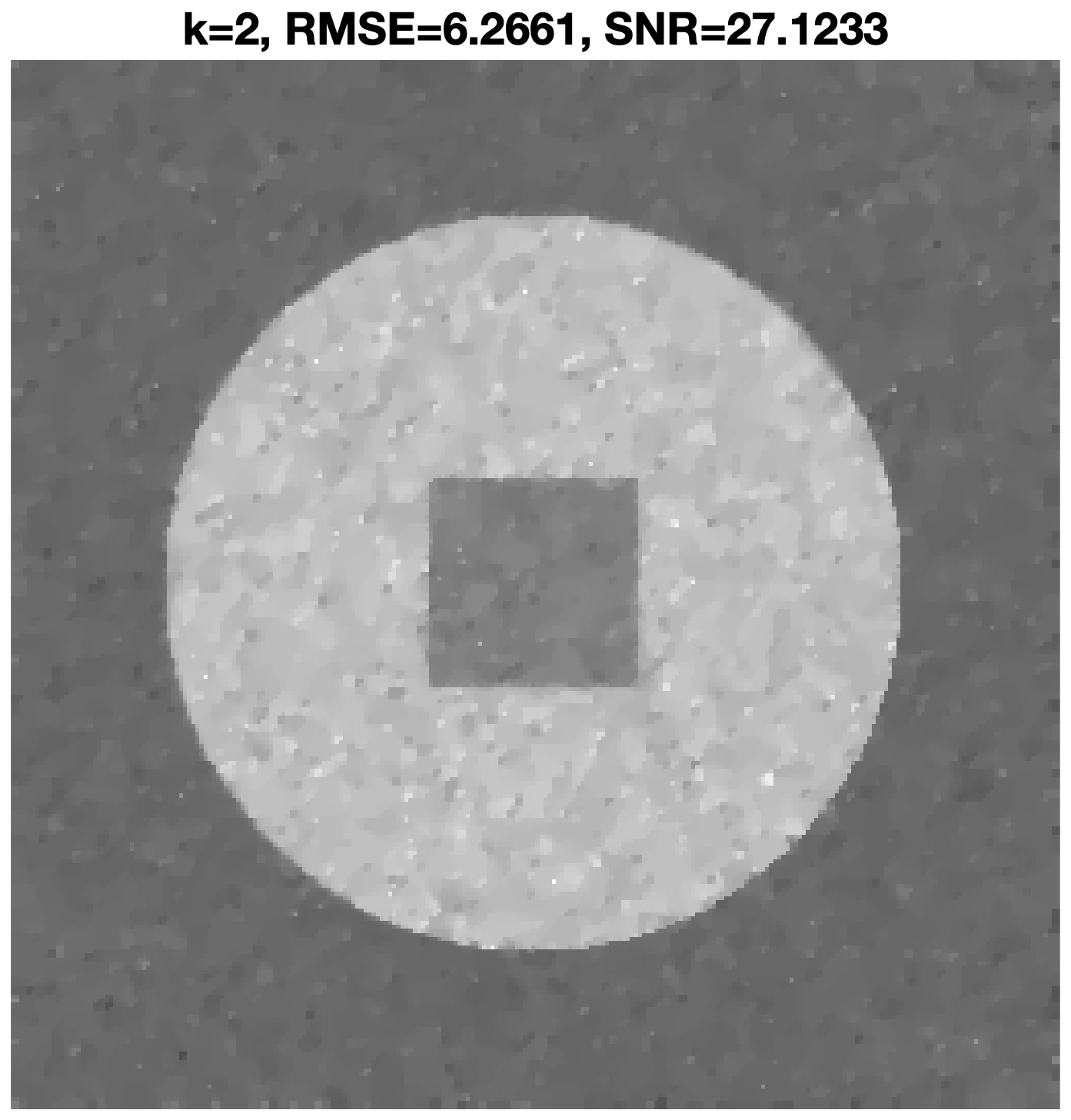}
		\end{subfigure}\\%
		\begin{subfigure}{0.24\textwidth}
			\includegraphics[width=\linewidth]{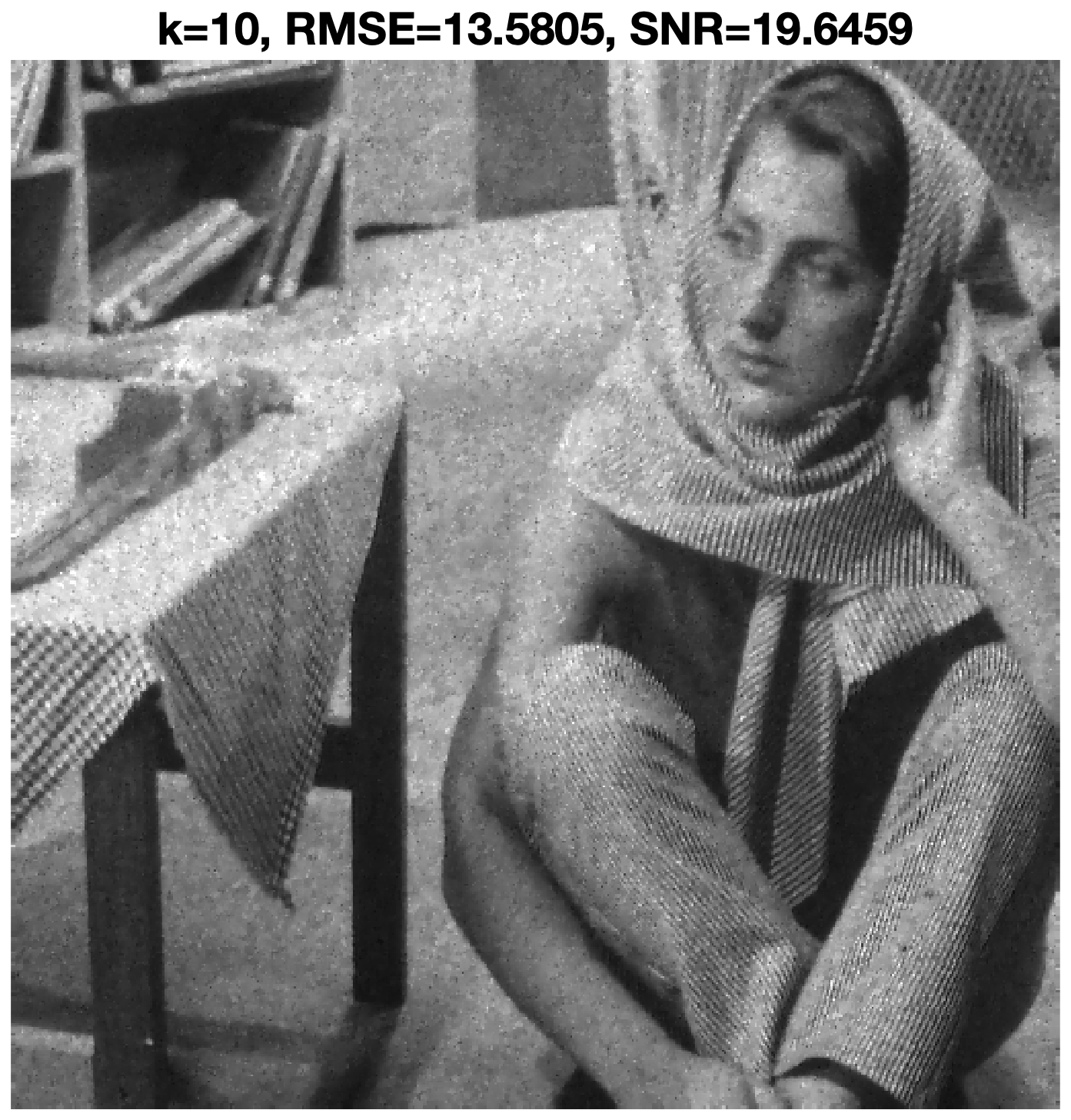}
		\end{subfigure}%
		\hspace{\fill}
		\begin{subfigure}{0.24\textwidth}
			\includegraphics[width=\linewidth]{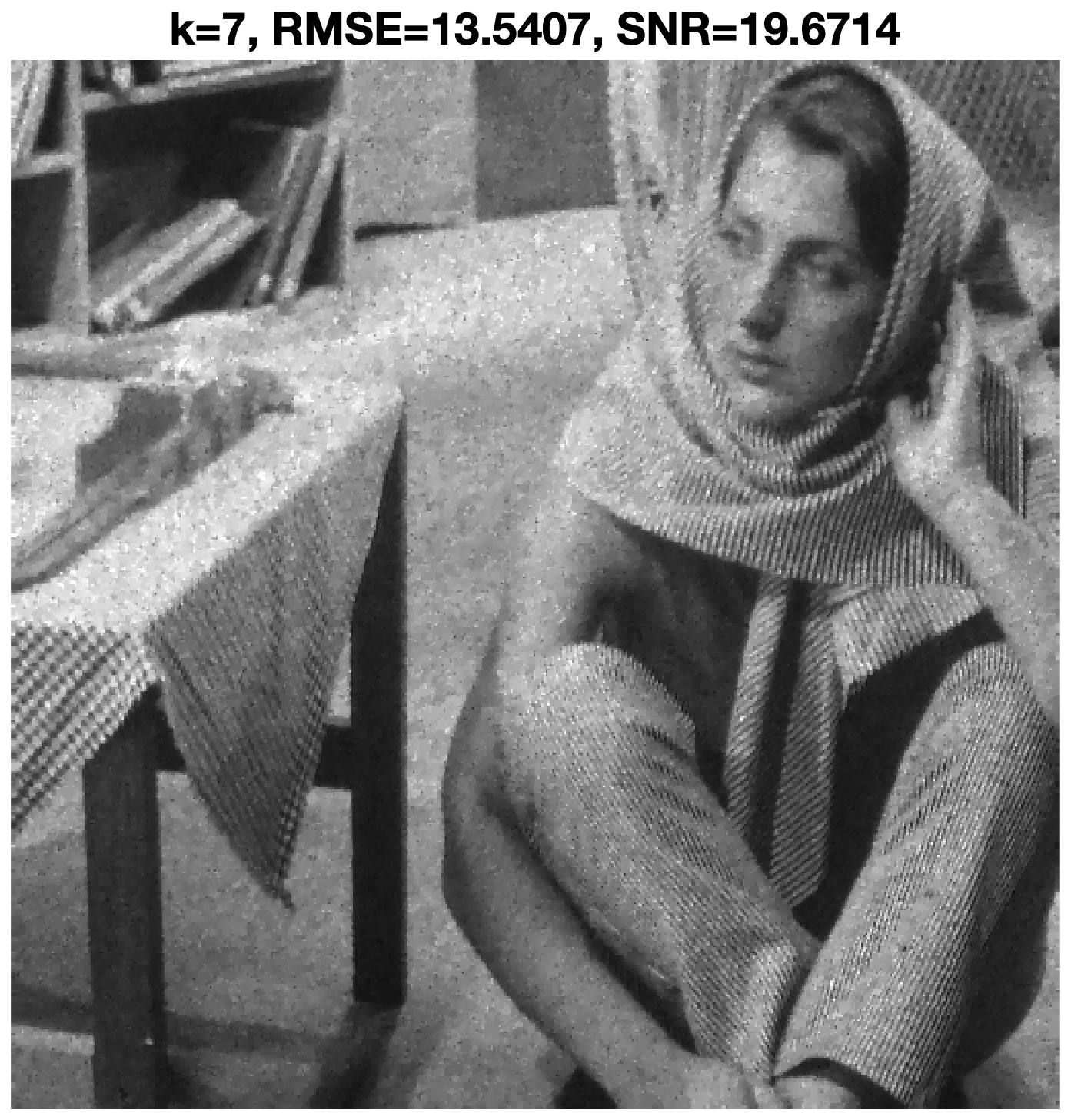}
		\end{subfigure}%
		\begin{subfigure}{0.24\textwidth}
			\includegraphics[width=\linewidth]{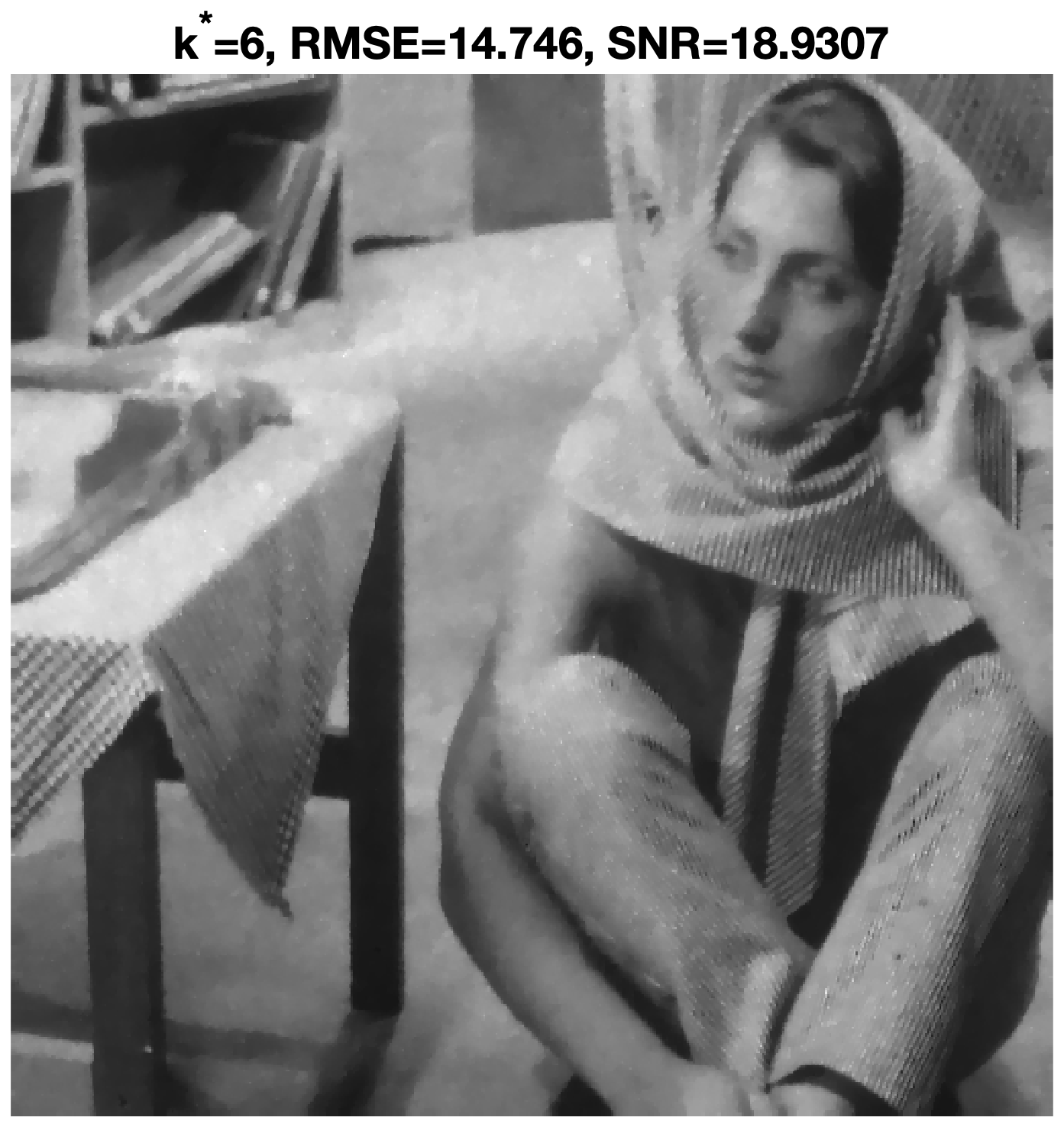}
		\end{subfigure}%
		\hspace{\fill}
		\begin{subfigure}{0.24\textwidth}
			\includegraphics[width=\linewidth]{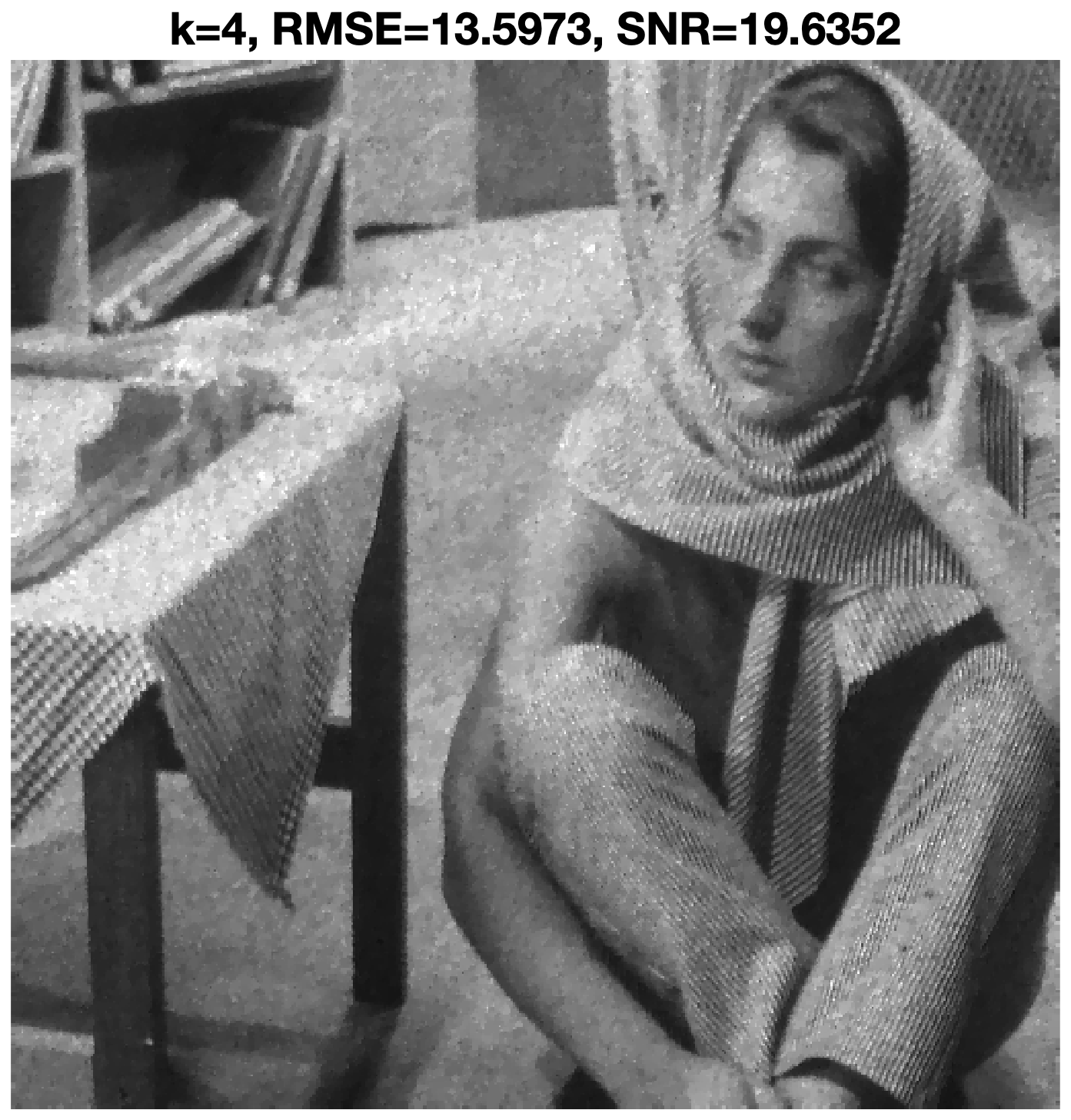}
		\end{subfigure}%
		\caption{AA-log MHDM denoised images. Columns 1, 2, and 4 are the $k_{min}$ regular, tight and refined restorations, respectively. Columns 3 is the tight recovery at the stopping criterion $k^*$. }%
		\label{fig:AAlog_denoise_disk}%
	\end{figure}%
	\begin{figure}[h]
		\centering
		\begin{subfigure}{0.24\textwidth}
			\includegraphics[width=\linewidth]{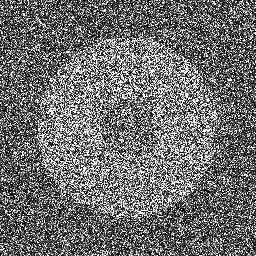}
		\end{subfigure}\hspace{2.3mm}%
		\begin{subfigure}{0.24\textwidth}
			\includegraphics[width=\linewidth]{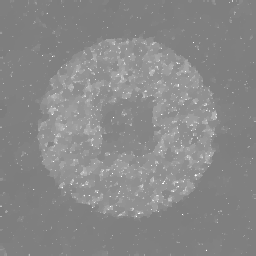}
		\end{subfigure}\hspace{2.3mm}%
		\begin{subfigure}{0.24\textwidth}
			\includegraphics[width=\linewidth]{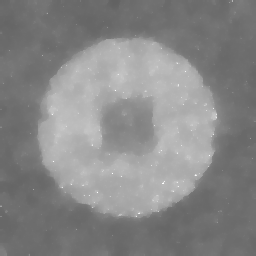}
		\end{subfigure}%
		\caption{Recoveries from severe noise (standard deviation 1). From left to right: noisy image ($SNR = -0.079$), original AA model \cite{auber_aujol_2008} ($SNR = 13.42$), and the AA-log MHDM tight recovery ($k_{min} = 4$, $SNR = 20.22$).}
		\label{fig:AA_high_noise}
	\end{figure}
	\subsubsection{Denoising-deblurring}
	One of the primary advantages of the AA-log MHDM models is handling deblurring in addition to denoising. Figure \ref{fig:blurred} gives the blurry, noisy counterparts of the test images, and Fig.\,\ref{fig:AA_blur} shows  the AA-log MHDM recoveries. All images are blurred with a $5\times 5$ Gaussian filter with standard deviation $\sqrt{2}$. The method effectively sharpens edges (see ``Cameraman's" jacket and ``Geometry") while maintaining texture (the tablecloth in ``Barbara" and whiskers in ``Mandril"). The refined version with its weaker $*$-norm suffers from greater numerical instability, exhibiting a slight drop from the tight version in SNR. 
	\begin{figure}[h]
		\centering
		\begin{subfigure}{0.24\textwidth}
			\includegraphics[width=\linewidth]{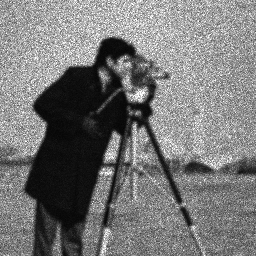}
			\caption{RMSE= 30.58, SNR=12.90}
		\end{subfigure}%
		\hspace*{\fill}%
		\begin{subfigure}{0.24\textwidth}
			\includegraphics[width=\linewidth]{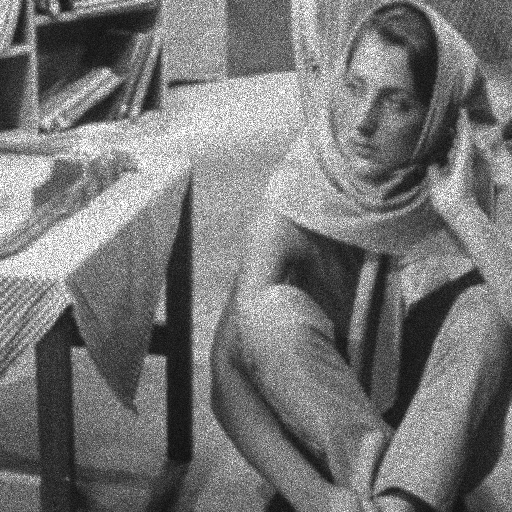}
			\caption{RMSE=30.06, SNR=12.74}
		\end{subfigure}%
		\hspace*{\fill}
		\begin{subfigure}{0.24\textwidth}
			\includegraphics[width=\linewidth]{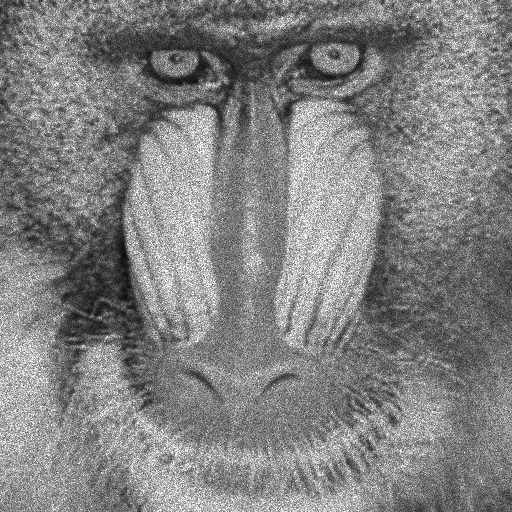}
			\caption{RMSE=29.50, SNR=13.24}
		\end{subfigure}%
		\hspace*{\fill}%
		\begin{subfigure}{0.24\textwidth}
			\includegraphics[width=\linewidth]{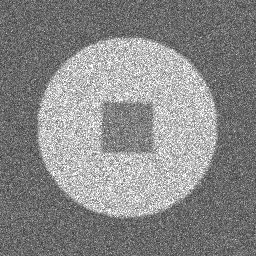}
			\caption{RMSE=28.84, SNR=13.86}
		\end{subfigure}%
		\caption{Images degraded with multiplicative gamma noise and Gaussian blur. Blurring from a $5\times 5$ filter with standard deviation $\sqrt{2}$. Noise as before with previous images.} \label{fig:blurred}
	\end{figure}%
	\begin{figure}[h]
		\centering
		\begin{subfigure}{0.24\textwidth}
			\includegraphics[width=\linewidth]{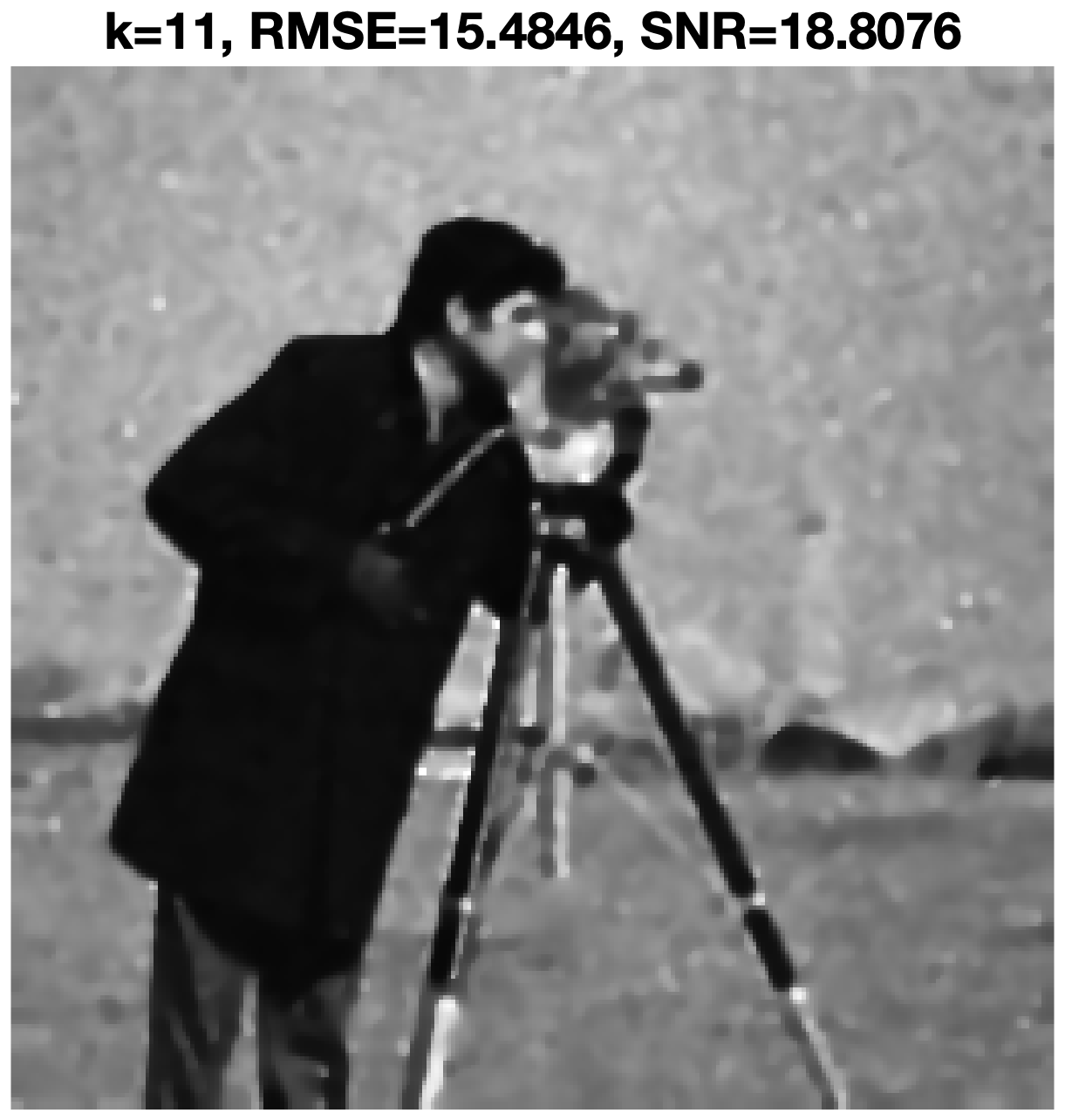}
		\end{subfigure}%
		\hspace*{\fill}  
		\begin{subfigure}{0.24\textwidth}
			\includegraphics[width=\linewidth]{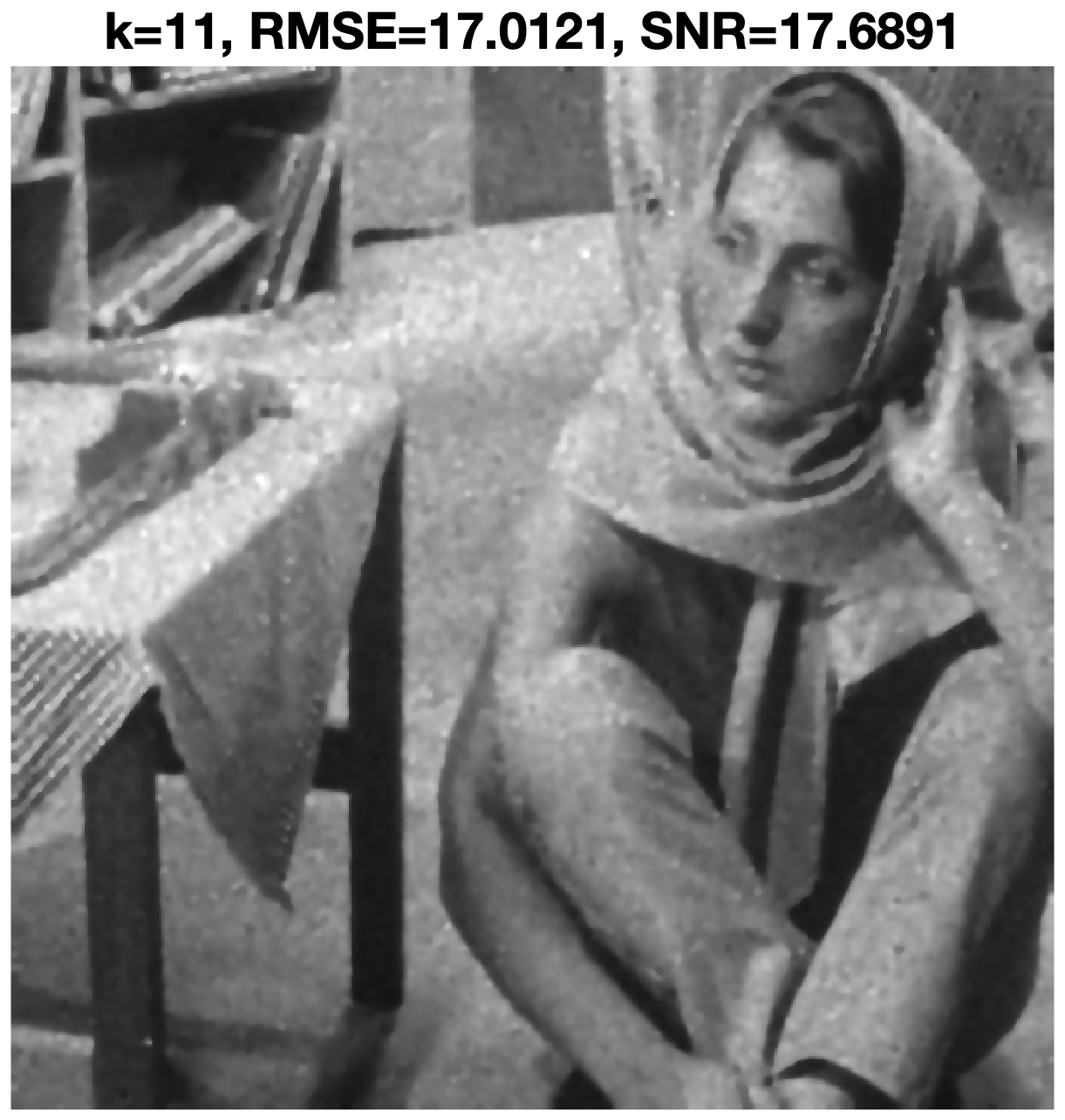}
		\end{subfigure}%
		\hspace*{\fill}   
		\begin{subfigure}{0.24\textwidth}
			\includegraphics[width=\linewidth]{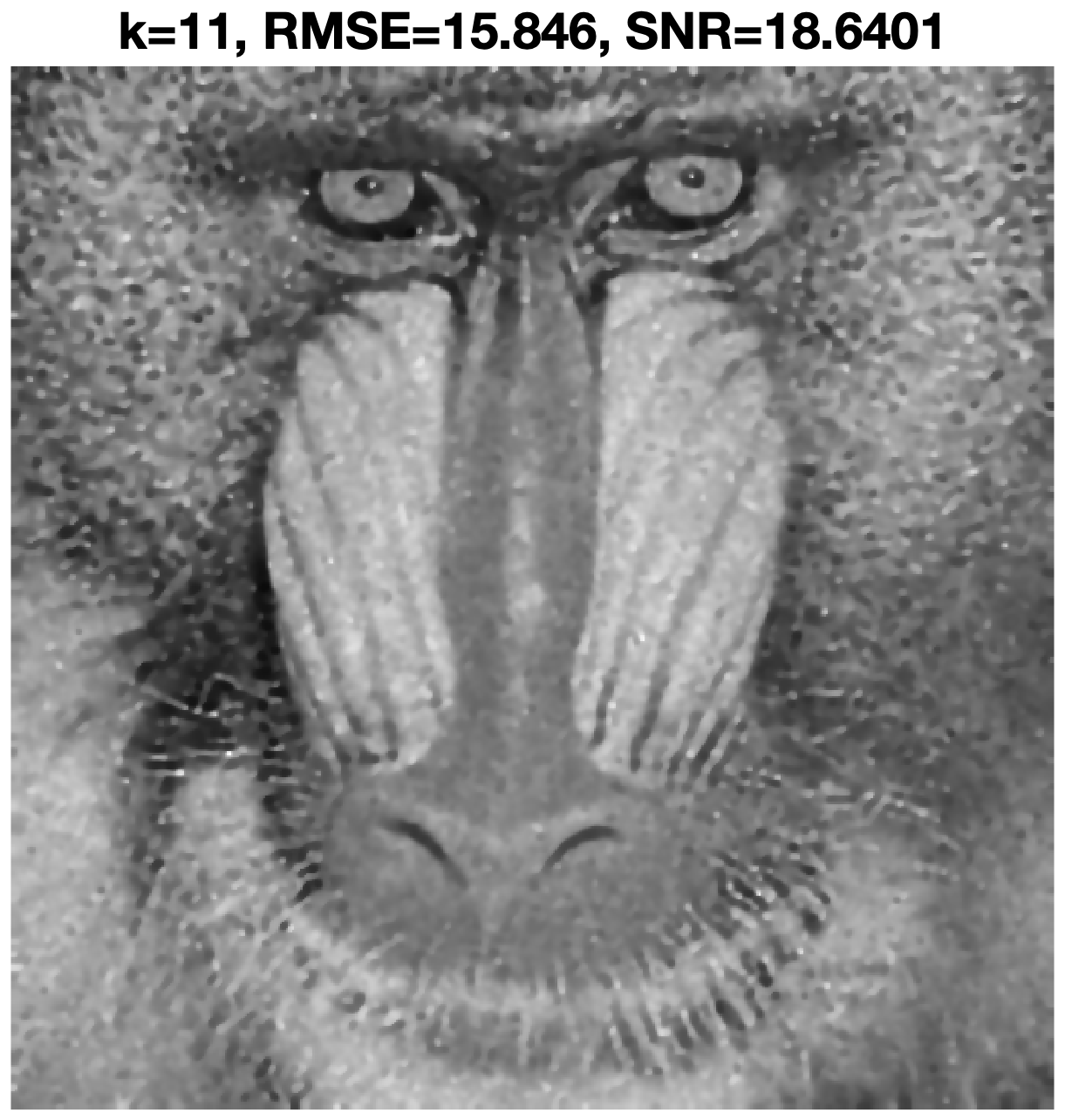}
		\end{subfigure}%
		\hspace*{\fill}
		\begin{subfigure}{0.24\textwidth}
			\includegraphics[width=\linewidth]{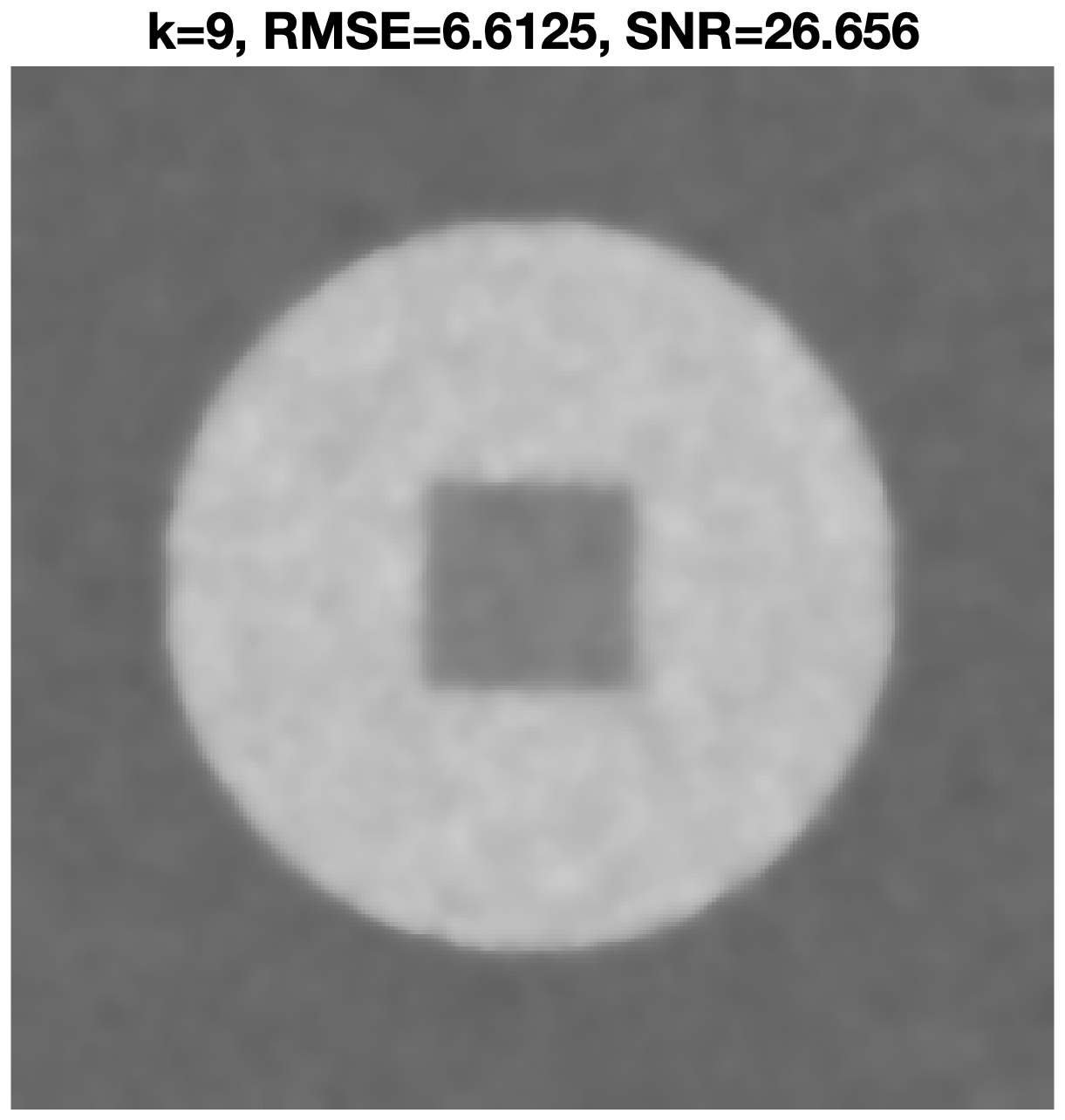}
		\end{subfigure}\\%
		\begin{subfigure}{0.24\textwidth}
			\includegraphics[width=\linewidth]{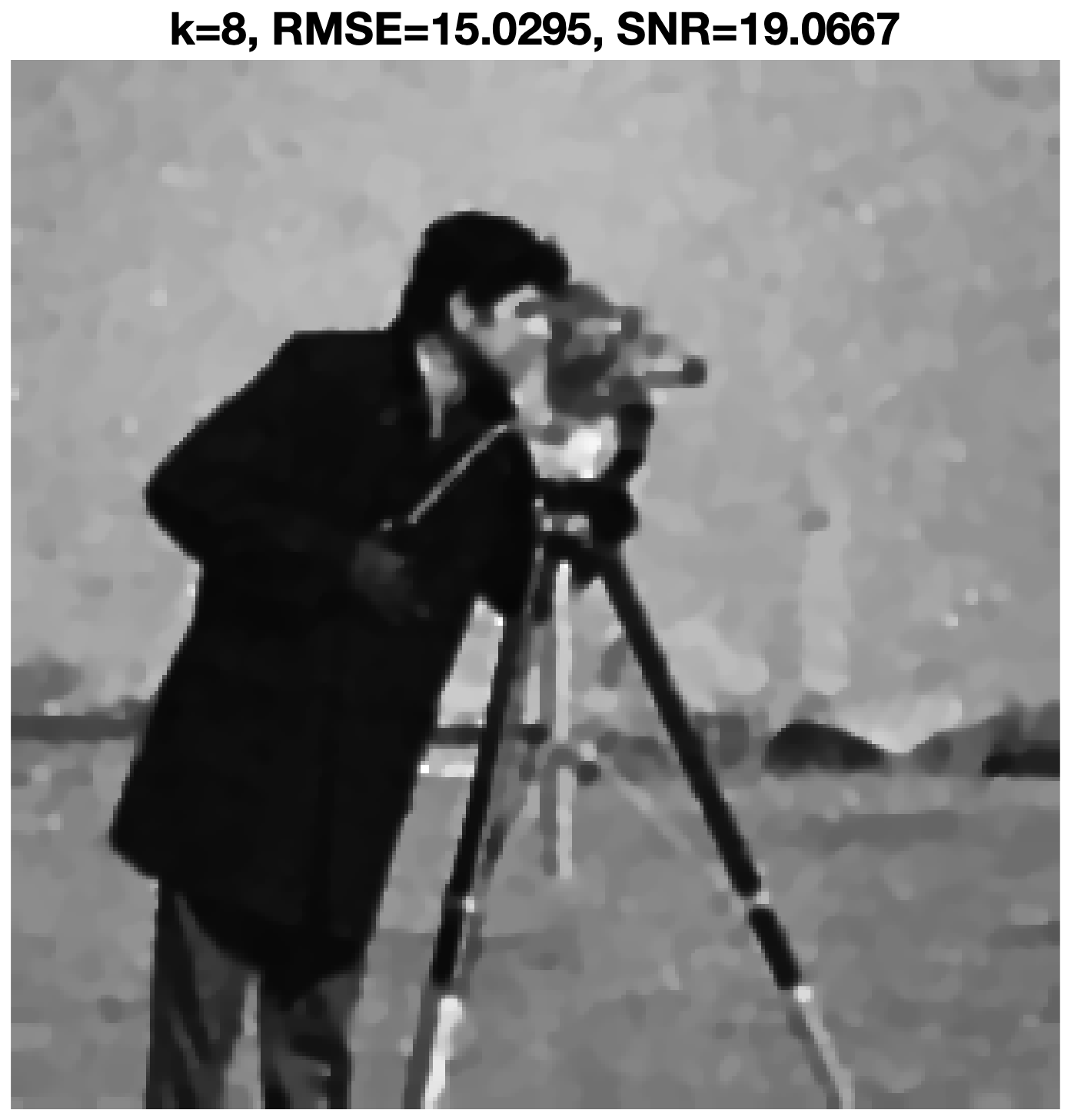}
		\end{subfigure}%
		\hspace*{\fill}  
		\begin{subfigure}{0.24\textwidth}
			\includegraphics[width=\linewidth]{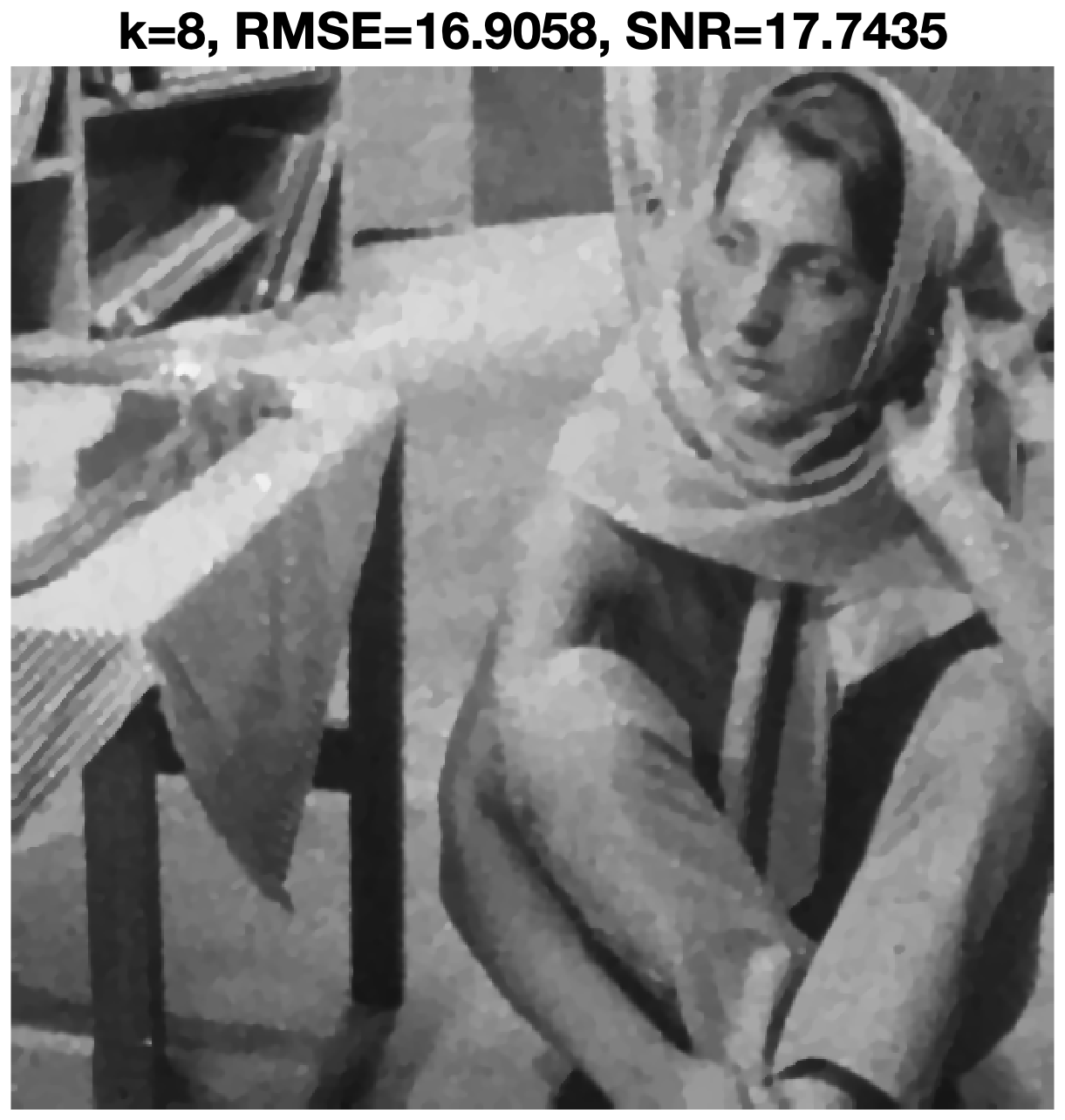}
		\end{subfigure}%
		\hspace*{\fill}   
		\begin{subfigure}{0.24\textwidth}
			\includegraphics[width=\linewidth]{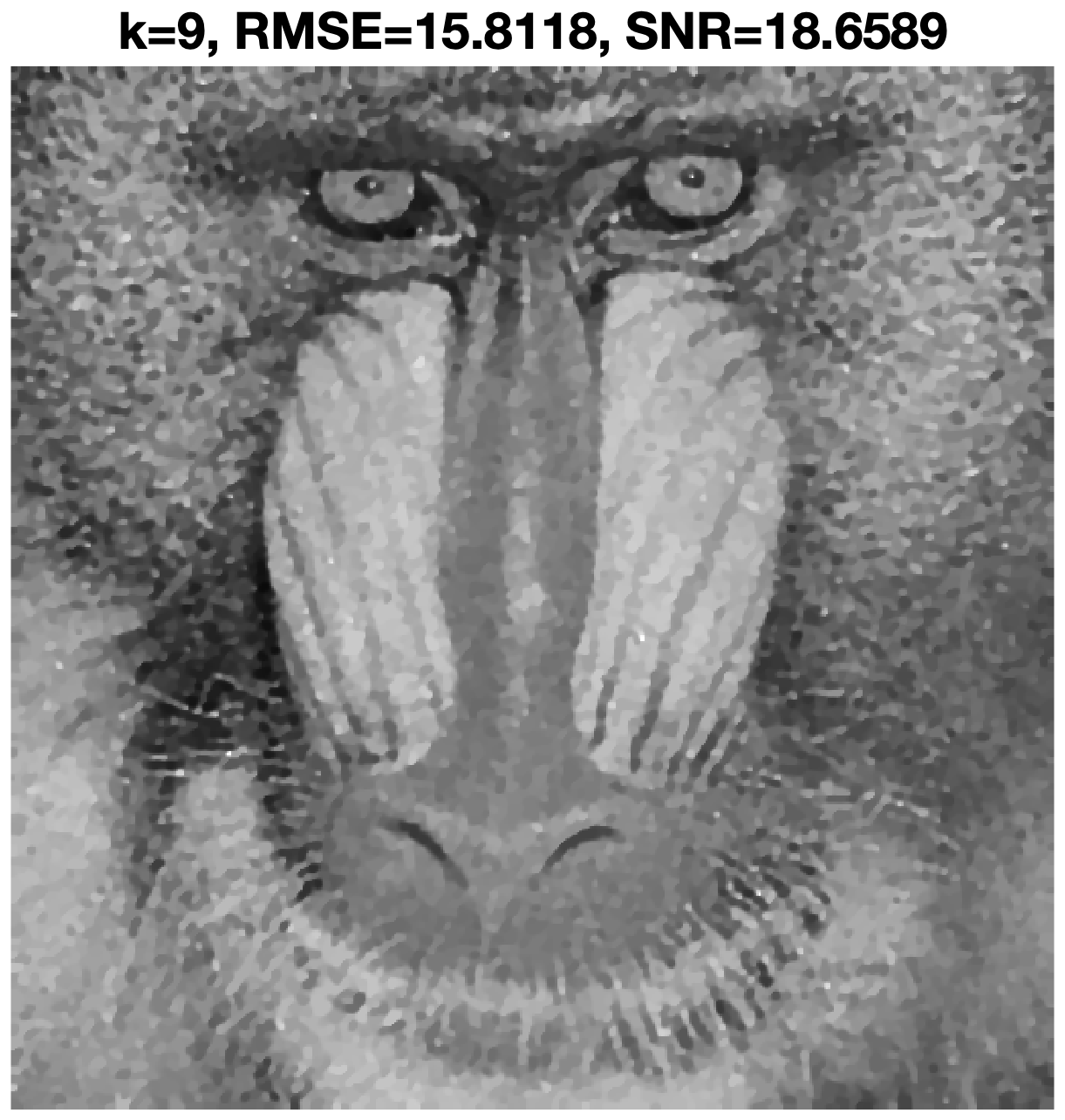}
		\end{subfigure}%
		\hspace*{\fill}
		\begin{subfigure}{0.24\textwidth}
			\includegraphics[width=\linewidth]{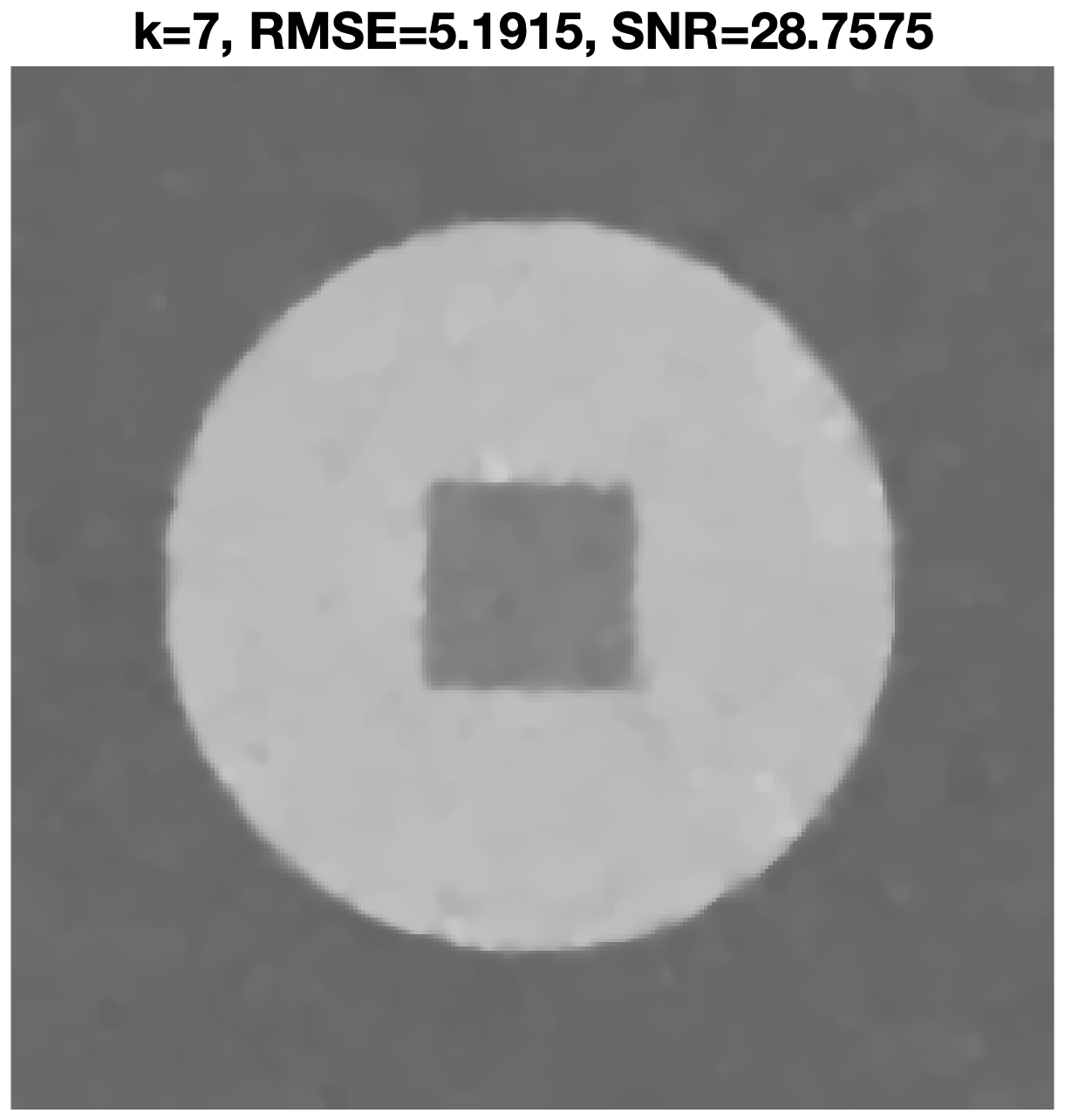}
		\end{subfigure}\\%
		\begin{subfigure}{0.24\textwidth}
			\includegraphics[width=\linewidth]{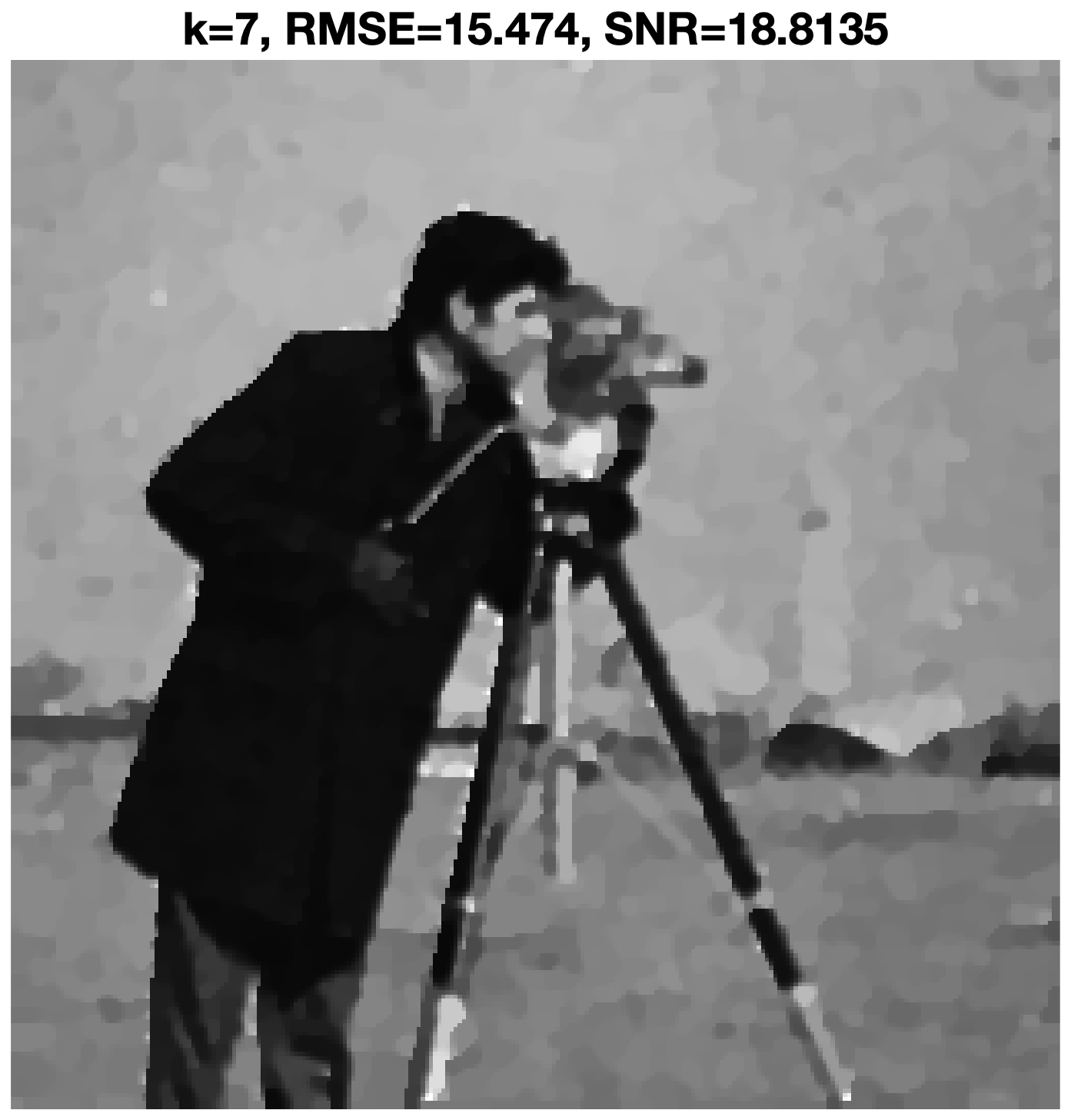}
		\end{subfigure}%
		\hspace*{\fill}  
		\begin{subfigure}{0.24\textwidth}
			\includegraphics[width=\linewidth]{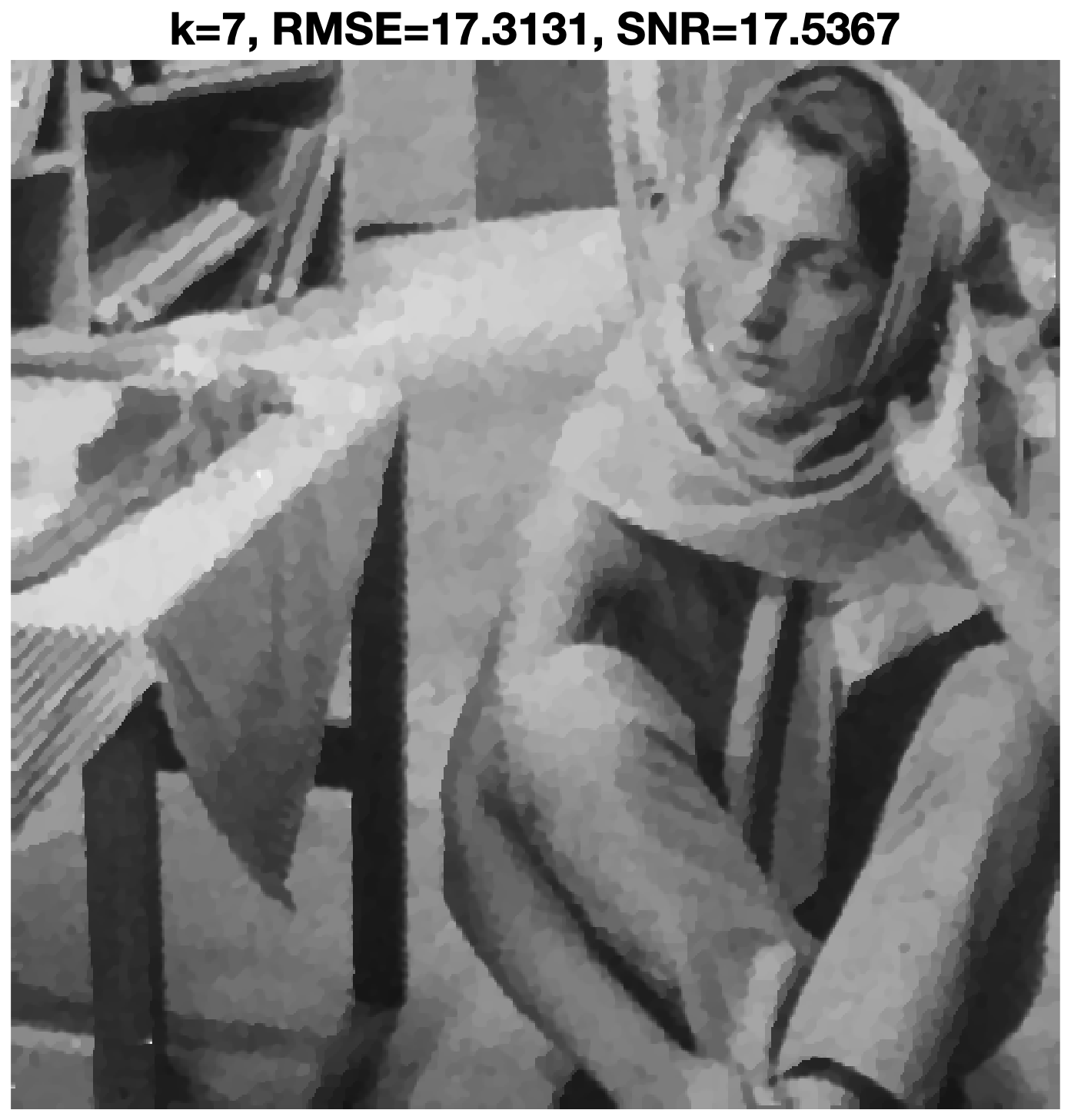}
		\end{subfigure}%
		\hspace*{\fill}   
		\begin{subfigure}{0.24\textwidth}
			\includegraphics[width=\linewidth]{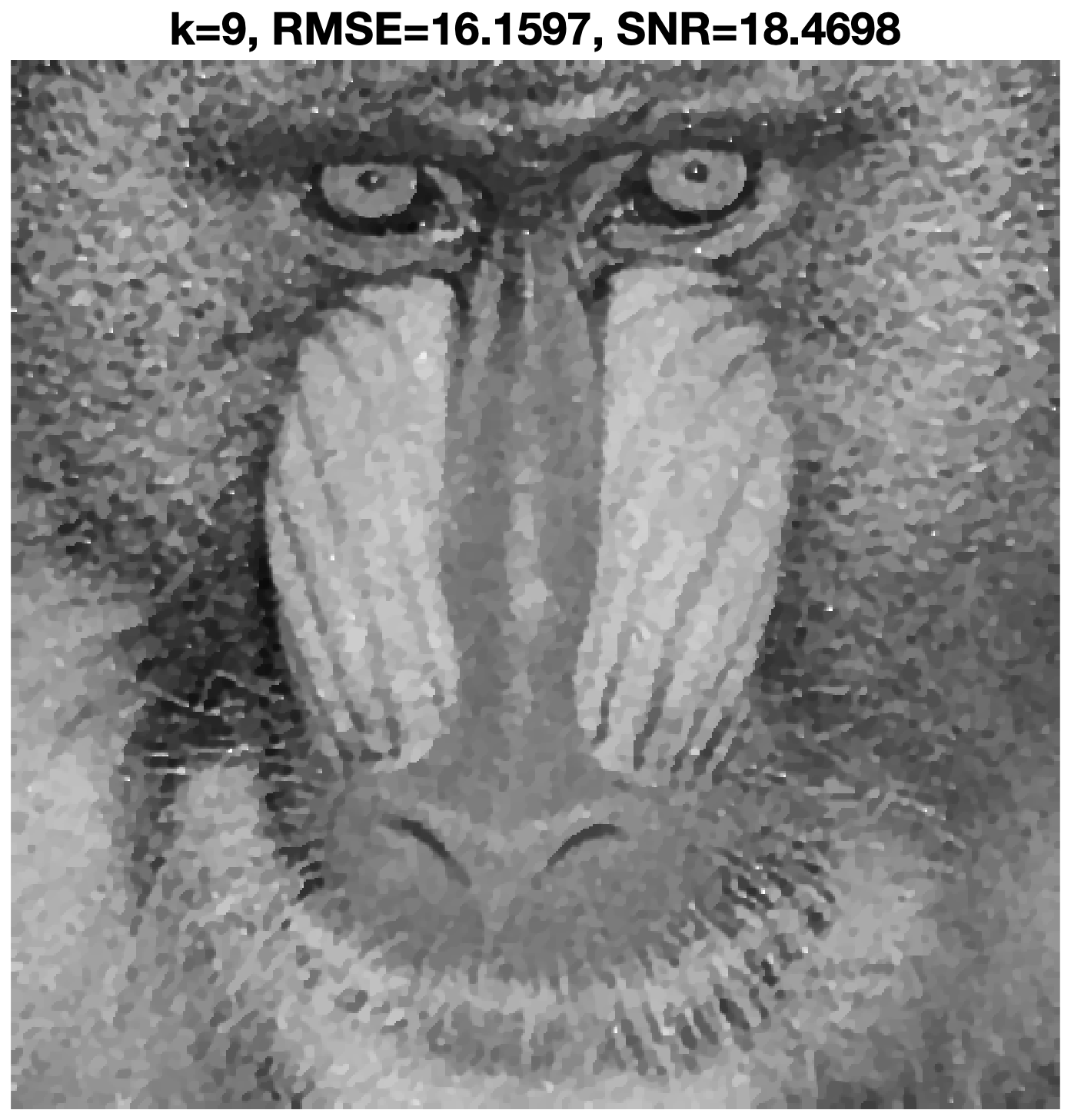}
		\end{subfigure}%
		\hspace*{\fill}
		\begin{subfigure}{0.24\textwidth}
			\includegraphics[width=\linewidth]{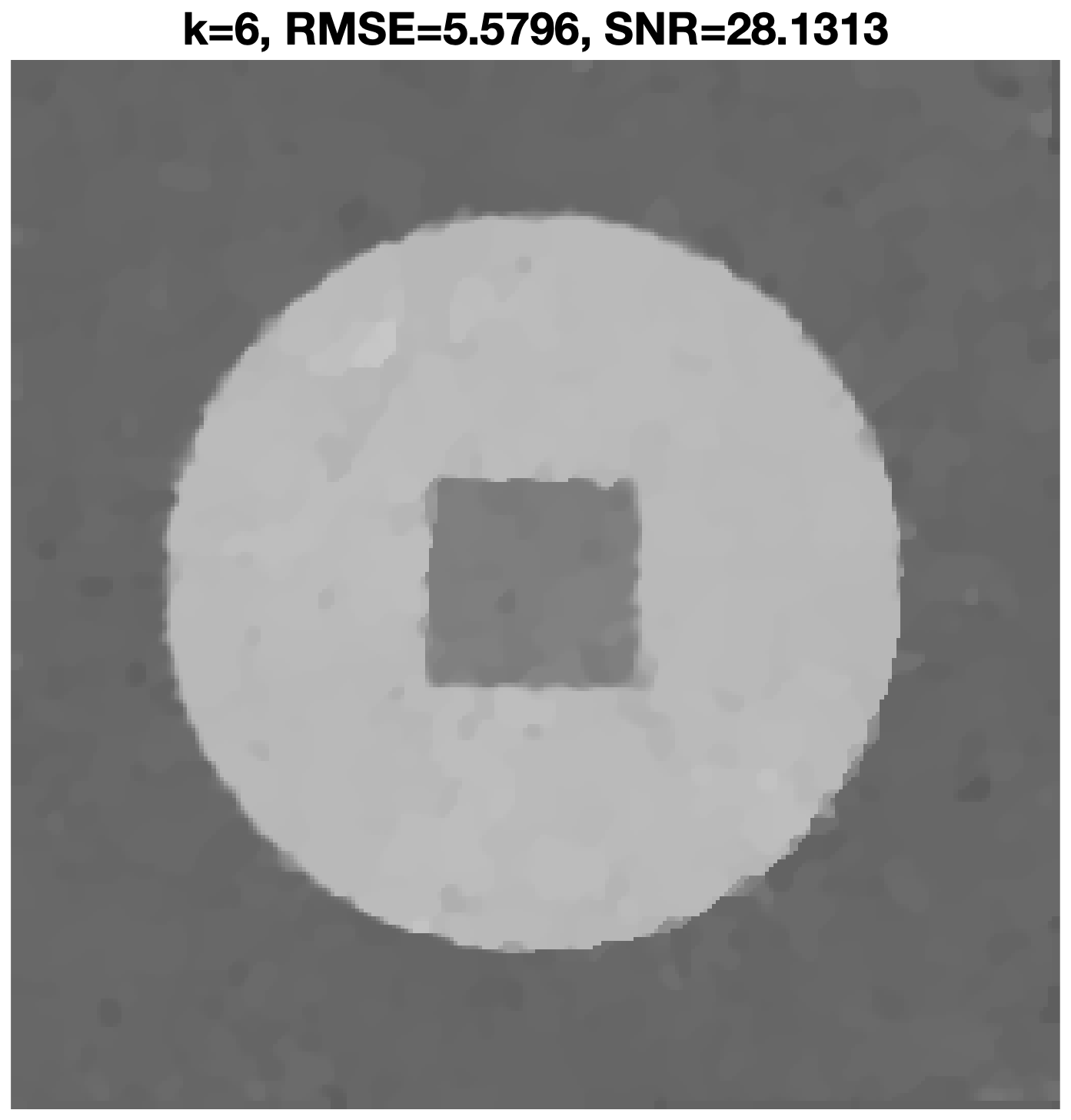}
		\end{subfigure}%
		\caption{AA-log denoised-deblurred images. Rows one, two and three are the regular, tight and refined AA-log MHDM models.} \label{fig:AA_blur}
	\end{figure}%
	\subsection{TNV-log models}
	\subsubsection{Denoising}
	The TNV-log model is an adaptation of the TNV model by using the $TV(\log(u))$ penalty. Recall that TNV is a multiscale procedure based on  \cite{rlo}, which aimed to recover a degraded image $f = u\cdot \eta$ by minimizing $TV(u)$ subject to the constraints $ \int f/u =1$ (mean) and $\int (f/u -1)^2 = \sigma^2$ (variance). Notice that no assumptions are required on the noise distribution (besides its mean and variance), so this method is suitable for more general multiplicative noise restorations. TNV \cite{tad_nez_ves2} dropped the mean constraint and converted the result to a multiscale method. We follow this lead while adding the contribution of a modified penalty. We omit a refined scheme because of the reduced performance from the weaker regularization $\|\log(\cdot)\|_*$, but here is hope, though, that other types of discretizations might provide better results in the refined case. Figure \ref{fig:ARO_restored} shows the TNV-log recoveries. While the method does not outperform the multiscale methods designed for gamma distributed noise, TNV-log performs well against the DZ model on highly textured images.
	
	The paper \cite{auber_aujol_2008} pointed out that empirically the RO model did not preserve the average image intensity, with the recoveries shifting to lighter---higher mean---values. 
	We check the mean intensity of the TNV-log ``Geometry" restorations and observe a similar, yet reduced trend, with only minor upward shifts in the average intensity of both the regular ($+3.3$) and tight ($+5.8$) methods compared with the original image (an 8-bit grayscale image).
	\begin{figure}[h]
		\centering
		\begin{subfigure}{0.24\textwidth}
			\includegraphics[width=\linewidth]{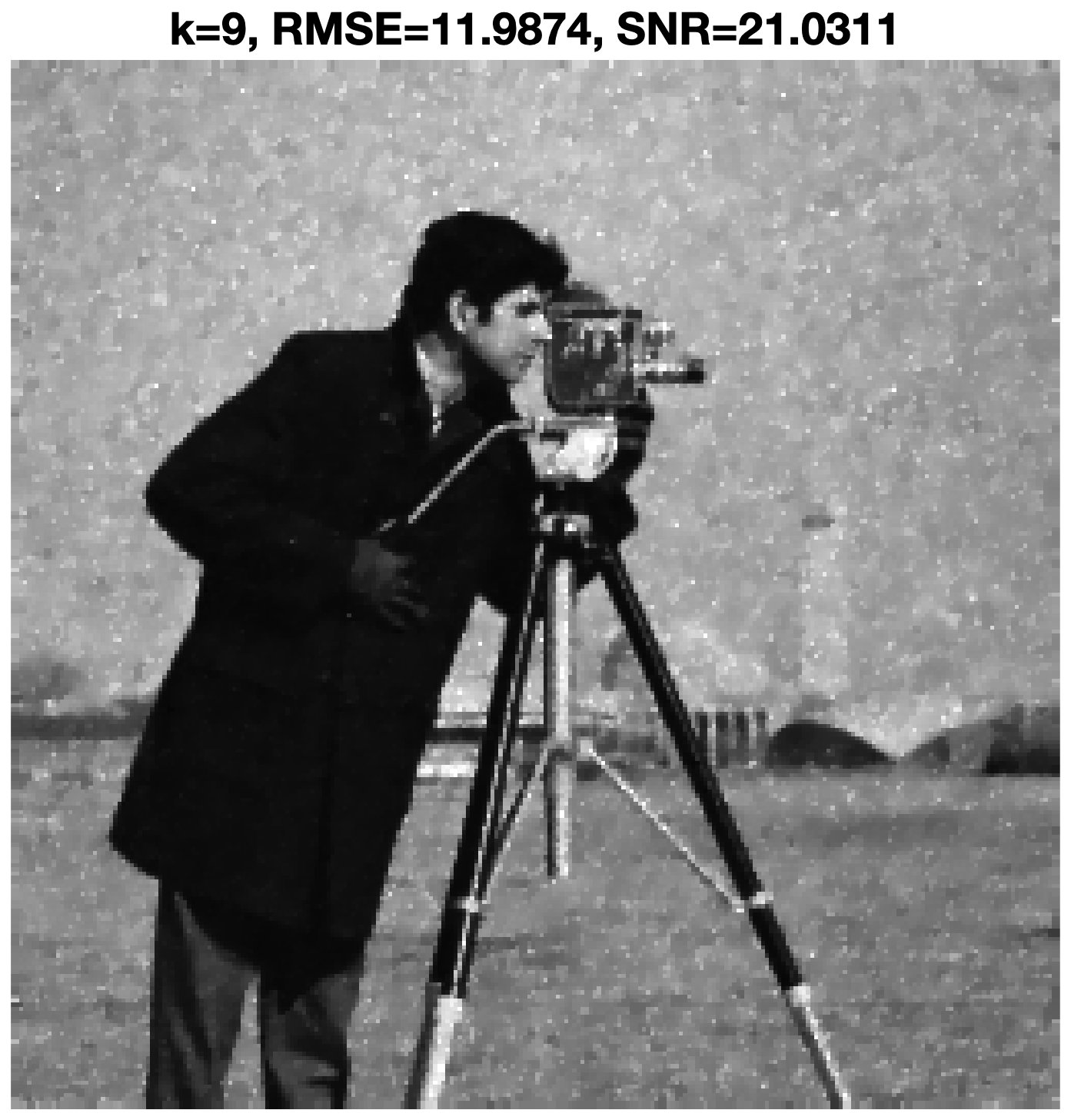}
		\end{subfigure}%
		\hspace*{\fill}  
		\begin{subfigure}{0.24\textwidth}
			\includegraphics[width=\linewidth]{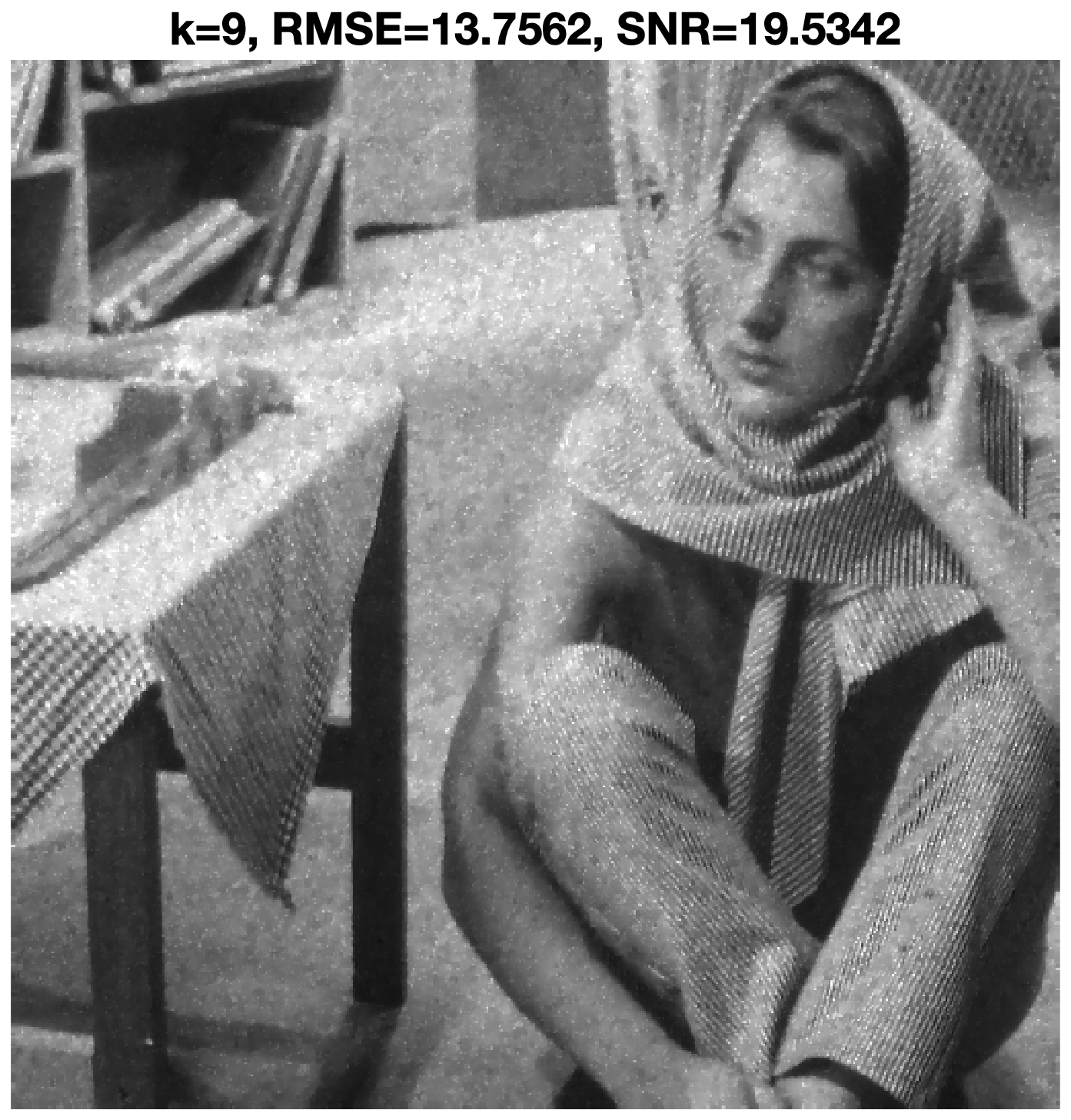}
		\end{subfigure}%
		\hspace*{\fill}   
		\begin{subfigure}{0.24\textwidth}
			\includegraphics[width=\linewidth]{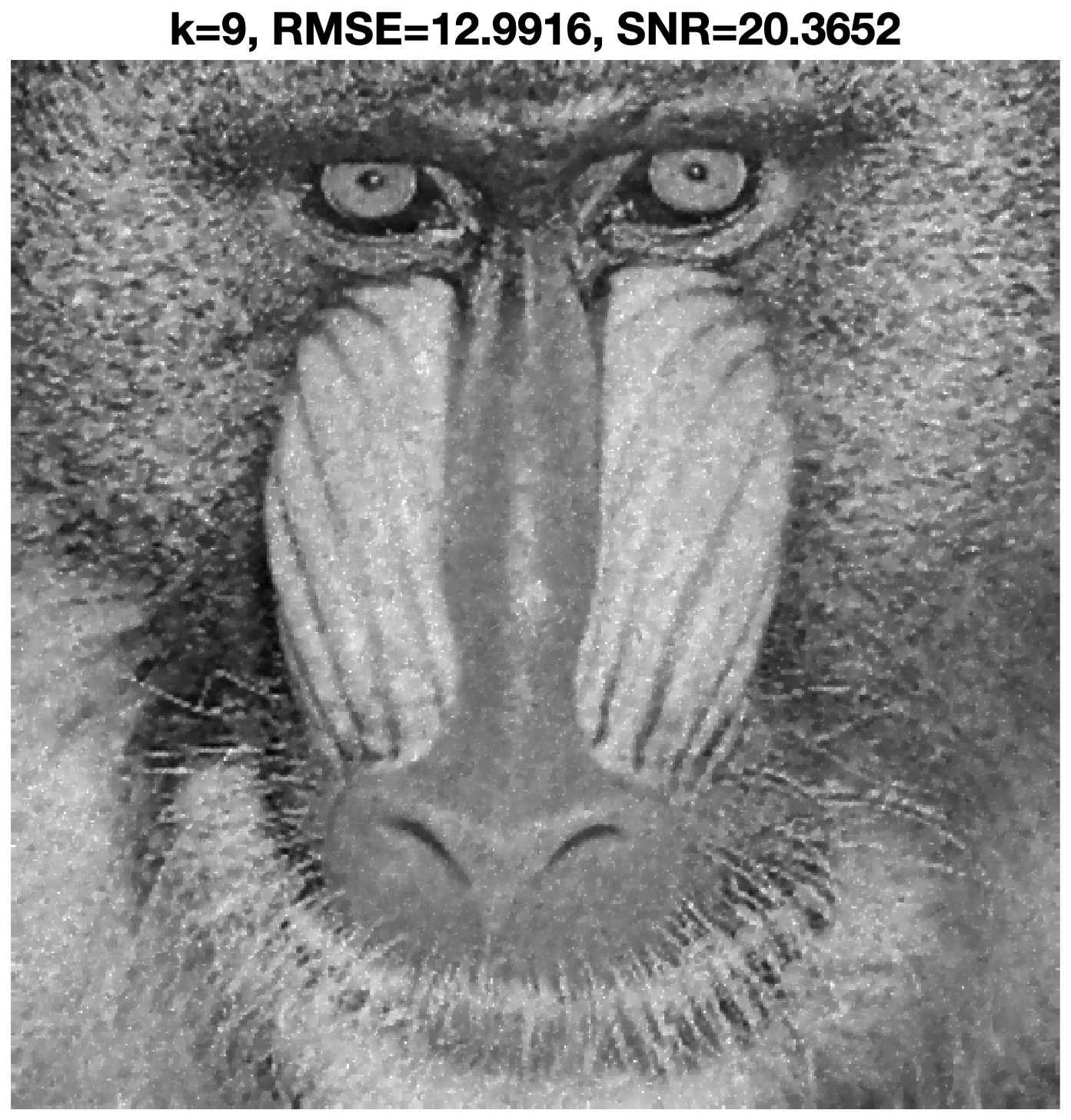}
		\end{subfigure}%
		\hspace*{\fill}
		\begin{subfigure}{0.24\textwidth}
			\includegraphics[width=\linewidth]{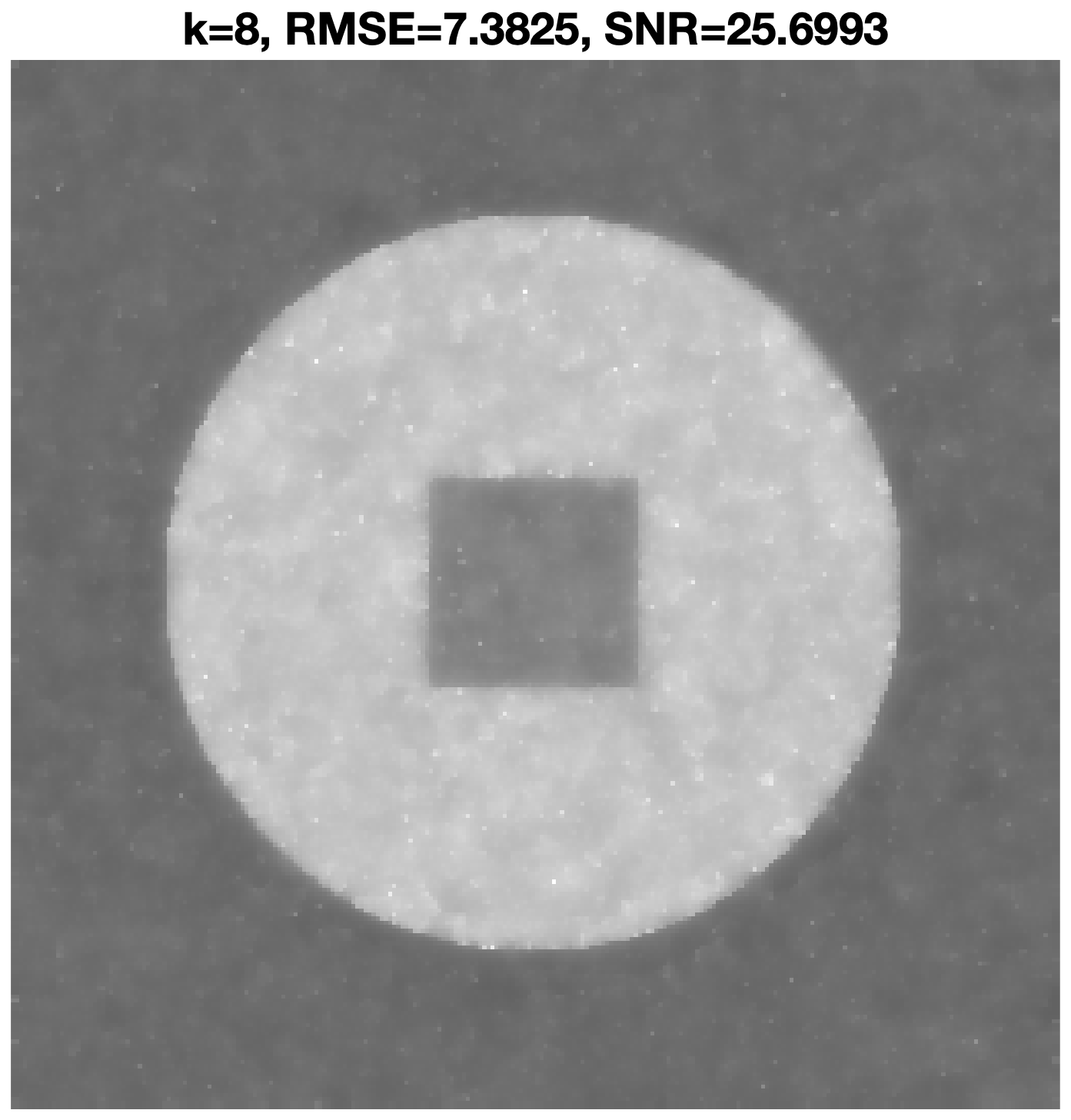}
		\end{subfigure}\\%
		\begin{subfigure}{0.24\textwidth}
			\includegraphics[width=\linewidth]{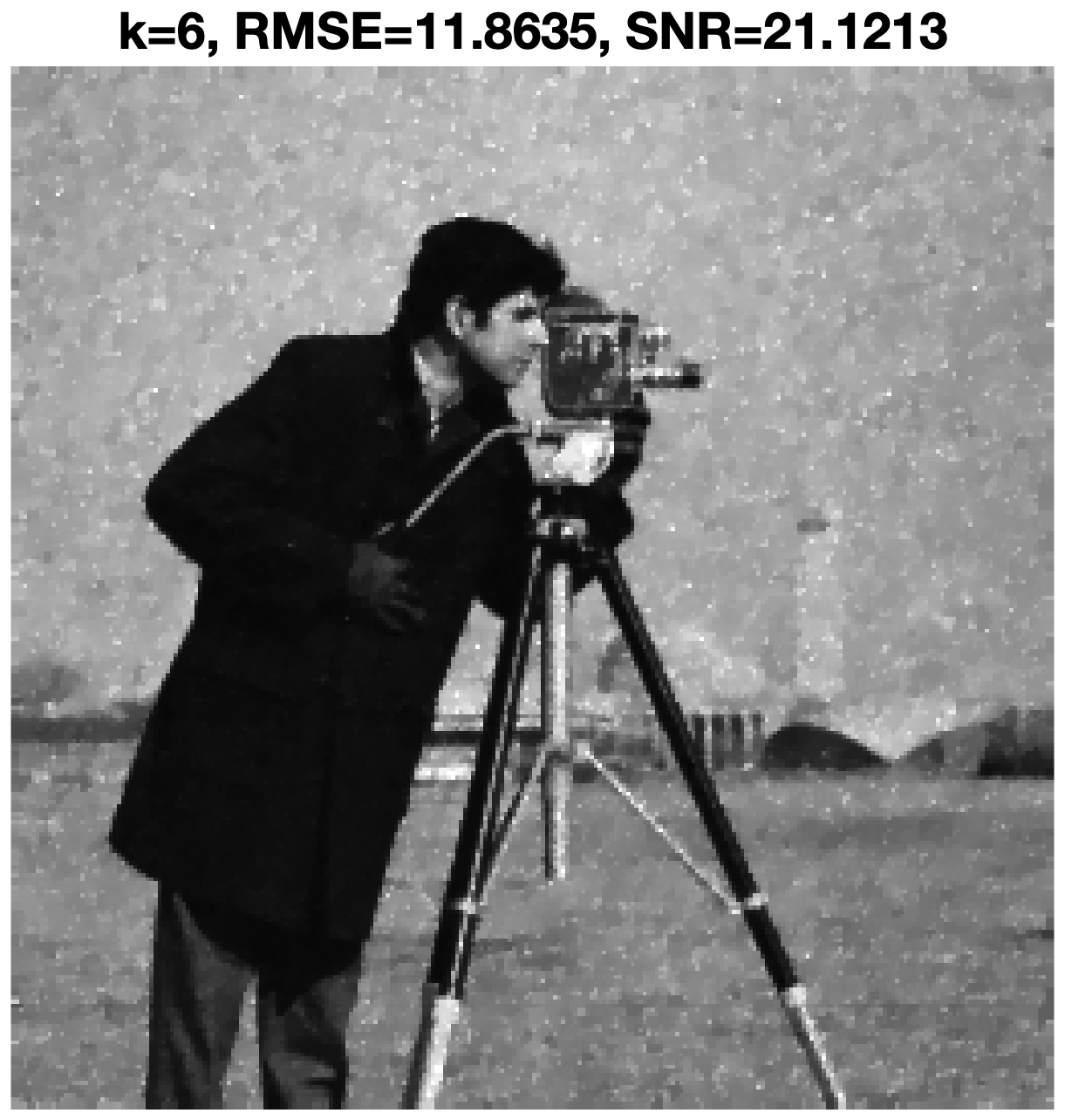}
		\end{subfigure}%
		\hspace*{\fill}  
		\begin{subfigure}{0.24\textwidth}
			\includegraphics[width=\linewidth]{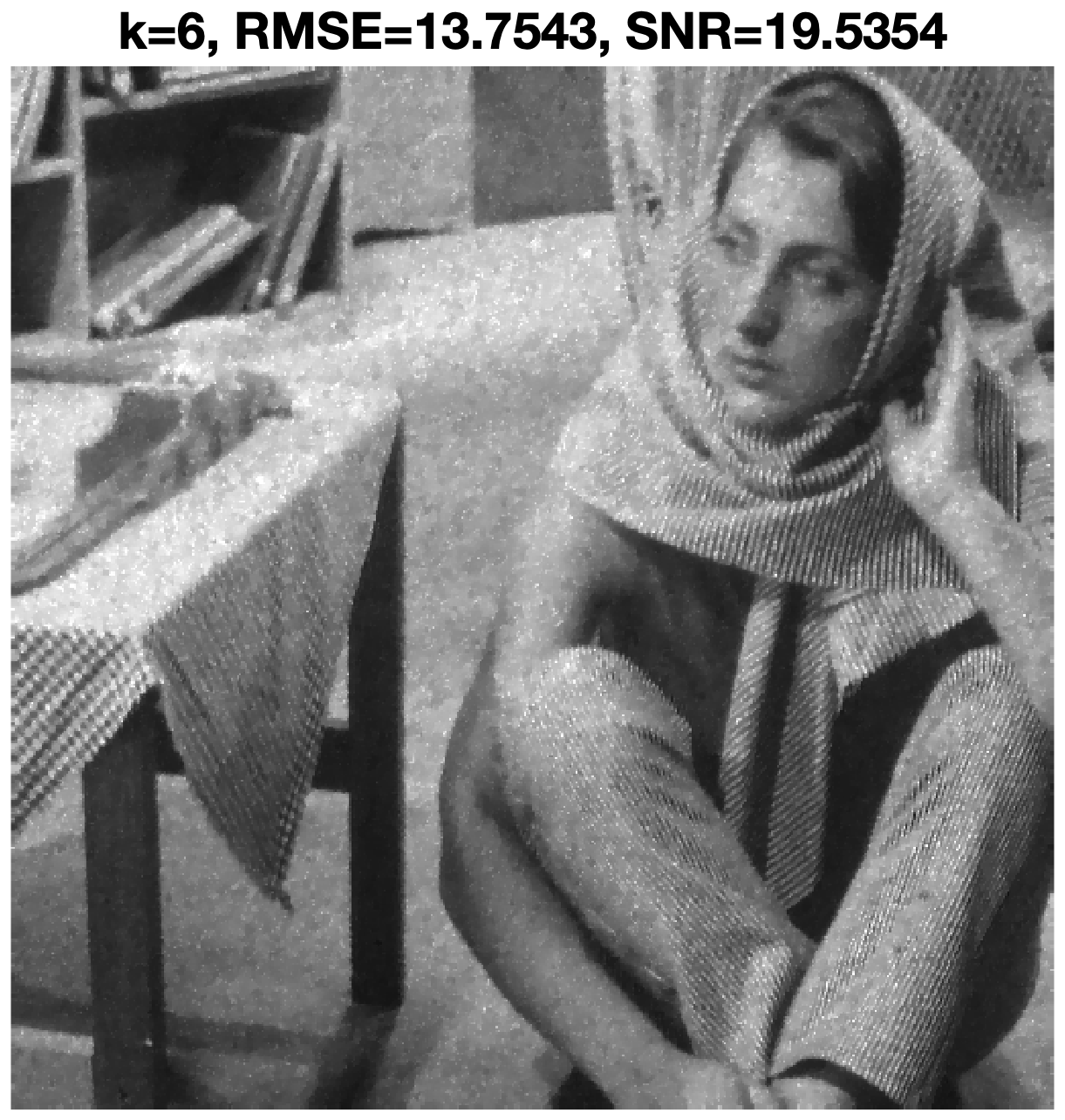}
		\end{subfigure}%
		\hspace*{\fill}   
		\begin{subfigure}{0.24\textwidth}
			\includegraphics[width=\linewidth]{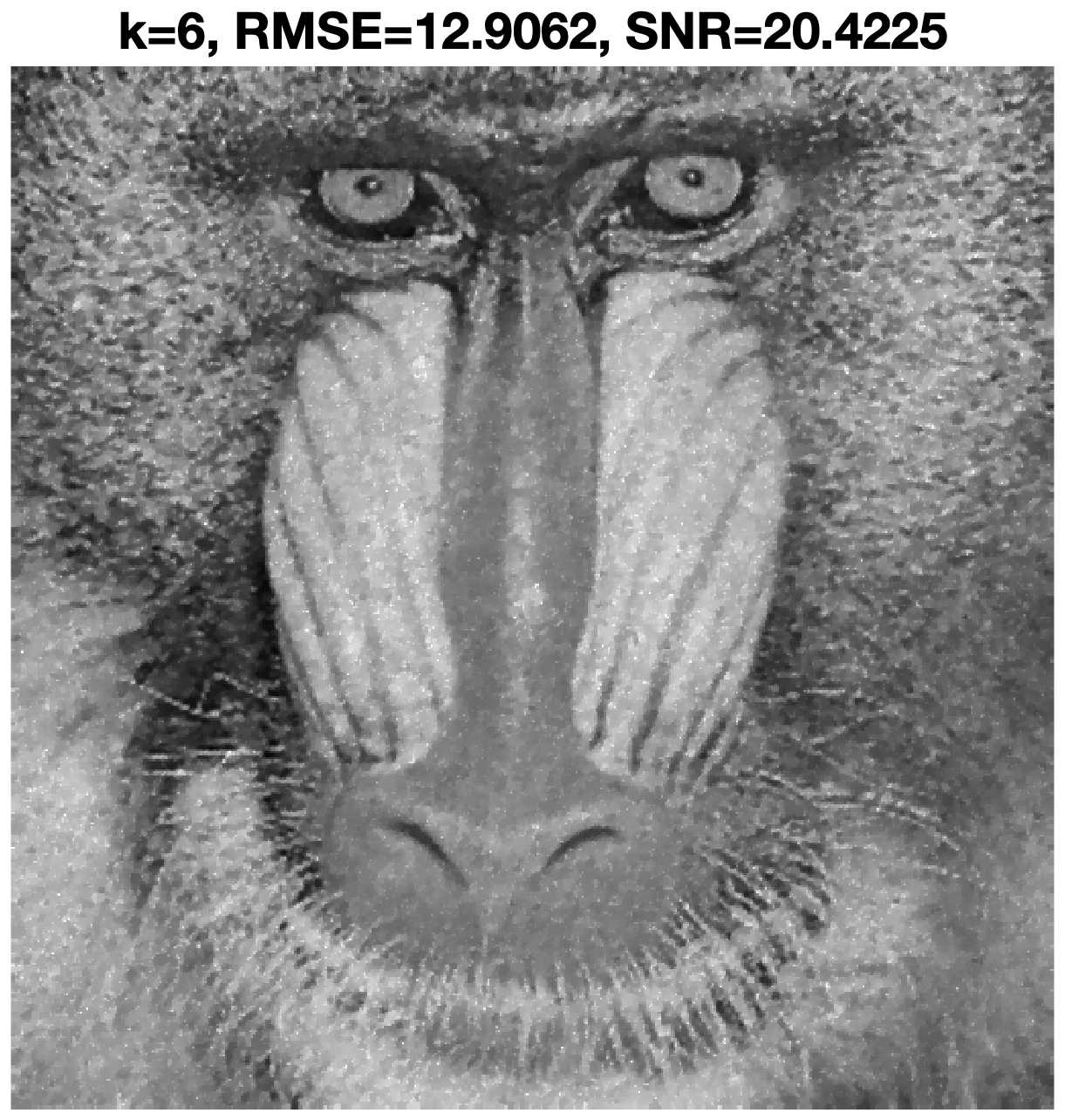}
		\end{subfigure}%
		\hspace*{\fill}
		\begin{subfigure}{0.24\textwidth}
			\includegraphics[width=\linewidth]{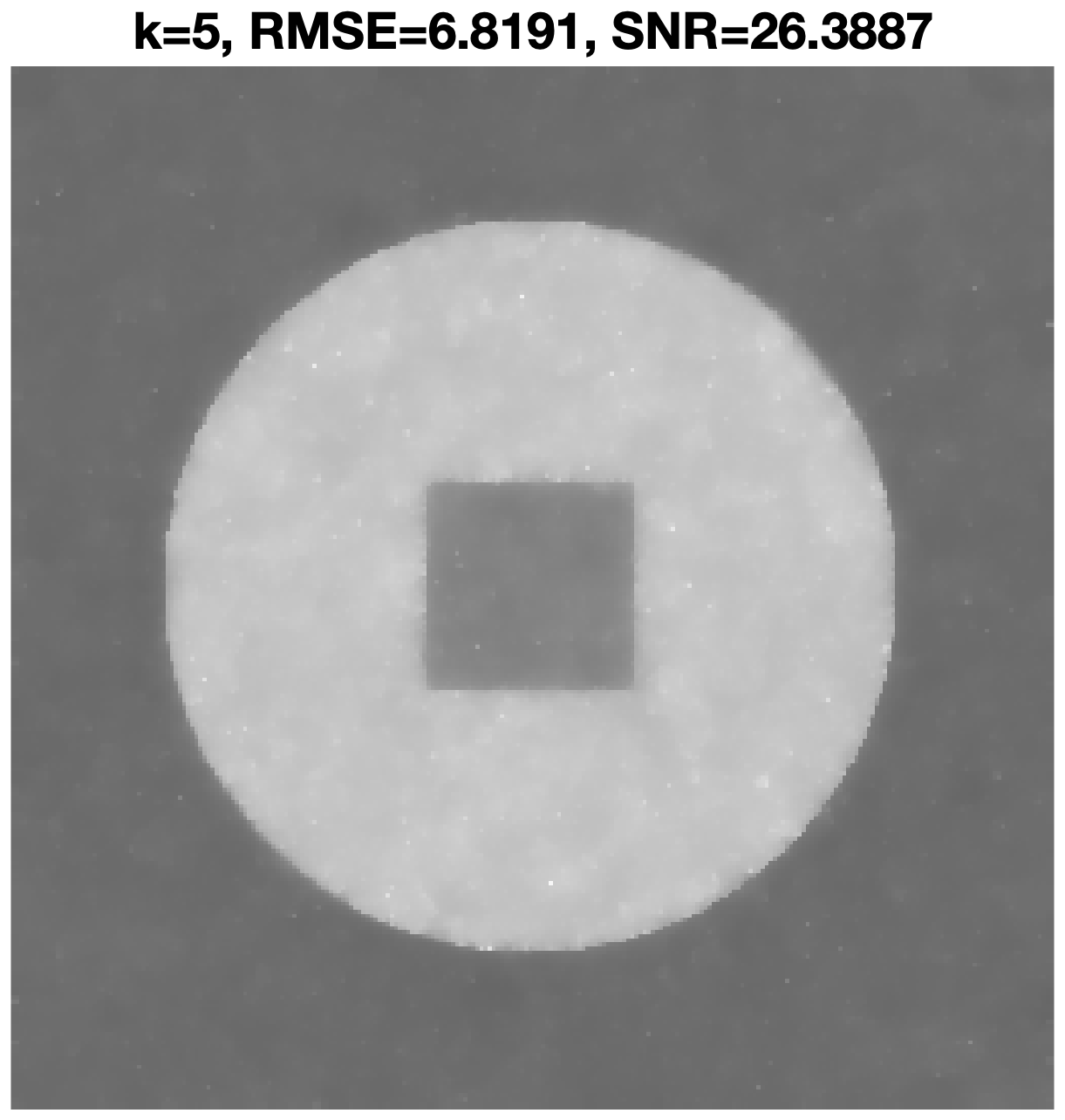}
		\end{subfigure}%
		\caption{TNV-log denoised images. Row one is the regular and row two is the tight recoveries.} \label{fig:ARO_restored}
	\end{figure}%
	
	\subsubsection{Denoising-deblurring}
	The TNV-log model can also tackle denoising-deblurring, as shown in Fig. \ref{fig:ARO_blur_restored}. We see that the performance is not as good as with the AA-log methods. For example, the edges are not that well preserved 
	and the ``Cameraman" image has more noise than the AA-log in Fig.\,\ref{fig:AA_blur}. Yet, this is expected for a technique not tailored for gamma noise.
	\begin{figure}[h]
		\centering
		\begin{subfigure}{0.24\textwidth}
			\includegraphics[width=\linewidth]{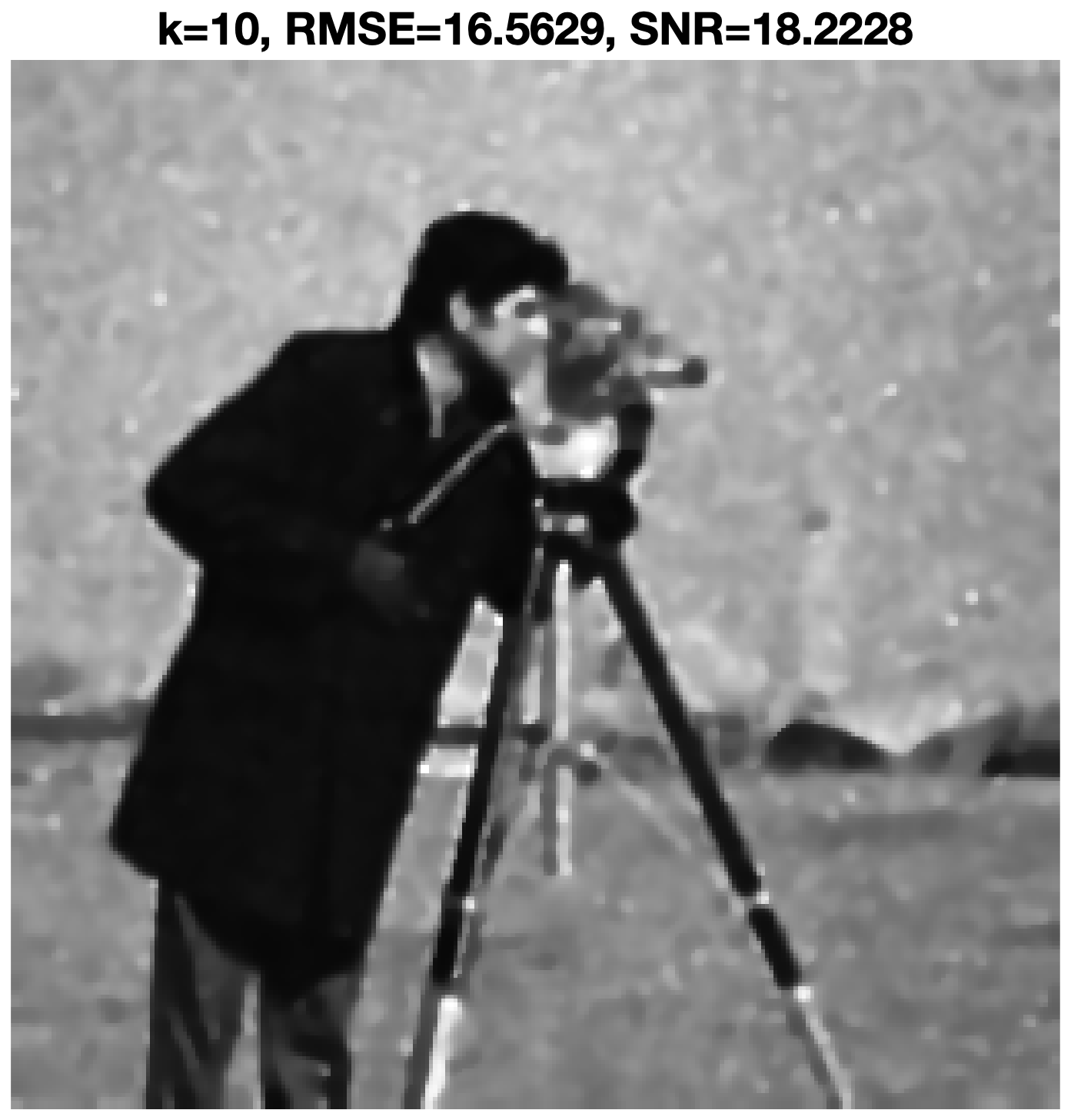}
		\end{subfigure}%
		\hspace*{\fill}%
		\begin{subfigure}{0.24\textwidth}
			\includegraphics[width=\linewidth]{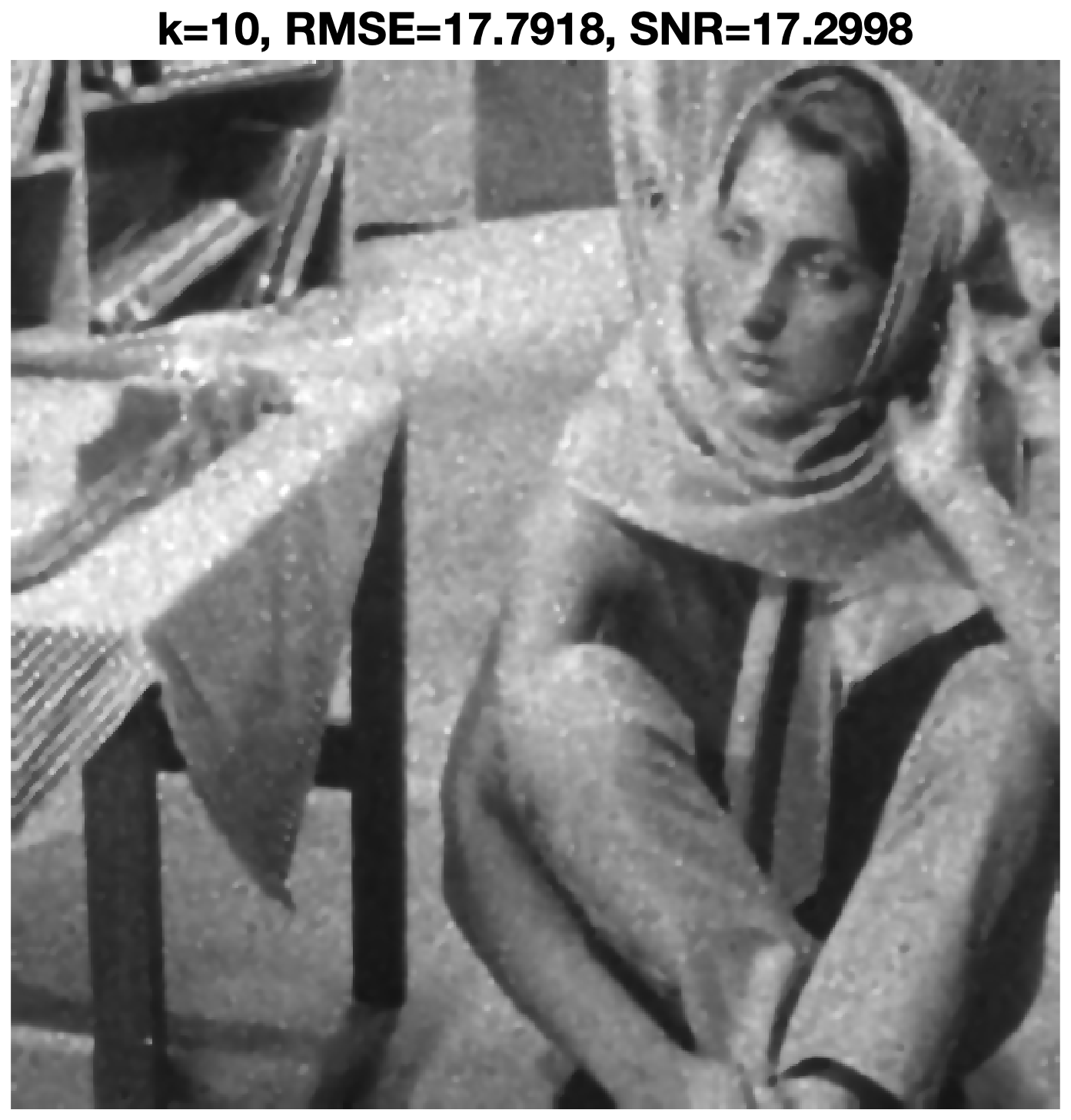}
		\end{subfigure}%
		\hspace*{\fill}
		\begin{subfigure}{0.24\textwidth}
			\includegraphics[width=\linewidth]{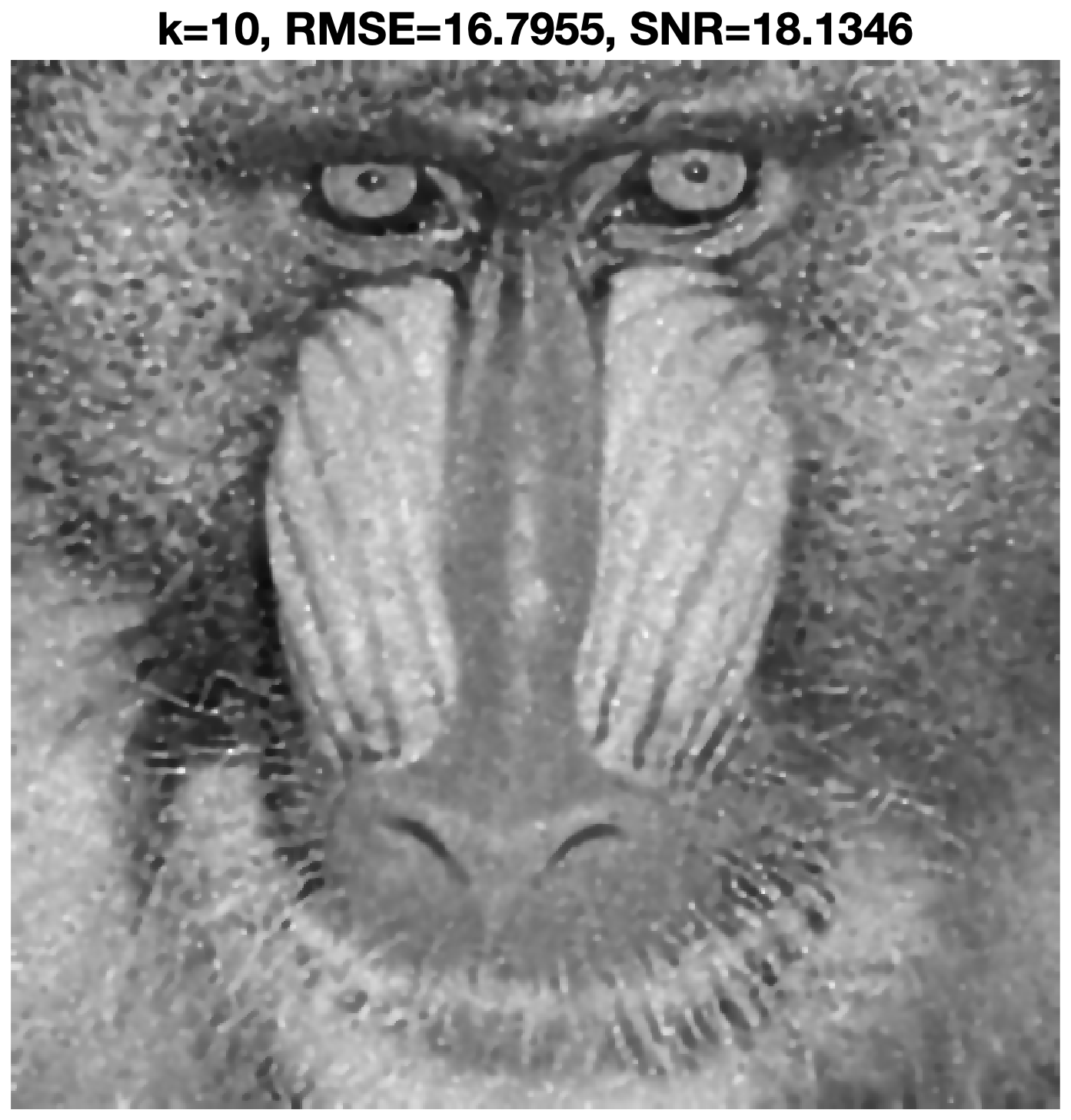}
		\end{subfigure}%
		\hspace*{\fill}%
		\begin{subfigure}{0.24\textwidth}
			\includegraphics[width=\linewidth]{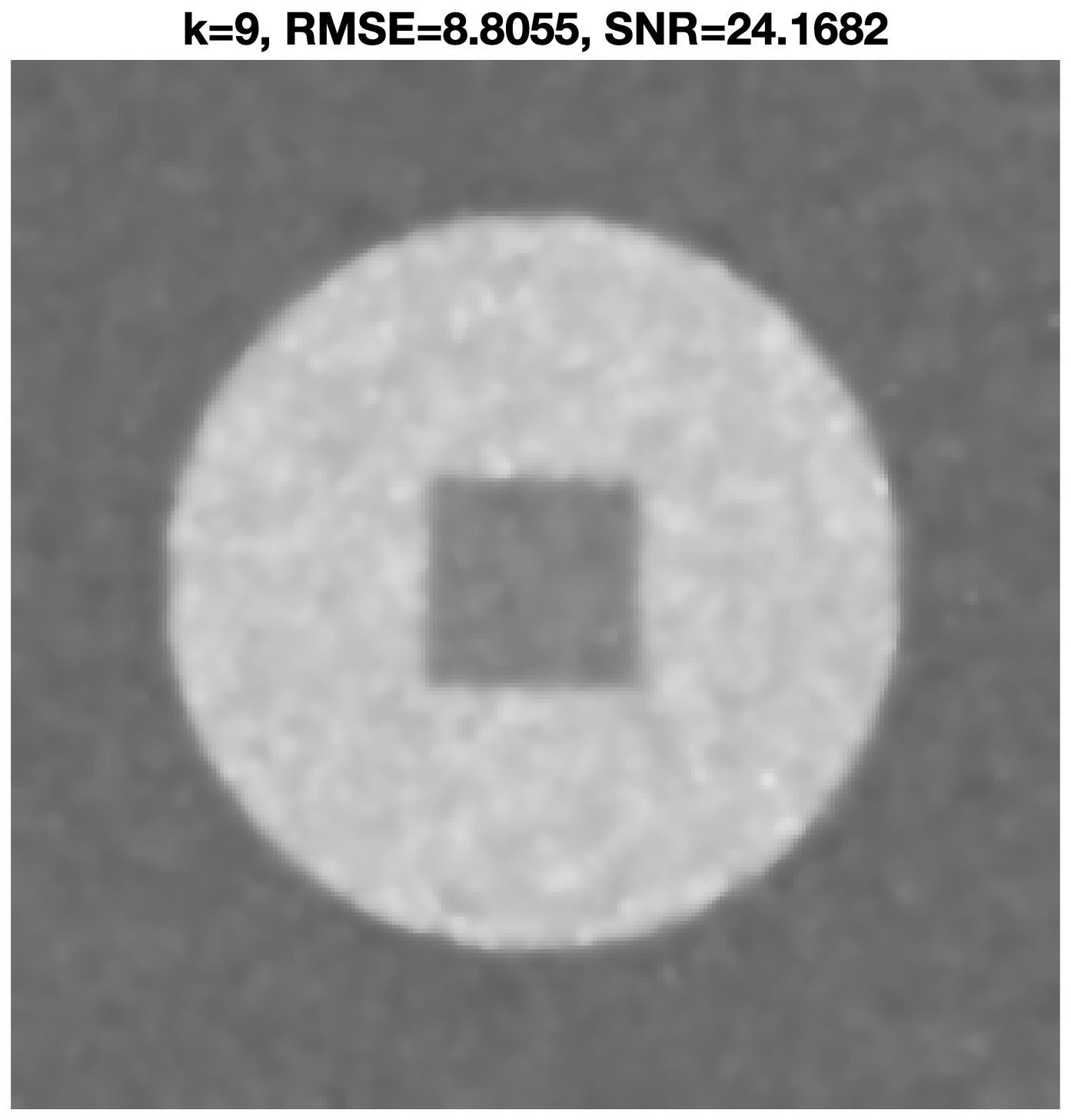}
		\end{subfigure}\\%
		\begin{subfigure}{0.24\textwidth}
			\includegraphics[width=\linewidth]{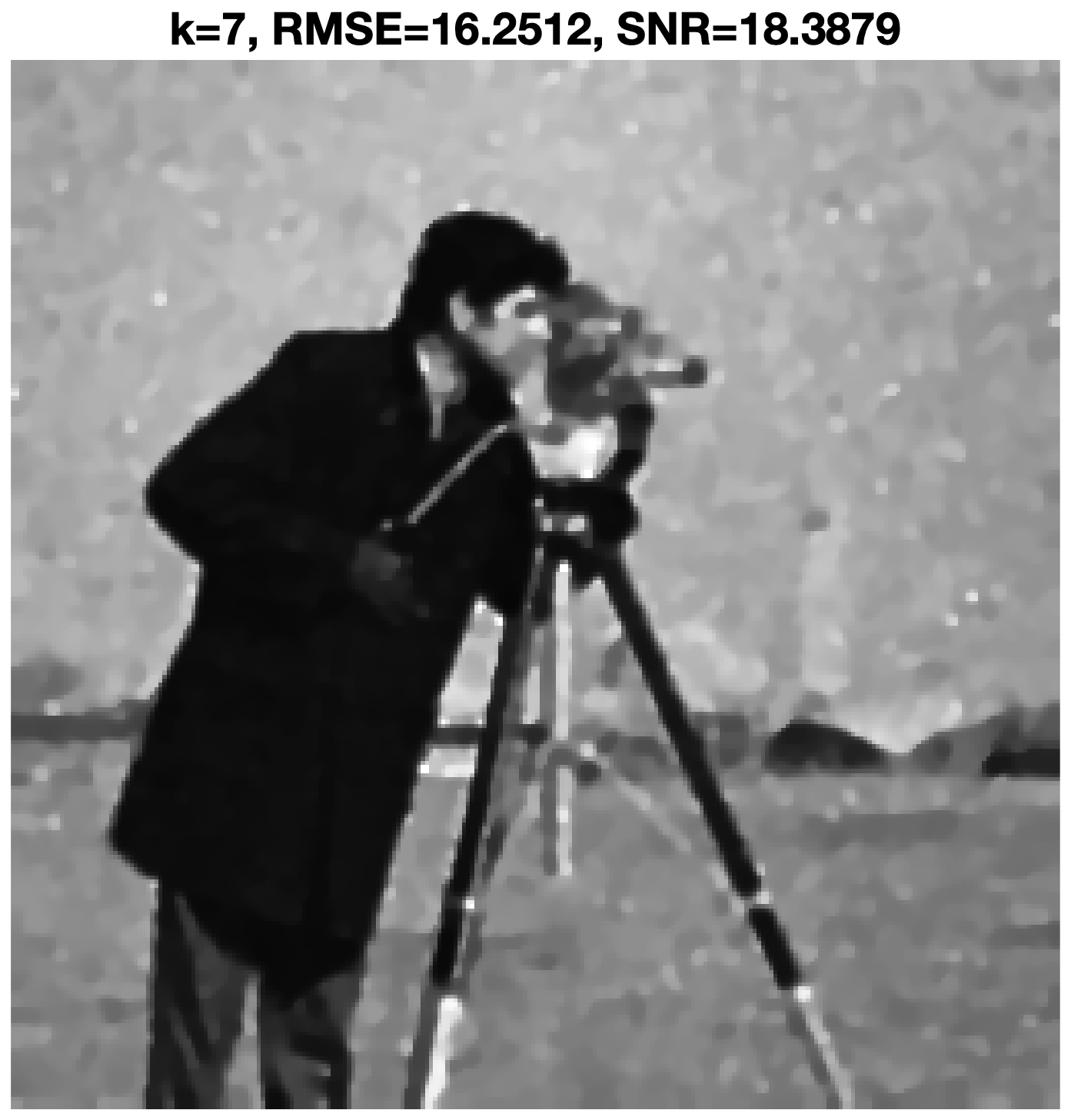}
		\end{subfigure}%
		\hspace*{\fill}%
		\begin{subfigure}{0.24\textwidth}
			\includegraphics[width=\linewidth]{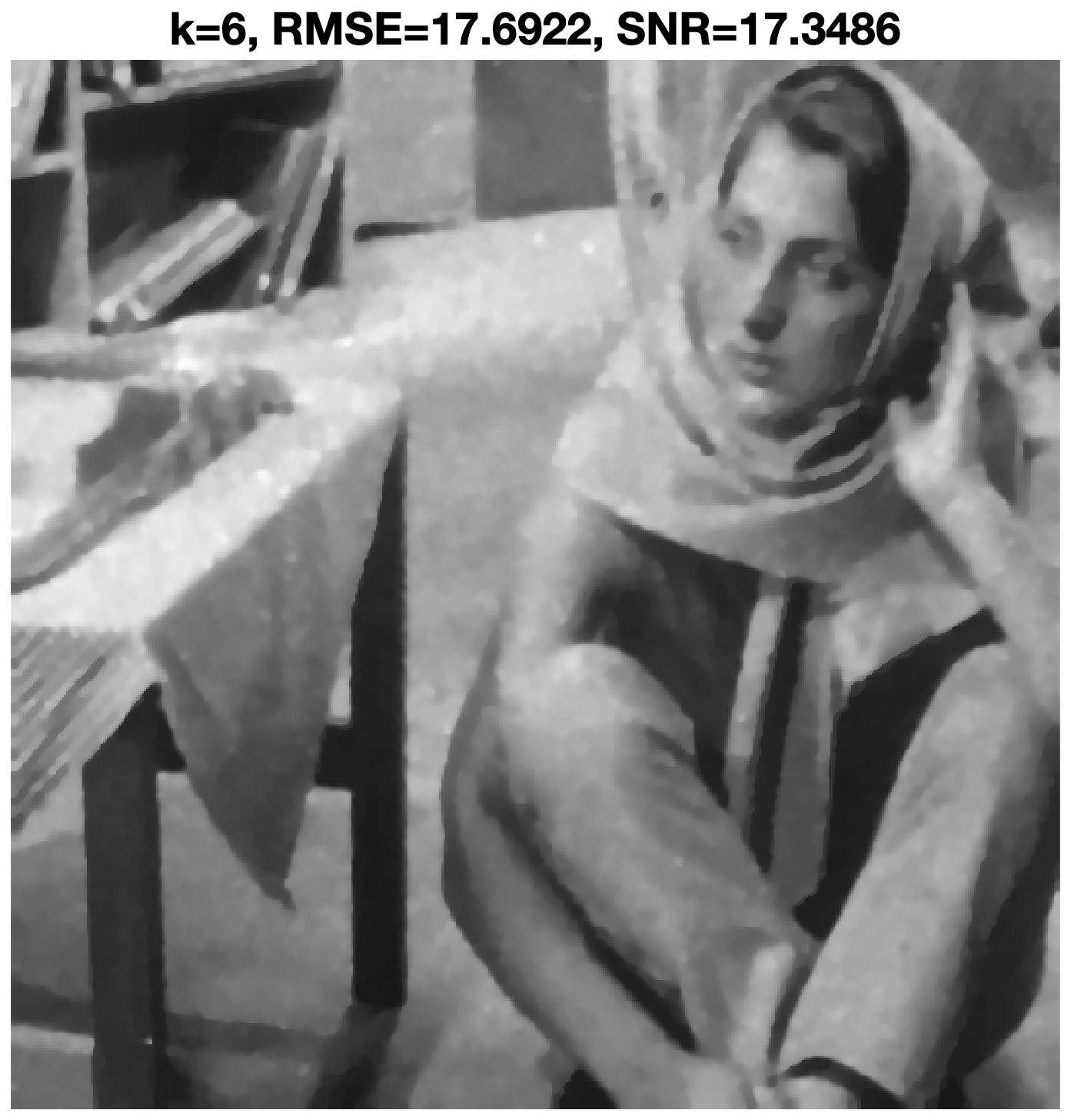}
		\end{subfigure}%
		\hspace*{\fill}
		\begin{subfigure}{0.24\textwidth}
			\includegraphics[width=\linewidth]{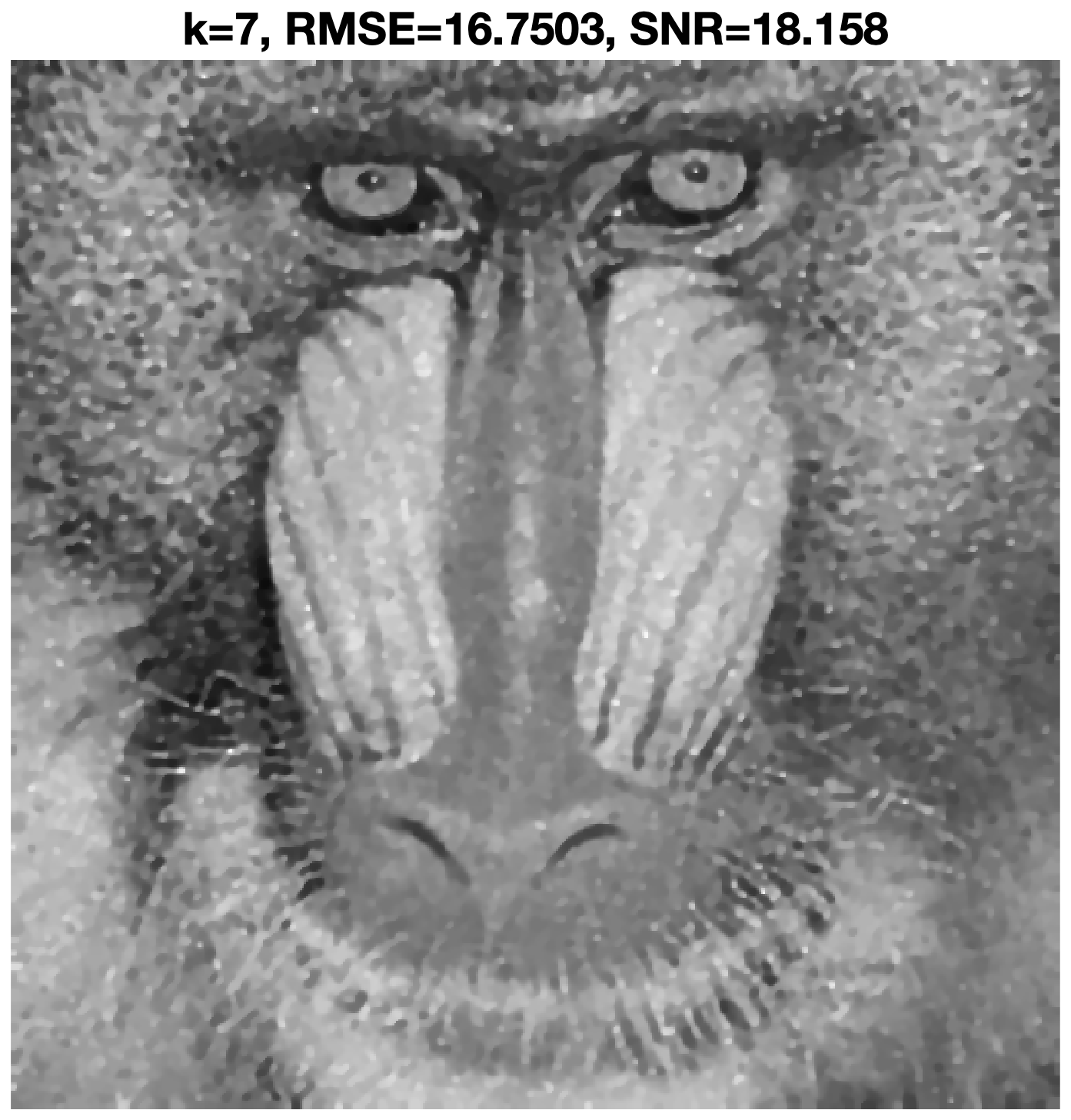}
		\end{subfigure}%
		\hspace*{\fill}%
		\begin{subfigure}{0.24\textwidth}
			\includegraphics[width=\linewidth]{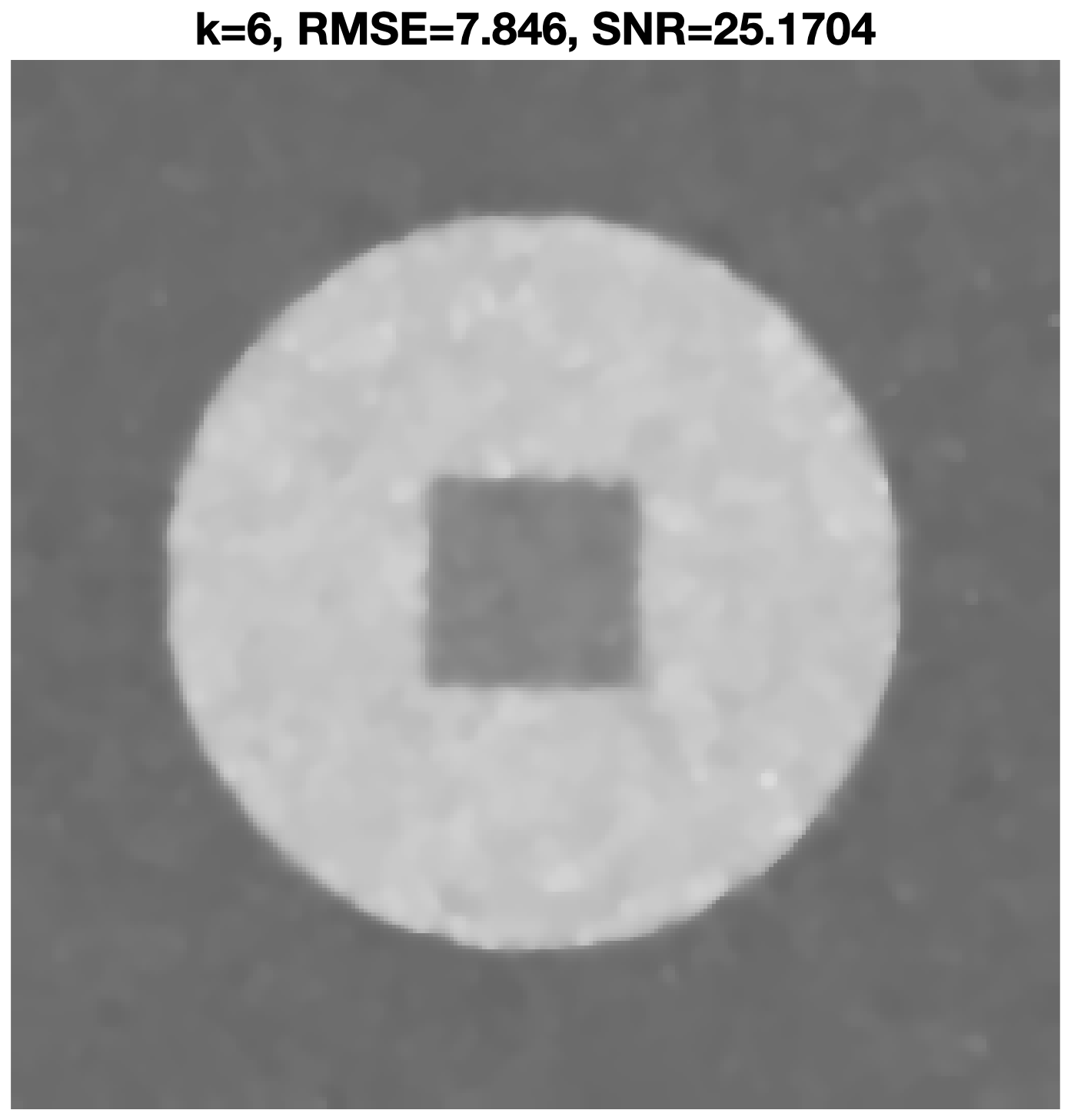}
		\end{subfigure}%
		\caption{TNV-log denoised-deblurred images: regular (row 1) and tight (row 2) recoveries.} \label{fig:ARO_blur_restored}
	\end{figure}%
	\subsection{Comparisons} \label{sec:comparisons}
	Having seen the individual recoveries across the proposed methods, we now compare them against one another, and also with the existing TNV \cite{tad_nez_ves2} and DZ \cite{dong_2013} models for both denoising and deblurring tasks. For DZ, we choose $\alpha = 16$ in \eqref{eq:dong_2013} for all images and select the best weight $\lambda$ by means of a grid search (see Fig.\,\ref{fig:dong_tnv}).
	
	\subsubsection{Sensitivity to Initialization}
	To aid the reader in using the MHDM procedure, we briefly discuss a heuristic for initializing each method  and provide some intuition for these choices. The AA-log MHDM methods are not convex, and consequently initialization can affect recoveries. We find that the penalty minimizing initializations, which generally are smoother (see Subsection \ref{sec:inits}), work well for the regular and tight AA-log MHDM schemes, while the fidelity minimizing initialization $\fd/x_{k-1}$ produces poor restorations and appears to be near a suboptimal local minimum. The refined AA-log MHDM method, however, does well with the less smooth fidelity minimizing initializations, likely because the $*$-norm is finite only for zero-mean functions (see Subsection \ref{sec:inits}).
	
	The SO MHDM methods are convex, and as a result exhibit much less dependence on initialization. All SO methods perform well with both the penalty and fidelity minimizing initializations. We do see a slight improvement following the convention employed for the AA-log MHDM and use this pattern for the results in this work. 
	
	Like the AA-log MHDM methods, the TNV-log schemes are not convex. We observe a similar trend to the AA-log method's dependence on initialization, preferring the smoother penalty minimizing initialization for the TNV-log regular and tight schemes. 
	\subsubsection{Denoising}
	We present the TNV and the DZ recoveries in Fig.\,\ref{fig:dong_tnv}. The DZ model recovers the smoother, more cartoon image features well, but fails to capture the textural details in the way a multiscale method like the TNV model does. However, the latter suffers from restoring perhaps too much noise within smoother images. 
	%
	\begin{figure}[h!]
		\centering
		\begin{subfigure}{0.24\textwidth}
			\includegraphics[width=\linewidth]{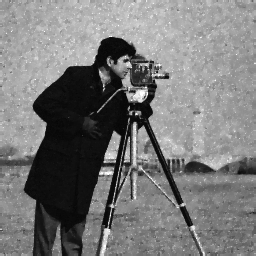}
			\caption{RMSE=11.86, SNR=21.12}%
		\end{subfigure}%
		\hspace*{\fill}%
		\begin{subfigure}{0.24\textwidth}
			\includegraphics[width=\linewidth]{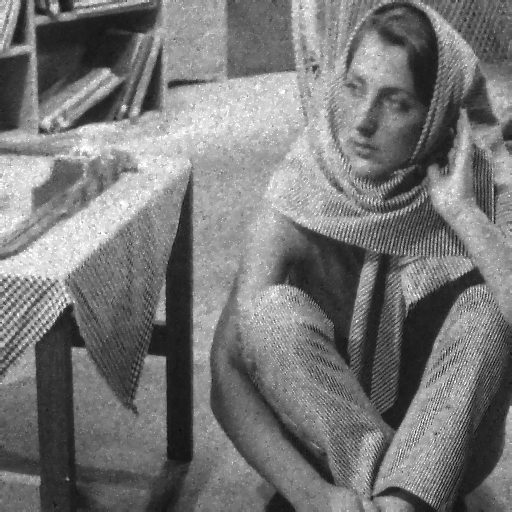}
			\caption{RMSE=13.46, SNR=19.73}%
		\end{subfigure}%
		\hspace*{\fill}%
		\begin{subfigure}{0.24\textwidth}
			\includegraphics[width=\linewidth]{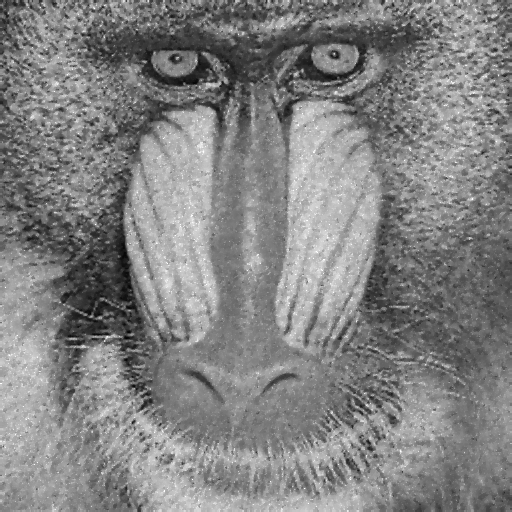}
			\caption{RMSE=12.93, SNR=20.40}%
		\end{subfigure}%
		\hspace*{\fill}
		\begin{subfigure}{0.24\textwidth}
			\includegraphics[width=\linewidth]{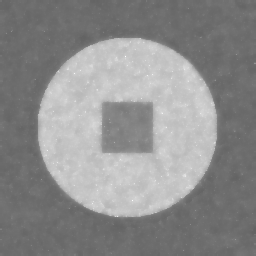}
			\caption{RMSE=7.183, SNR=25.94}%
		\end{subfigure}\\%
		\begin{subfigure}{0.24\textwidth}
			\includegraphics[width=\linewidth]{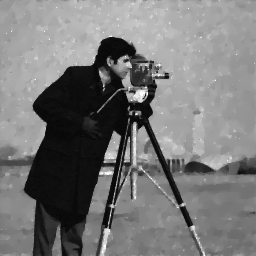}
			\caption{RMSE=10.87, SNR=21.87}
		\end{subfigure}%
		\hspace*{\fill}%
		\begin{subfigure}{0.24\textwidth}
			\includegraphics[width=\linewidth]{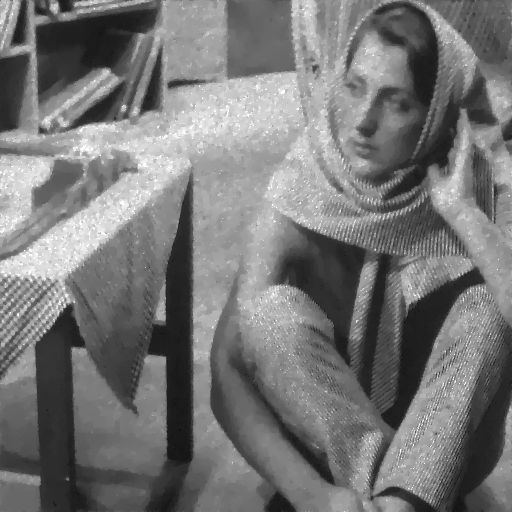}
			\subcaption{RMSE=14.37, SNR=19.13}
		\end{subfigure}%
		\hspace*{\fill}
		\begin{subfigure}{0.24\textwidth}
			\includegraphics[width=\linewidth]{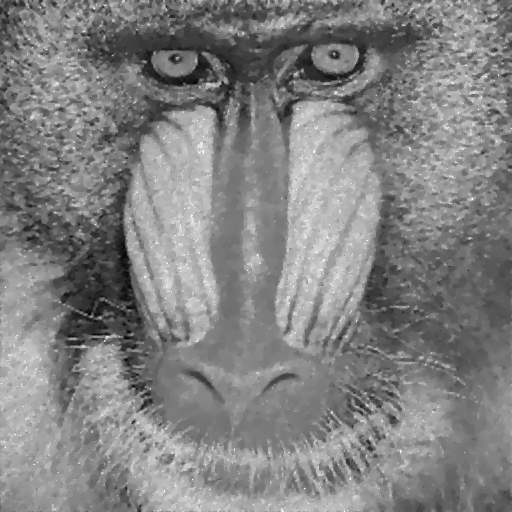}
			\caption{RMSE=13.16, SNR=20.22}
		\end{subfigure}%
		\hspace*{\fill}%
		\begin{subfigure}{0.24\textwidth}
			\includegraphics[width=\linewidth]{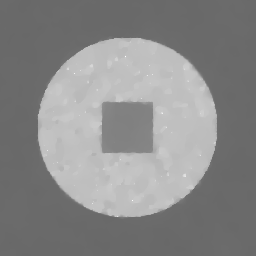}
			\caption{RMSE=3.87, SNR=31.32} 
		\end{subfigure}%
		\caption{Row one: TNV recoveries at $k_{min} = 9$ for all images except ``Geometry", for which $k_{min} = 8$. Row two: DZ model denoised images. Parameters for the DZ model recoveries (refer to \eqref{eq:dong_2013}): (Cameraman) $\lambda = 0.06$, (Barbara) $\lambda = 0.05$, (Mandril) $\lambda = 0.05$, and (Geometry) $\lambda = 0.13$. For all images, $\alpha = 16$.} \label{fig:dong_tnv}
	\end{figure}%
	We examine our proposed MHDM schemes against the TNV and DZ models. To easily compare across all models, we list the SNR values of the denoising restorations at $k_{min}$ and $k^*$ in Table \ref{tab:SNRcomparisons}. We note that the refined SO MHDM recovery performs best on images with more detail and texture (``Barbara" and ``Mandril"), while the ADMM schemes recover best those with larger smooth regions (``Cameraman" and ``Geometry"). 
	
	Figure \ref{fig:comparisons} gives a detailed crop of each method's recovery for the test images. The multiscale reconstruction's ability to recover greater texture is clear in the ``Barbara" recovery, where one can see that the DZ recovery flattens and mutes the details. The SO MHDM (ADMM) schemes are very effective at removing noise in smooth regions (``Geometry"). In a more textured image such as the ``Mandril", we remark the importance of the tight adjustment preventing over-smoothing, as for example, is seen with the ``Mandril" image's nose for SO MHDM (ADMM).
	
	For greater comparison, we also note that the MHDM schemes handle higher noise levels quite well ($g(x;10)$, standard deviation 0.32), as shown for a few methods in Fig.\,\ref{fig:mid_noise_compare}, in comparison with the DZ recovery. Importantly, $SNR$ and $RMSE$ are improved, 
	and texture is retained throughout the image when compared with the DZ method.
	\begin{figure}[h]
		\centering
		\begin{subfigure}{0.24\textwidth}
			\includegraphics[width=\linewidth]{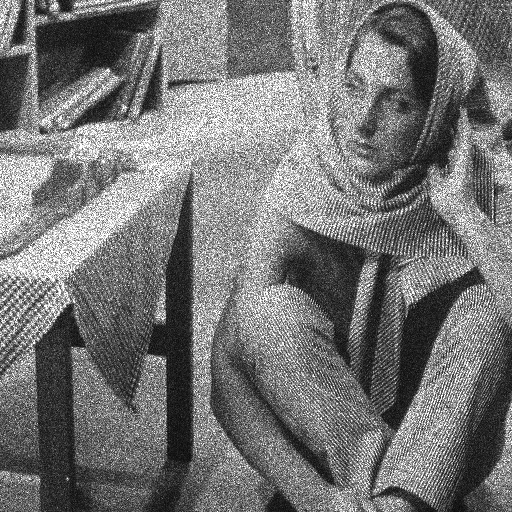}
		\end{subfigure}%
		\hspace*{\fill}%
		\begin{subfigure}{0.24\textwidth}
			\includegraphics[width=\linewidth]{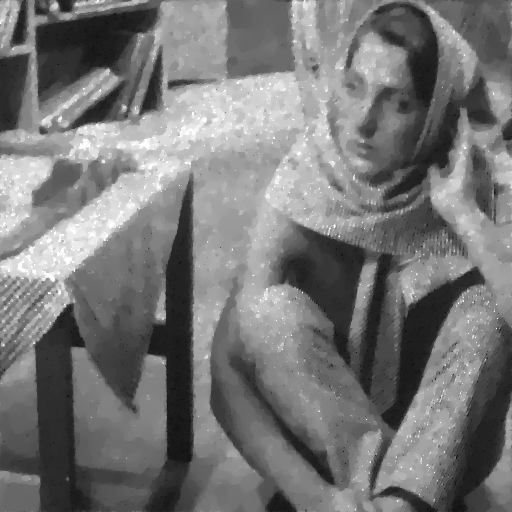}
		\end{subfigure}%
		\hspace*{\fill}%
		\begin{subfigure}{0.24\textwidth}
			\includegraphics[width=\linewidth]{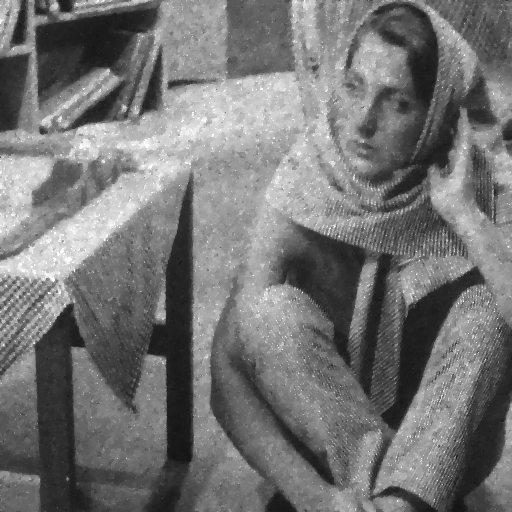}
		\end{subfigure}%
		\hspace*{\fill}%
		\begin{subfigure}{0.24\textwidth}
			\includegraphics[width=\linewidth]{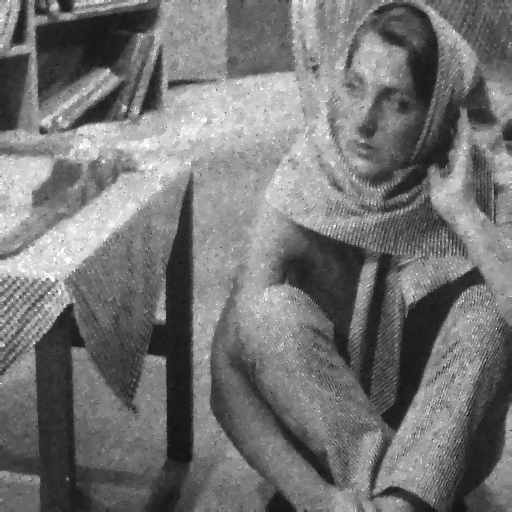}
		\end{subfigure}%
		\caption{Recoveries from high noise (standard deviation 0.32). From left to right: noisy image ($RMSE = 41.25$, $SNR = 10.00$), DZ model ($RMSE = 17.5233$, $SNR= 17.44 $, $\lambda = 0.08$), the SO MHDM tight ($k_{min} =6  $, $RMSE = 16.37$, $SNR = 18.03 $) and AA-log MHDM tight recoveries ($k_{min} = 6$,$RMSE = 16.38 $,  $SNR = 18.02$).}
		\label{fig:mid_noise_compare}
	\end{figure}
	\begin{table}[htbp]
		\centering
		\caption{SNR values from various denoising recoveries at the minimizing indices $k_{min}$ and the stopping criteria $k^*(\delta)$. Bold entries are the maximum of their respective columns.}
		\begin{tabular}{c|cccccccc}
			&\multicolumn{2}{c}{Cameraman}& \multicolumn{2}{c}{Barbara} & \multicolumn{2}{c}{Mandril} & \multicolumn{2}{c}{Geometry} \\
			SNR at & $k_{min}$& $k^*$ & $k_{min}$ & $k^*$ & $k_{min}$ & $k^*$ & $k_{min}$ & $k^*$ \\\hline
			SO MHDM (EL)& 21.89 & 21.89 & 19.70 & 19.08 & 20.38 & 20.06 & 28.21 & 24.82\\
			SO Tight (EL) & 21.94 & 21.94 & 19.71 & 18.98 & 20.37 & 19.89 & 29.46 & 27.50 \\
			SO Refined (EL) & 22.25 & 22.05 & \textbf{19.88} & 19.33 & \textbf{20.73} & 20.32 & 30.11 & 29.66 \\
			SO MHDM (ADMM) & 22.16 & \textbf{22.16} & 19.31 & 19.12 & 19.93 & 19.93 & \textbf{34.67} & 31.53 \\
			SO Tight (ADMM)& \textbf{22.31} & 21.97 & 19.74 & \textbf{19.36} & 20.46 & 20.19 & 34.60 & \textbf{34.60}\\
			AA MHDM & 21.76 & 19.45 & 19.53 & 19.14 & 20.21 & 20.05 & 28.21 & 24.77\\
			AA-log MHDM & 21.74 & 21.74 & 19.65 & 19.04 & 20.35 & 20.02 & 28.21 & 24.83 \\
			AA-log Tight & 21.48 & 21.48 & 19.67 & 18.93 & 20.37 & 19.79 & 30.22 & 26.11 \\
			AA-log Refined & 21.68 & 21.19 & 19.64 & 19.21 & 20.55 & \textbf{20.41} & 27.12 & 13.81 \\
			TNV-log & 21.03 & 20.13 & 19.53 & 17.97 & 20.37 & 19.02 & 25.70 & 23.17 \\
			TNV-log Tight & 21.12 & 12.61 & 19.54 & 17.35 & 20.42 & 17.89 & 26.39 & 21.02\\\hline
			TNV   & 21.12 & & 19.73  & & 20.40 & & 25.94 &\\
			DZ  & 21.87 & & 19.13 & & 20.22 & & 31.32 &\\
		\end{tabular}%
		\label{tab:SNRcomparisons}%
	\end{table}%
	\newlength{\figw}
	\setlength{\figw}{0.13\textwidth}
	\begin{figure}[h!]
		\centering
		\begin{subfigure}{\figw}
			\includegraphics[width=\linewidth]{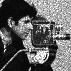}
		\end{subfigure}%
		\begin{subfigure}{\figw}
			\includegraphics[width=\linewidth]{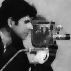}
		\end{subfigure}%
		\begin{subfigure}{\figw}
			\includegraphics[width=\linewidth]{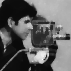}
		\end{subfigure}%
		\begin{subfigure}{\figw}
			\includegraphics[width=\linewidth]{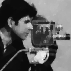}
		\end{subfigure}%
		\begin{subfigure}{\figw}
			\includegraphics[width=\linewidth,cframe=black 1.5pt]{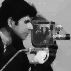}\ %
		\end{subfigure}
		\begin{subfigure}{\figw}
			\includegraphics[width=\linewidth]{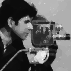}
		\end{subfigure}%
		\begin{subfigure}{\figw}
			\includegraphics[width=\linewidth]{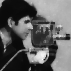}
		\end{subfigure}\\%
		\begin{subfigure}{\figw}
			\includegraphics[width=\linewidth]{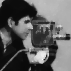}
		\end{subfigure}%
		\begin{subfigure}{\figw}
			\includegraphics[width=\linewidth]{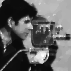}
		\end{subfigure}%
		\begin{subfigure}{\figw}
			\includegraphics[width=\linewidth]{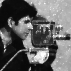}
		\end{subfigure}%
		\begin{subfigure}{\figw}
			\includegraphics[width=\linewidth]{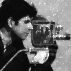}
		\end{subfigure}%
		\begin{subfigure}{\figw}
			\includegraphics[width=\linewidth]{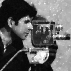}
		\end{subfigure}%
		\begin{subfigure}{\figw}
			\includegraphics[width=\linewidth]{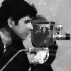}
		\end{subfigure}%
		\begin{subfigure}{\figw}
			\includegraphics[width=\linewidth]{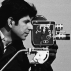}
		\end{subfigure}\\%
		\begin{subfigure}{\figw}
			\includegraphics[width=\linewidth]{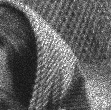}
		\end{subfigure}%
		\begin{subfigure}{\figw}
			\includegraphics[width=\linewidth]{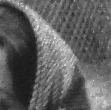}
		\end{subfigure}%
		\begin{subfigure}{\figw}
			\includegraphics[width=\linewidth]{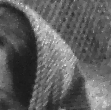}
		\end{subfigure}%
		\begin{subfigure}{\figw}
			\includegraphics[width=\linewidth, cframe=black 1.5pt]{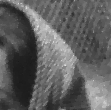}
		\end{subfigure}\ %
		\begin{subfigure}{\figw}
			\includegraphics[width=\linewidth]{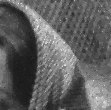}
		\end{subfigure}%
		\begin{subfigure}{\figw}
			\includegraphics[width=\linewidth]{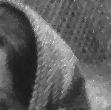}
		\end{subfigure}%
		\begin{subfigure}{\figw}
			\includegraphics[width=\linewidth]{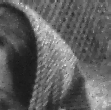}
		\end{subfigure}\\%
		\begin{subfigure}{\figw}
			\includegraphics[width=\linewidth]{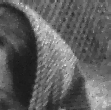}
		\end{subfigure}%
		\begin{subfigure}{\figw}
			\includegraphics[width=\linewidth]{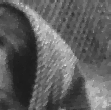}
		\end{subfigure}%
		\begin{subfigure}{\figw}
			\includegraphics[width=\linewidth]{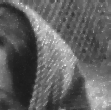}
		\end{subfigure}%
		\begin{subfigure}{\figw}
			\includegraphics[width=\linewidth]{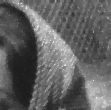}
		\end{subfigure}%
		\begin{subfigure}{\figw}
			\includegraphics[width=\linewidth]{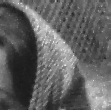}
		\end{subfigure}%
		\begin{subfigure}{\figw}
			\includegraphics[width=\linewidth]{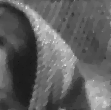}
		\end{subfigure}%
		\begin{subfigure}{\figw}
			\includegraphics[width=\linewidth]{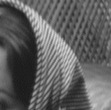}
		\end{subfigure}\\%
		\begin{subfigure}{\figw}
			\includegraphics[width=\linewidth]{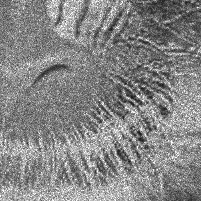}
		\end{subfigure}%
		\begin{subfigure}{\figw}
			\includegraphics[width=\linewidth]{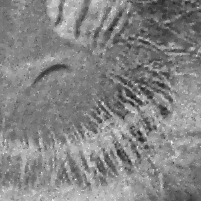}
		\end{subfigure}%
		\begin{subfigure}{\figw}
			\includegraphics[width=\linewidth]{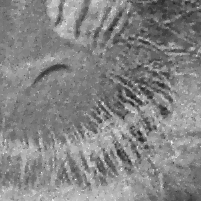}
		\end{subfigure}%
		\begin{subfigure}{\figw}
			\includegraphics[width=\linewidth, cframe = black 1.5pt]{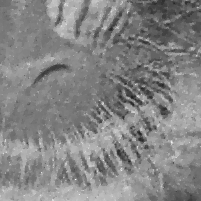}
		\end{subfigure}\ %
		\begin{subfigure}{\figw}
			\includegraphics[width=\linewidth]{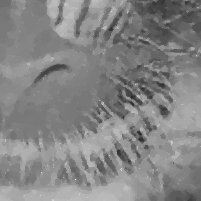}
		\end{subfigure}%
		\begin{subfigure}{\figw}
			\includegraphics[width=\linewidth]{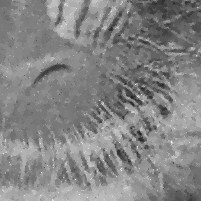}
		\end{subfigure}%
		\begin{subfigure}{\figw}
			\includegraphics[width=\linewidth]{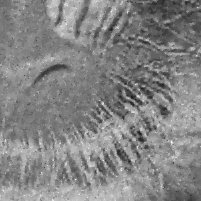}
		\end{subfigure}\\%
		\begin{subfigure}{\figw}
			\includegraphics[width=\linewidth]{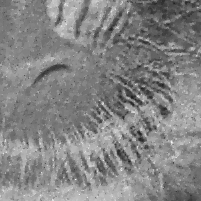}
		\end{subfigure}%
		\begin{subfigure}{\figw}
			\includegraphics[width=\linewidth]{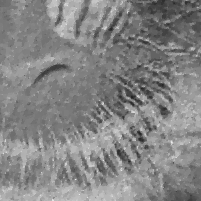}
		\end{subfigure}%
		\begin{subfigure}{\figw}
			\includegraphics[width=\linewidth]{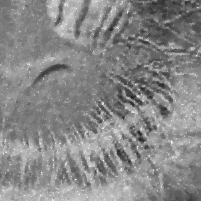}
		\end{subfigure}%
		\begin{subfigure}{\figw}
			\includegraphics[width=\linewidth]{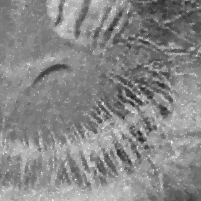}
		\end{subfigure}%
		\begin{subfigure}{\figw}
			\includegraphics[width=\linewidth]{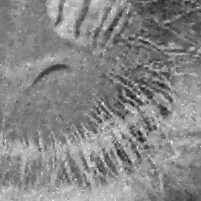}
		\end{subfigure}%
		\begin{subfigure}{\figw}
			\includegraphics[width=\linewidth]{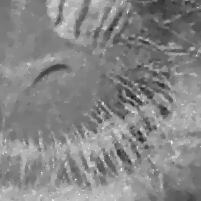}
		\end{subfigure}
		\begin{subfigure}{\figw}
			\includegraphics[width=\linewidth]{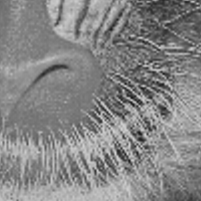}
		\end{subfigure}\\%
		\begin{subfigure}{\figw}
			\includegraphics[width=\linewidth]{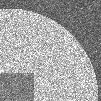}
		\end{subfigure}%
		\begin{subfigure}{\figw}
			\includegraphics[width=\linewidth]{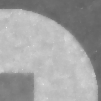}
		\end{subfigure}%
		\begin{subfigure}{\figw}
			\includegraphics[width=\linewidth]{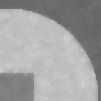}
		\end{subfigure}%
		\begin{subfigure}{\figw}
			\includegraphics[width=\linewidth]{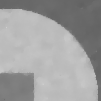}
		\end{subfigure}%
		\begin{subfigure}{\figw}
			\includegraphics[width=\linewidth]{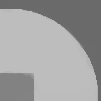}
		\end{subfigure}%
		\begin{subfigure}{\figw}
			\includegraphics[width=\linewidth,cframe=black 1.5pt]{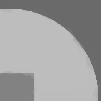}
		\end{subfigure}\ %
		\begin{subfigure}{\figw}
			\includegraphics[width=\linewidth]{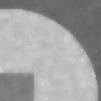}
		\end{subfigure}\\%
		\begin{subfigure}{\figw}
			\includegraphics[width=\linewidth]{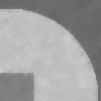}
		\end{subfigure}%
		\begin{subfigure}{\figw}
			\includegraphics[width=\linewidth]{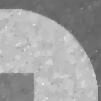}
		\end{subfigure}%
		\begin{subfigure}{\figw}
			\includegraphics[width=\linewidth]{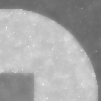}
		\end{subfigure}%
		\begin{subfigure}{\figw}
			\includegraphics[width=\linewidth]{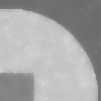}
		\end{subfigure}%
		\begin{subfigure}{\figw}
			\includegraphics[width=\linewidth]{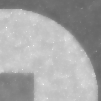}
		\end{subfigure}%
		\begin{subfigure}{\figw}
			\includegraphics[width=\linewidth]{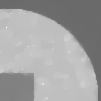}
		\end{subfigure}%
		\begin{subfigure}{\figw}
			\includegraphics[width=\linewidth]{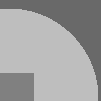}
		\end{subfigure}%
		\caption{Detailed MHDM denoising. From left to right, top to bottom: noisy images; SO MHDM (EL) regular, tight, and refined versions; SO MHDM (ADMM) regular and tight versions; AA-log MHDM regular, tight, and refined versions; TNV-log regular and tight versions; TNV, DZ and original images. Black borders indicate the best recoveries (highest SNR) for each image. } \label{fig:comparisons}
	\end{figure}%
	\subsection{Deblurring-Denoising}
	Table  \ref{tab:SNR_blur} lists comparisons for denoising-deblurring. The AA-log MHDM models perform better overall. The DZ model greatly over-smooths details compared to the multiscale schemes, especially in the detailed crops of Fig.\,\ref{fig:comparisons_blur}, and naturally, it performs well on ``Geometry". 
	We emphasize the improvement of the TNV and the TNV-log methods over the DZ model on textured images in either SNR or the visual metric, demonstrating the effectiveness of multiscale recoveries. 
	\begin{table}[h]
		\centering
		\caption{SNR values for restoring noisy-blurred images at the $k_{min}$. Bold entries are the maximums of their respective columns. }
		\begin{tabular}{c|cccc}
			SNR at $k_{min}$ & Cameraman & Barbara  & Mandril & Geometry \\\hline
			AA MHDM & 18.81 & 17.69 &18.65 &26.45\\
			AA-log MHDM & 18.81 & 17.69  & 18.64 & 26.66 \\
			AA-log Tight & \textbf{19.07} & \textbf{17.74}  & \textbf{18.66} & \textbf{28.76} \\
			AA-log Refined & 18.81 & 17.54  & 18.47 & 27.12 \\
			TNV-log & 18.22  & 17.30 & 18.13 & 24.17 \\
			TNV-log Tight & 18.39 & 17.35 & 18.16 & 25.17 \\\hline
			TNV   & 18.26 & 17.32  & 18.17 & 24.23 \\
			DZ   & 17.14& 17.36 & 17.24 & 28.32 
		\end{tabular}%
		\label{tab:SNR_blur}%
	\end{table}%
	%
	\setlength{\figw}{0.15\textwidth} 
	\begin{figure}[h!]
		\centering
		\begin{subfigure}{\figw}
			\includegraphics[width=\linewidth]{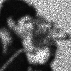}
		\end{subfigure}%
		\begin{subfigure}{\figw}
			\includegraphics[width=\linewidth]{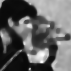}
		\end{subfigure}%
		\begin{subfigure}{\figw}
			\includegraphics[width=\linewidth, cframe = black 1.5pt]{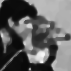}
		\end{subfigure}\ %
		\begin{subfigure}{\figw}
			\includegraphics[width=\linewidth]{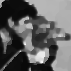}
		\end{subfigure}%
		\begin{subfigure}{\figw}
			\includegraphics[width=\linewidth]{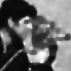}
		\end{subfigure}\\%
		\hspace*{\figw}%
		\begin{subfigure}{\figw}
			\includegraphics[width=\linewidth]{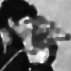}
		\end{subfigure}%
		\begin{subfigure}{\figw}
			\includegraphics[width=\linewidth]{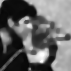}
		\end{subfigure}%
		\begin{subfigure}{\figw}
			\includegraphics[width=\linewidth]{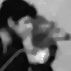}
		\end{subfigure}%
		\begin{subfigure}{\figw}
			\includegraphics[width=\linewidth]{./additive/cameraman_noise_tight/cameraman_zoom.png}
		\end{subfigure}\\%
		\begin{subfigure}{\figw}
			\includegraphics[width=\linewidth]{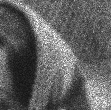}
		\end{subfigure}%
		\begin{subfigure}{\figw}
			\includegraphics[width=\linewidth]{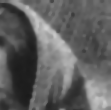}
		\end{subfigure}%
		\begin{subfigure}{\figw}
			\includegraphics[width=\linewidth, cframe = black 1.5pt]{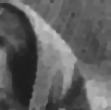}
		\end{subfigure}\ %
		\begin{subfigure}{\figw}
			\includegraphics[width=\linewidth]{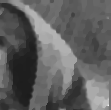}
		\end{subfigure}%
		\begin{subfigure}{\figw}
			\includegraphics[width=\linewidth]{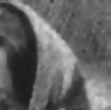}
		\end{subfigure}\\%
		\hspace*{\figw}%
		\begin{subfigure}{\figw}
			\includegraphics[width=\linewidth]{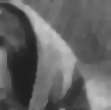}
		\end{subfigure}%
		\begin{subfigure}{\figw}
			\includegraphics[width=\linewidth]{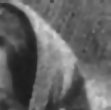}
		\end{subfigure}%
		\begin{subfigure}{\figw}
			\includegraphics[width=\linewidth]{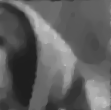}
		\end{subfigure}%
		\begin{subfigure}{\figw}
			\includegraphics[width=\linewidth]{./additive/barbara_noise_tight/barbara_zoom.png}
		\end{subfigure}\\%
		\begin{subfigure}{\figw}
			\includegraphics[width=\linewidth]{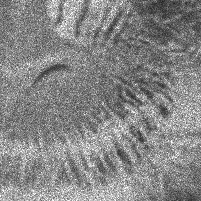}
		\end{subfigure}%
		\begin{subfigure}{\figw}
			\includegraphics[width=\linewidth]{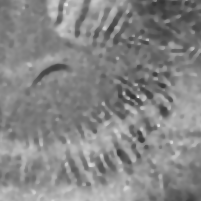}
		\end{subfigure}%
		\begin{subfigure}{\figw}
			\includegraphics[width=\linewidth, cframe = black 1.5pt]{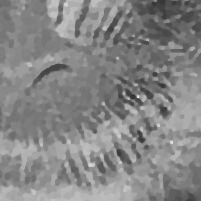}
		\end{subfigure}\ %
		\begin{subfigure}{\figw}
			\includegraphics[width=\linewidth]{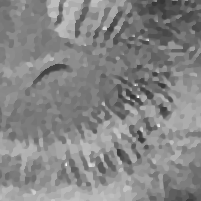}
		\end{subfigure}%
		\begin{subfigure}{\figw}
			\includegraphics[width=\linewidth]{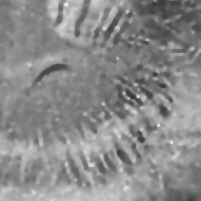}
		\end{subfigure}\\%
		\hspace*{\figw}%
		\begin{subfigure}{\figw}
			\includegraphics[width=\linewidth]{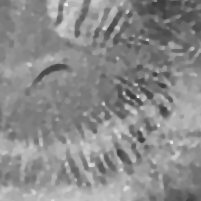}
		\end{subfigure}%
		\begin{subfigure}{\figw}
			\includegraphics[width=\linewidth]{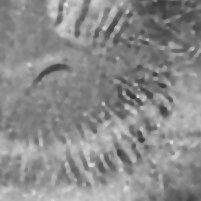}
		\end{subfigure}%
		\begin{subfigure}{\figw}
			\includegraphics[width=\linewidth]{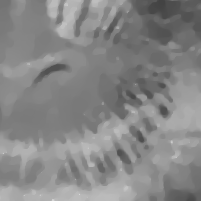}
		\end{subfigure}%
		\begin{subfigure}{\figw}
			\includegraphics[width=\linewidth]{./additive/mandril_noise_tight/mandril_zoom.png}
		\end{subfigure}\\%
		\begin{subfigure}{\figw}
			\includegraphics[width=\linewidth]{./additive/disc_square_noise_tight/disc_square_noisy_zoom.png}
		\end{subfigure}%
		\begin{subfigure}{\figw}
			\includegraphics[width=\linewidth]{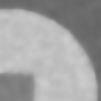}
		\end{subfigure}%
		\begin{subfigure}{\figw}
			\includegraphics[width=\linewidth, cframe = black 1.5pt]{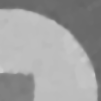}
		\end{subfigure}\ %
		\begin{subfigure}{\figw}
			\includegraphics[width=\linewidth]{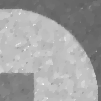}
		\end{subfigure}%
		\begin{subfigure}{\figw}
			\includegraphics[width=\linewidth]{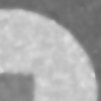}
		\end{subfigure}\\%
		\hspace*{\figw}%
		\begin{subfigure}{\figw}
			\includegraphics[width=\linewidth]{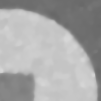}
		\end{subfigure}%
		\begin{subfigure}{\figw}
			\includegraphics[width=\linewidth]{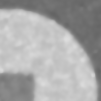}
		\end{subfigure}%
		\begin{subfigure}{\figw}
			\includegraphics[width=\linewidth]{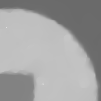}
		\end{subfigure}%
		\begin{subfigure}{\figw}
			\includegraphics[width=\linewidth]{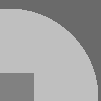}
		\end{subfigure}%
		\caption{Detailed MHDM denoising-deblurring. From left to right, top to bottom: noisy images; AA-log MHDM regular, tight, and refined versions; TNV-log regular and tight versions; TNV, DZ, and original images. Black borders indicate the best recoveries (highest SNR) for each image.}\label{fig:comparisons_blur}
	\end{figure}%
	\section{Conclusion}
	We introduce several multiscale hierarchical decomposition methods for images degraded by multiplicative noise which are able to retain texture and image features at different scales while reducing noise. We demonstrate that the fidelity terms decrease monotonically with increasing hierarchical depth, and propose an effective stopping criterion which limits restoring excess noise. Additionally, we consider extensions of the multiplicative MHDM (the so-called tight and refined versions), which are shown theoretically to push the regularity of recoveries to match that of the original image and empirically demonstrate better convergence properties of the iterates. The AA MHDM and AA-log MHDM methods are aimed specifically at gamma noise, and additionally handle deblurring tasks quite well, outperforming the existing DZ and TNV models in our tests. The convex SO MHDM models are quite robust with respect to initialization. They behave exceptionally well for images with smooth regions and still prevent over-smoothing on regions with oscillating patterns when implemented with ADMM, while the Euler-Lagrange method retains slightly more details in very textured images.
	Finally, we consider the TNV-log method which handles blurring and noise without assuming a specific distribution for the noise (e.g.\,gamma). Accordingly, it is suitable for images corrupted by more general multiplicative noises. While TNV-log does not top the other MHDM methods when tested on gamma noise corrupted samples, it still outperforms the DZ model when deblurring and denoising more textured images, while continuing to maintain fine-scaled features. Collectively, these MHDM schemes provide a means to address multiplicative noise-degraded images by constructing decompositions across several scales. It is hoped that the schemes and the included analysis might be helpful for the reader to identify applications  beyond the denoising and deblurring tasks investigated here, and to extend to additional schemes outside those studied within. In the future, we aim to extend the proposed models to image segmentation and to vector-valued data, for instance, to restore colored images perturbed by multiplicative noise and blurring.

	\section{Acknowledgments}
	J.B.~, W.L.~and L.V.~received support from NSF grant DMS 2012868 while working on this project. E.R. is supported by the
	Austrian Science Fund (FWF): DOC 78. W.L.~completed part of the research while a member of the Department of Mathematics at UCLA
	 and is supported by the Faculty Research Grants at Fordham University. The authors thank Tobias Wolf (University of Klagenfurt) for the manuscript proofreading.
	\clearpage
	\bibliographystyle{siamplain}
	\bibliography{refs}

\begin{thebibliography}{10}

\bibitem{amb_dal}
{\sc L.~Ambrosio and G.~Dal~Maso}, {\em A general chain rule for distributional
  derivatives}, Proceedings of the American Mathematical Society, 108 (1990),
  pp.~691--702.

\bibitem{auber_aujol_2008}
{\sc G.~Aubert and J.-F. Aujol}, {\em A variational approach to removing
  multiplicative noise}, SIAM Journal on Applied Mathematics, 68 (2008),
  pp.~925--946.

\bibitem{bavirisetti2018multi}
{\sc D.~P. Bavirisetti and R.~Dhuli}, {\em Multifocus image fusion using
  multiscale image decomposition and saliency detection}, Ain Shams Engineering
  Journal, 9 (2018), pp.~1103--1117.

\bibitem{bijaoui1997multiscale}
{\sc A.~Bijaoui, Y.~Bobichon, Y.~Fang, and F.~Rue}, {\em Multiscale methods
  applied to the analysis of synthetic aperture radar images}, Traitement du
  Signal, 14 (1997), pp.~179--194.

\bibitem{burckhardt1978}
{\sc C.~B. Burckhardt}, {\em Speckle in ultrasound b-mode scans}, IEEE
  Transactions on Sonics and ultrasonics, 25 (1978), pp.~1--6.

\bibitem{chambolle_multiplicative}
{\sc A.~Chambolle}, {\em Total variation minimization for image reconstruction:
  The multiplicative noise case}, preprint,  (2023).

\bibitem{Chambolle2009OnTV}
{\sc A.~Chambolle and J.~Darbon}, {\em On total variation minimization and
  surface evolution using parametric maximum flows}, International Journal of
  Computer Vision, 84 (2009), pp.~288--307.

\bibitem{cha_lio95}
{\sc A.~Chambolle and P.-L. Lions}, {\em Image restoration by constrained total
  variation minimization and variants}, vol.~2567, 1995.
\newblock SPIE Electronic Imaging Proceedings.

\bibitem{cui2015detail}
{\sc G.~Cui, H.~Feng, Z.~Xu, Q.~Li, and Y.~Chen}, {\em Detail preserved fusion
  of visible and infrared images using regional saliency extraction and
  multi-scale image decomposition}, Optics Communications, 341 (2015),
  pp.~199--209.

\bibitem{learned_models}
{\sc S.~Cuomo, M.~De~Rosa, S.~Izzo, F.~Piccialli, and M.~Pragliola}, {\em
  Speckle noise removal via learned variational models}, Applied Numerical
  Mathematics,  (2023).

\bibitem{dar_men_res22}
{\sc J.~Darbon, T.~Meng, and E.~Resmerita}, {\em On hamilton?jacobi pdes and
  image denoising models with certain nonadditive noise}, J Math Imaging Vis,
  64 (2022), pp.~408--441.

\bibitem{deb_guy_ves}
{\sc N.~Debroux, C.~Le~Guyader, and L.~A. Vese}, {\em A multiscale deformation
  representation}, SIAM Journal on Imaging Sciences, 16 (2023), pp.~802--841.

\bibitem{dong_2013}
{\sc Y.~Dong and T.~Zeng}, {\em A convex variational model for restoring
  blurred images with multiplicative noise}, SIAM Journal on Imaging Sciences,
  6 (2013), pp.~1598--1625.

\bibitem{goodman1976}
{\sc J.~W. Goodman}, {\em Some fundamental properties of speckle}, JOSA, 66
  (1976), pp.~1145--1150.

\bibitem{huang_2009}
{\sc Y.~Huang, M.~Ng, and Y.-W. Wen}, {\em A new total variation method for
  multiplicative noise removal}, SIAM J. Imaging Sci., 2 (2009), pp.~20--40.

\bibitem{TGV_multiplicative}
{\sc Z.~Jin, J.~Wang, L.~Min, and M.~Zheng}, {\em An adaptive total generalized
  variational model for speckle reduction in ultrasound images}, J. Franklin
  Inst., 359 (2022), pp.~8377--8394.

\bibitem{jin_yang_2010}
{\sc Z.~Jin and X.~Yang}, {\em Analysis of a new variational model for
  multiplicative noise removal}, Journal of Mathematical Analysis and
  Applications, 362 (2010), pp.~415--426.

\bibitem{mod_nac_ron}
{\sc L.~R. K.~Modin, A.~Nachman}, {\em A multiscale theory for image
  registration and nonlinear inverse problems}, Advances in Mathematics, 346
  (2019), pp.~1009--1066.

\bibitem{multiscale_theory}
{\sc S.~Kindermann, E.~Resmerita, and T.~Wolf}, {\em Multiscale hierarchical
  decomposition methods for ill-posed problems}, arXiv:2304.08332,  (2023).

\bibitem{li_res_ves}
{\sc W.~Li, E.~Resmerita, and L.~Vese}, {\em Multiscale hierarchical image
  decomposition and refinements: Qualitative and quantitative results}, SIAM J.
  Imaging Sci., 14 (2021), pp.~844--877.

\bibitem{liu2016modified}
{\sc M.~Liu and Q.~Fan}, {\em A modified convex variational model for
  multiplicative noise removal}, Journal of Visual Communication and Image
  Representation, 36 (2016), pp.~187--198.

\bibitem{nao2022speckle}
{\sc S.~Nao and Y.~Wang}, {\em Speckle noise removal model based on diffusion
  equation and convolutional neural network}, Computational Intelligence and
  Neuroscience,  (2022).

\bibitem{multi_registration}
{\sc D.~Paquin, D.~Levy, E.~Schreibmann, and L.~Xing}, {\em Multiscale image
  registration}, Math. Biosci. Eng., 3 (2006), pp.~389--418.

\bibitem{rlo}
{\sc L.~Rudin, P.-L. Lions, and S.~Osher}, {\em Multiplicative denoising and
  deblurring: Theory and algorithms}, in Geometric Level Set Methods in Imaging
  Vision and Graphics, New York, 2003, Springer, pp.~201--213.

\bibitem{rud_osh94}
{\sc L.~Rudin and S.~Osher}, {\em Total variation based image restoration with
  free local constraints}, vol.~I, 1994, pp.~31--35.
\newblock Proc. IEEE ICIP, Austin (Texas) USA.

\bibitem{rue1996multiscale}
{\sc F.~Ru{\'e} and A.~Bijaoui}, {\em A multiscale vision model applied to
  astronomical images}, Vistas in Astronomy, 40 (1996), pp.~495--502.

\bibitem{osher_shi_multiplic}
{\sc J.~Shi and S.~Osher}, {\em A nonlinear inverse scale space method for a
  convex multiplicative noise model}, SIAM Journal on Imaging Sciences, 1
  (2008), pp.~294--321.

\bibitem{steidl_teuber}
{\sc G.~Steidl and T.~Teuber}, {\em Removing multiplicative noise by
  {D}ouglas-{R}achford splitting methods}, J. Math. Imaging Vision, 36 (2010),
  pp.~168--184.

\bibitem{tadmor2009multiscale}
{\sc E.~Tadmor and P.~Athavale}, {\em Multiscale image representation using
  novel integro-differential equations}, Inverse Problems \& Imaging, 3 (2009),
  p.~693.

\bibitem{tad_nez_ves1}
{\sc E.~Tadmor, S.~Nezzar, and L.~Vese}, {\em A multiscale image representation
  using hierarchical ({BV}, {$L^2$}) decompositions}, Multiscale Modeling \&
  Simulation, 2 (2004), pp.~554--579.

\bibitem{tad_nez_ves2}
{\sc E.~Tadmor, S.~Nezzar, and L.~Vese}, {\em Multiscale hierarchical
  decomposition of images with applications to deblurring, denoising, and
  segmentation}, Commun. Math. Sci., 6 (2008), pp.~281--307.

\bibitem{tang_he}
{\sc L.~Tang and C.~He}, {\em Multiscale variational decomposition and its
  application for image hierarchical restoration}, Computers \& Electrical
  Engineering, 54 (2016), pp.~354--369.

\bibitem{ullah2017new}
{\sc A.~Ullah, W.~Chen, M.~A. Khan, and H.~Sun}, {\em A new variational
  approach for multiplicative noise and blur removal}, PloS ONE, 12 (2017),
  p.~e0161787.

\bibitem{vese2016variational}
{\sc L.~A. Vese and C.~Le~Guyader}, {\em Variational methods in image
  processing}, CRC Press Boca Raton, FL, 2016.

\bibitem{vol67}
{\sc A.~I. Vol'pert}, {\em The spaces bv and quasilinear equations},
  Matematicheskii Sbornik, 115 (1967), pp.~255--302.

\bibitem{wang2021color}
{\sc W.~Wang, M.~Yao, and M.~K. Ng}, {\em Color image multiplicative noise and
  blur removal by saturation-value total variation}, Applied Mathematical
  Modelling, 90 (2021), pp.~240--264.

\bibitem{wu2020convex}
{\sc T.~Wu, W.~Li, L.~Li, and T.~Zeng}, {\em A convex variational approach for
  image deblurring with multiplicative structured noise}, IEEE Access, 8
  (2020), pp.~37790--37807.

\bibitem{xu2014multiscale}
{\sc R.~Xu, P.~Athavale, A.~Nachman, and G.~A. Wright}, {\em Multiscale
  registration of real-time and prior mri data for image-guided cardiac
  interventions}, IEEE Transactions on Biomedical Engineering, 61 (2014),
  pp.~2621--2632.

\bibitem{zhang2022image}
{\sc Y.~Zhang, S.~Li, Z.~Guo, B.~Wu, and S.~Du}, {\em Image multiplicative
  denoising using adaptive euler's elastica as the regularization}, Journal of
  Scientific Computing, 90 (2022), pp.~1--34.

\end{thebibliography}
	
\end{document}